\documentclass[12pt,a4paper]{amsart}
\usepackage{mystyle,mymacros}
\usepackage[all]{xy} 		
\usepackage{longtable}
\usepackage{tikz}
\usepackage{amsxtra,amsmath,amsfonts,amscd, amssymb, mathrsfs}

\theoremstyle{definition}

\theoremstyle{theorem}

\renewcommand*\val{v}
\newcommand*\multone[1][X]{\Trop(#1)_{m_{\Trop}=1}}

\newcommand*\barNRD{\bar N{}_\R^\Delta}
\newcommand*\barNGD{\bar N{}_\Gamma^\Delta}
\newcommand*\barNGpD{\bar N{}_{\Gamma'}^\Delta}
\newcommand*\barNRs{\bar N{}_\R^\sigma}
\newcommand*\STrop{S_{\Trop}}

\newcommand*\ccirc{{\circ\circ}}
\mathtextb{can}
\mathtext{LC}
\mathtext{HK}

\makeatletter
\@namedef{subjclassname@2010}{%
 \textup{2010} Mathematics Subject Classification}
\makeatother

\title{Tropical Skeletons}

\author{Walter Gubler}
\address{Walter Gubler, Fakult\"at f\"ur Mathematik,  Universit\"at Regensburg,
Universit\"atsstra{\ss}e 31, D-93040 Regensburg}
\email{walter.gubler@mathematik.uni-regensburg.de}

\author{Joseph Rabinoff}
\address{Joseph Rabinoff, School of Mathematics, Georgia Institute of Technology, Atlanta GA 30332-0160, USA}
\email{jrabinoff@math.gatech.edu}

\author{Annette Werner}
\address{Annette Werner, Institut f\"ur Mathematik, Goethe-Universit\"at
  Frankfurt, Robert-Mayer-Stra{\ss}e 8, D-60325 Frankfurt a.M.}
\email{werner@math.uni-frankfurt.de}

\subjclass[2010]{14G22, 14T05}

\begin{document}

\begin{abstract}
  In this paper, we study the interplay between tropical and
  analytic geometry for closed subschemes of toric varieties. Let $K$ be a
  complete non-Archimedean field, and let $X$ be a closed subscheme of a toric
  variety over $K$. We define the tropical skeleton of $X$ as the subset of the
  associated Berkovich space $X^{\an}$ which collects all Shilov boundary points
  in the fibers of the Kajiwara--Payne tropicalization map. We develop
  polyhedral criteria for limit points to belong to the tropical skeleton, and
  for the tropical skeleton to be closed.
  We apply the limit point criteria to the question of continuity of the
  canonical section of the tropicalization map on the multiplicity-one
  locus.  
  This map is known to be continuous on all torus orbits; we prove criteria for
  continuity when crossing torus orbits.
  When $X$ is sch\"on and defined over a discretely valued field, we show that the tropical skeleton coincides with a
  skeleton of a strictly semistable pair, and is naturally isomorphic to
  the parameterizing complex of Helm--Katz.

\end{abstract}
\maketitle

\section{Introduction}
Let $K$ be a field which is complete with respect to a non-Archimedean absolute value, which might be trivial.
Tropicalizing a scheme  $X$ of finite type over $K$ means, roughly speaking, applying the valuation map to a set of coordinates on $X$. This produces  a combinatorial shadow of $X$ called the tropical variety of $X$. Such coordinates are obtained by embedding $X$ (or an open subscheme) into a torus or, more generally, into a toric variety. The tropicalization map extends uniquely to a proper continuous map from the Berkovich space $X^{\an}$ to a Kajiwara--Payne compactification of Euclidean space.  In fact, by a result of Payne \cite{payne09:analytif_tropical} {and the generalizations given by Foster, Gross and Payne \cite{foster_gross_payne14:limits_trop}}, for any subscheme $X$ of a toric variety, the topological space underlying $X^{\an}$ is the inverse limit over all tropicalizations of $X$ with respect to suitable choices of coordinates.

An interesting question is the relationship between the Berkovich space
$X^{\an}$ and an individual tropicalization. 
If $X$ is a curve, the problem of
finding subgraphs of $X^{\an}$ which are isometric to tropical varieties of $X$
was investigated by Baker, Payne and Rabinoff in
\cite{baker_payne_rabinoff16:tropical_curves,baker_payne_rabinoff13:analytic_curves}. 

In  \cite{gubler_rabinoff_werner16:skeleton_tropical},
we generalize several of these results to the higher-dimensional setting. 
Among other results, it is shown in
\textit{loc.\ cit.}\ that the tropicalization map for a subvariety of a torus has a canonical continuous section on the locus of all points with tropical multiplicity one. This section is defined by associating to every point $\omega$ of tropical multiplicity one the unique Shilov boundary point in the fiber of tropicalization over $\omega$. For an overview of these results, see \cite{werner16:analytif_tropical}. 

 In the present paper we consider higher dimensional subvarieties of toric varieties. By the previous results we can define the locus of multiplicity one and the section map from the tropicalization to the Berkovich analytification of the variety stratum by stratum. Quite surprisingly, it turns out that the section map is in general no longer continuous when passing from one stratum to another. We provide an example where continuity fails in  \ref{eg:counter.eg.3}.  A delicate investigation of polyhedra in the tropicalization is necessary to obtain criteria for continuity. We approach this problem from a more general angle by investigating a subset of the Berkovich space $X^{\an}$ which we call the tropical skeleton.

Before we will describe this notion and before we explain more general continuity criteria, we formulate our main application which is continuity of the section map in the case of a proper intersection with torus orbits.

Let $Y_\Delta$ be the toric variety {over $K$} associated to a pointed rational fan $\Delta$. Then $Y_\Delta$ can be stratified into torus orbits $O(\sigma)$, where $\sigma$ runs over the cones in $\Delta$.  The Kajiwara-Payne tropicalization of $Y_\Delta$ is a natural partial compactification $\barNRD$ of the real cocharacter space $N_\R$ of the dense torus 
{$T$ in $Y_\Delta$.}
As a set (but not as a topological space) it is the disjoint union of all cocharacter spaces $N_\R(\sigma) \coloneq N_\R / \langle \sigma \rangle$ associated to the torus orbits $O(\sigma)$ in $Y_\Delta$. There is a natural continuous tropicalization map $Y^\an \rightarrow\barNRD$. We consider a closed subscheme $X$ of $Y$ and its tropicalization $\Trop(X)$, which is defined as the image of $X^{\an}$ under the
 tropicalization map 
 \[\trop\colon~ X^{\an} \hookrightarrow Y_\Delta^{\an} \To\barNRD.\]
The intersection of $X$ with each toric stratum $O(\sigma)$ of $\barNRD$ is  a
closed subscheme of a torus, so that its tropical variety, which is simply
$\Trop(X) \cap N_\R(\sigma)$, can be equipped with the structure of an integral
affine polyhedral complex. 
{Hence there is a notion of tropical multiplicity $m_{\Trop}(\omega)$ for any $\omega \in \Trop(X)\cap N_\R(\sigma)$ (see  \S\ref{sec:initial.degen}).
Using our previous work \cite{gubler_rabinoff_werner16:skeleton_tropical} orbitwise,  
this defines a canonical section $s_X$ of the tropicalization map on the subset
$\Trop(X)_{m_{\Trop}=1}$ of all points of tropical multiplicity one in
$\Trop(X)$. }

\newtheorem*{proper.int.continuity}{Theorem~\ref{cor:proper.int.continuity}}
\begin{proper.int.continuity}
 Let $X$ be a closed subscheme of $Y_\Delta$ such that $X \cap T$ is {equidimensional and dense in $X$.}  Assume additionally that  for all $\sigma\in\Delta$, either
   $X\cap O(\sigma)$ is empty or of dimension $\dim(X)-\dim(\sigma)$.  
  Then
  $s_X\colon\multone\to X^\an$ is continuous.
  
 \end{proper.int.continuity}

Under the hypotheses of Theorem \ref{cor:proper.int.continuity}, the map $s_X: \Trop(X)_{m_{\Trop}=1} \to X^{\rm an}$ induces a homeomorphism onto its image and we may realize the locus $\Trop(X)_{m_{\Trop}=1}$ as a closed subset of $\trop^{-1}(\Trop(X)_{m_{\Trop}=1})$ in $X^{\rm an}$. Theorem \ref{cor:proper.int.continuity} follows from Theorem \ref{thm:continuity} mentioned below in the introduction which yields a completely combinatorial criterion for continuity of $s_X$.

In the higher-dimensional example of the Grassmannian of planes it was shown in
\cite{cueto_habich_werner13:grassmannians} that the tropical Grassmannian is
homeomorphic to a closed subset of the analytic Grassmannian. This result relies
heavily on combinatorial arguments using the interpretation of the tropical
Grassmannian of planes as a space of phylogenetic trees. Draisma and Postinghel~\cite{draisma_postingh16:faithful_tropical} provide an alternative proof using tropical torus
actions.

Another interesting case from the point of view of moduli spaces is discussed in Theorem 3.14 of~\cite{cavalier_hampe_etal16:moduli_spaces_rational_weighted_stable}, where it is shown that the tropicalization of a suitable Hassett space can be identified with its Berkovich skeleton. 

We will now explain the other results in this paper and how they lead to Theorem \ref{cor:proper.int.continuity}.
In ~\S\S\ref{sec:analytic spaces}--\ref{sec:toric and tropical} we provide some background material on tropical and analytic geometry. Working with several torus orbits at once forces us to consider reducible subschemes of tori. We work out some fundamental properties in this situation which we did not find in the literature. 
Note that our ground field may be an arbitrary field endowed with the trivial absolute value. This often requires a careful study of  the behaviour of our objects under non-Archimedean field extensions.

Let  $X$ be any closed subscheme of the toric variety $Y_\Delta$ over a  $K$ and let $O(\sigma)$ be the orbit associated to the cone $\sigma$ in the fan $\Delta$. For every $\omega \in \Trop(X)$ the fiber
$\trop^{-1}(\{\omega\})$ of the tropicalization map over $\omega$ is an affinoid
domain in $(X \cap O(\sigma))^{\an}$. Therefore it contains a finite subset
of points $B$, the Shilov boundary, such that every element in the associated
affinoid algebra achieves its maximum absolute value on $B$. Then we define the tropical skeleton of $X$ in~\S\ref{sec:tropical.skeleton} as the subset 
\[\STrop(X) = \big\{ \xi \in X^{\an} \mid \xi \mbox{ is a Shilov boundary point of } \trop^{-1}(\trop(\xi))\big\}\]
of $X^{\an}$. 
The tropical skeleton does not change by passing to the induced reduced
subscheme and it is compatible with valued field extensions. 
Moreover, we show that it is locally closed in $X^{\an}$, and that the tropical skeleton of $X$ is the union of the tropical skeletons of the irreducible components. 
In Example~\ref{eg:counter.eg} we discuss a concrete hypersurface in affine $3$-space such that its tropical skeleton is not closed. Motivated by this example, we define the local dimension $d(\omega)$ of a point $\omega \in \Trop(X)$ as the dimension of the local cone at $\omega$ of the tropicalization of $X$. This coincides with the dimension of the initial degeneration of $X$ at $\omega$. 
Then a  polyhedron in the tropicalization of $X \cap O(\sigma)$ containing $\omega$ is called $d$-maximal at $\omega$ if its dimension coincides with  $d(\omega)$.  
Note that the local cone of $\Trop(X)$ at $\omega$ can be identified with the tropicalization of the initial degeneration at $\omega$ over the residue field of $K$ endowed with the trivial absolute value.

In Theorem \ref{thm:limit.points}, we prove a very general criterion for a limit
point of a sequence $\xi_i \in \STrop(X)$ to stay inside the tropical
skeleton. For simplicity, we assume here that the sequence $\xi_i$ is in the dense torus $T$ as well, and that $X \cap T$ is of pure dimension $d$. 
As a consequence of Theorem~\ref{thm:limit.points}, we obtain the following result:

\newtheorem*{skeleton_equidimensional}{Theorem \ref{thm:skeleton_equidimensional}}
\begin{skeleton_equidimensional}
Let $X$ be a closed subscheme of the toric variety $Y_\Delta$ such that $X \cap O(\sigma)$ 
is equidimensional of dimension $d_\sigma$ for any $\sigma \in \Delta$.
We suppose that for all faces $\tau \prec \sigma$ of $\Delta$ and any $d_\tau$-dimensional polyhedron $P$ in $\Trop(X) \cap N_\R(\tau)$ such that its closure meets $N_\R(\sigma)$, the natural projection 
of $P$ to $N_\R(\sigma)$ has dimension $d_\sigma$. Then $\STrop(X)$ is closed.
\end{skeleton_equidimensional}

The situation is particularly nice if $X$ meets every torus orbit not at all or properly, which means that either
$X \cap O(\sigma) = \emptyset$  or that its  dimension is  equal to $\dim(X) - \dim (\sigma)$. We investigate this situation in~\S\ref{sec:proper intersection}. In particular, we show in Corollary~\ref{cor:proper.int.ST.closed} that for such $X$ the tropical skeleton $\STrop(X)$ is closed in $X^{\an}$. 

Section~\ref{sec:section.trop} deals with continuous sections of the
tropicalization map. Let $X$ be a closed subscheme of $Y_\Delta$ and consider a
point $\omega$ in $\Trop(X \cap O(\sigma))$ of tropical multiplicity one. We
show in Proposition~\ref{prop:shilov.section} that in this case there exists a
unique irreducible component $C$ of $X \cap O(\sigma)$ of (maximal possible)
dimension $d(\omega)$ such that $\omega$ lies in $\Trop(C)$ and such that
$\trop^{-1} (\omega) \cap C^{\an}$ has a unique Shilov boundary point. Hence for
every point $\omega$  of tropical multiplicity one we can single out a Shilov
boundary point in the fiber of the tropicalization map over $\omega$. This
defines a section $s_X$ of the tropicalization map on the subset
$\Trop(X)_{m_{\Trop}=1}$ of all points of tropical multiplicity one of
$\Trop(X)$. The question of continuity of $s_X$ is closely related to the
question of the tropical skeleton being closed.  We will deduce the following
theorem from the results of~\S\S\ref{sec:limit.points}--\ref{sec:proper intersection}:

\newtheorem*{continuity}{Theorem \ref{thm:continuity}}
\begin{continuity}
  Let $\Delta$ be a pointed rational fan in $N_\R$ and let $X\subset Y_\Delta$
  be a closed subscheme.  Let $\{\omega_i\}_{i=1}^\infty$ be a sequence in
  $\multone\cap N_\R$ converging to a point $\omega\in\multone \cap N_\R(\sigma)$
  for $\sigma\in\Delta$, $\sigma\neq\{0\}$.  Suppose that there exists a
  polyhedron $P\subset\Trop(X)\cap N_\R$ which is $d$-maximal at each
  $\omega_i$.  If the natural projection 
of $P$ to $N_\R(\sigma)$ is $d$-maximal at $\omega$, then
  $s_X(\omega_i)\to s_X(\omega)$.
\end{continuity}

We assume now that the intersection of $X$ with the dense torus $T$ in $Y_\Delta$ is dense in $X$. We can apply Theorem \ref{thm:continuity} to the case that the intersection $X \cap O(\sigma)$ of $X$ with all torus orbits is equidimensional. We deduce in Theorem~\ref{thm:continuity_equi} that if $\Trop(X) \cap N_\R$ can be covered by finitely many maximal-dimensional polyhedra $P$ which project to polyhedra of maximal dimension in all boundary strata which are met by the closure of $P$, then 
  $s_X\colon\multone\to X^\an$ is continuous.
As an immediate consequence, we get the neat criterion in Theorem \ref{cor:proper.int.continuity} which we highlighted before. 

In~\S\ref{sec:helm-katz} we specialize to the case of a so-called sch\"on
subvariety $X$ of a torus $T$, defined over a discretely valued subfield
$K_0\subset K$.  In this situation, $X$ admits a compactification $\sX$ in a
toric scheme over the valuation ring $K^\circ\subset K$, such that the pair of
$\sX$ with the boundary divisor on the generic fiber $H$ form a strictly
semistable pair in the sense
of~\cite{gubler_rabinoff_werner16:skeleton_tropical}.  This allows us to use the
results of \textit{loc.\ cit.}\ to define a skeleton $S(\sX,H)\subset X^\an$
associated to the pair $(\sX,H)$.  In Theorem~\ref{thm:skeletons.complex} we
show that $S(\sX,H)$ coincides with the tropical skeleton $\STrop(X)$ (as
subsets of $X^\an$), and that both are naturally isomorphic to the
parameterizing complex $\Gamma_X$ defined by
Helm--Katz~\cite{helm_katz12:monodrom_filtration}.  As a consequence, the
parameterizing complex is a deformation retract of $X^\an$, so the canonical isomorphism 
$H^r(\Gamma_X,\Q_\ell) \cong W_0 H_{\text{\'et}}^r(X_{\bar K_0}, \Q_\ell)$ 
of~\cite{helm_katz12:monodrom_filtration} follows, at least over local fields, from a very general comparison
result of Berkovich; see Remark~\ref{helm-katz:main result}. Here, $\ell$ is assumed to be different from the residue characteristic of $K$, and the isomorphism relates the singular cohomology to the weight zero part of the \'etale cohomology of the base change of $X$ to the algebraic closure $\bar K_0$.

The interplay between tropical and analytic geometry has been intensely studied
during the last years. It plays an important role in the investigation of
tropical moduli spaces~\cite{abramovic_caporaso_payne15:tropical_moduli_space}
and also for applications of tropical geometry to arithmetic
problems~\cite{krzb16:uniform_bounds}. We hope that the general conceptual
picture of the relationship between analytic and tropical subschemes of toric
varieties which we develop in this paper will prove useful for further
developments in this area.

\subsection{Acknowledgments}
Part of the work on this project took place during the Simons symposium on
{\it Non-archimedean and Tropical Geometry} in 2015.  The authors are very
grateful to the Simons Foundation for financial support of this meeting and to
the organizers, Matt Baker and Sam Payne. 
{The first author was
partly supported by the collaborative research center SFB 1085 funded by the
Deutsche Forschungsgemeinschaft.  The second author was sponsored by the
National Security Agency Young Investigator Grant number H98230-15-1-0105. 
We also thank the referee for his helpful comments.}

\section{Analytic spaces and their reductions}\label{sec:analytic spaces}

In this section we present some technical facts about analytic
spaces, mostly concerning reductions and Shilov boundaries.  
Since in our investigation of torus orbits we are forced to consider reducible varieties and since
we want to include the case of an arbitrary non-archimedean ground field endowed possibly with the trivial absolute value,
we have to provide some proofs which we could not locate in the literature in the required generality.

We assume that the reader is familiar with the terminology used in 
Berkovich's book \cite{berkovic90:analytic_geometry}.

\subsection{General notation}
This paper uses standard notations from the fields of non-Archimedean analytic
geometry and toric geometry.  Appendix~A contains a list of notations.

If $X$ is an object (scheme, analytic space, formal scheme, algebra, arrow) over
a ring $K$ and $K\to K'$ is a ring homomorphism, the extension of scalars of $X$
to $K'$ will be denoted $X_{K'}$ when convenient. \label{notn:ext.scalar}

\subsection{Non-Archimedean fields}
By a \emph{non-Archimedean field} we mean a field $K$ which is complete with
respect to a (potentially trivial) non-Archimedean valuation
\label{notn:valuation}
$\val\colon K\to\R\cup\{\infty\}$.  If $K$ is a non-Archimedean field then we
write \label{notn:ints} $K^\circ$ for its ring of integers,
$K^\ccirc$ \label{notn:ideal} for the maximal ideal in $K^\circ$, and
$\td K = K^\circ/K^\ccirc$ \label{notn:res.field} for the residue field.  We
also write $\Gamma = \Gamma_K = \val(K^\times)$ \label{notn:value.gp} for the value group of $K$ and
$\sqrt\Gamma$ \label{notn:saturation} for its saturation in $\R$.  Let $|\scdot| = \exp(-\val(\scdot))$
\label{notn:abs.val}
be a corresponding absolute value on $K$.

Throughout this paper, $K$ will denote a non-Archimedean field.
\label{notn:nonarch.field}
By a \emph{valued extension field} of $K$ we mean a
non-Archimedean field $K'$ containing $K$ such that the valuation on $K'$
restricts to the valuation on $K$. 

{For $r_1,\ldots,r_n\in\R_{>0}$, we denote the \emph{ generalized Tate algebra} by} 

\label{notn:gen.tate.alg}
\[ K\angles{r_1\inv x_1,\ldots,r_n\inv x_n} =
\bigg\{ \sum_{I\in\Z_{\geq 0}^n} a_Ix^I\mid |a_I|r^I\to 0
 \text{ as } |I|\to\infty \bigg\}. \]

\subsection{Analytic spaces}
{We will generally
use calligraphic letters to denote $K$-affinoid algebras. The Berkovich spectrum of a (strictly) $K$-affinoid
algebra $\cA$ is denoted \label{notn:spectrum} $\sM(\cA)$. These are the building blocks of a Berkovich (strictly) $K$-analytic
space $X$, see \cite{berkovic99:locally_contractible_I}. 
For $x\in X$ we let
$\sH(x)$ \label{notn:completed.res.field} denote the completed residue field at
$x$.  This is a valued extension field of $K$.

Of major importance for us are \emph{good  (strictly) $K$-analytic
spaces}  which means that every point has a  neighborhood of the form
$\sM(\cA)$, where $\cA$ is a (strictly) $K$-affinoid algebra. Note that only good $K$-analytic spaces are considered in \cite{berkovic90:analytic_geometry}.}

For any $K$-scheme $X$ locally of finite type, we let $X^\an$ \label{notn:analytification} denote its
analytification, as constructed
in~\cite[\S3.4--3.5]{berkovic90:analytic_geometry}.  This is a good strictly
$K$-analytic space. 

\subsection{Dimension theory and irreducible decomposition}
The basic dimension theory of $K$-analytic spaces is developed
in~\cite[\S2.3]{berkovic90:analytic_geometry}.  The \emph{dimension} $\dim(X)$
of a strictly $K$-affinoid space $X = \sM(\cA)$ is by definition the Krull
dimension of $\cA$.  The dimension of a general $K$-affinoid space $X$ is the
dimension of $X_{K'}$ for $K'/K$ a valued field extension such that $X_{K'}$ is
strictly $K'$-affinoid.  The \emph{dimension} $\dim(X)$ of a $K$-analytic space
$X$ is the maximum dimension of a $K$-affinoid domain in $X$.  If $X$ is strictly
$K$-analytic then this is equal to the maximal Krull dimension of the stalk
$\sO_{X,x}$ at a point $x$ of the rigid analytic variety associated to $X$.  We say that a $K$-analytic space $X$ is
\emph{equidimensional of dimension $d$} provided that every $K$-affinoid
domain in $X$ has dimension $d$.  The analytification of a $K$-scheme $X$ of
dimension $d$ (resp.\ of equidimension $d$) has dimension $d$ (resp.\
equidimension $d$) by~\cite[Lemma~A.1.2(2)]{conrad99:irred_comps}.

Let $X$ be a good $K$-analytic space and let $x\in X$.  The
\emph{local dimension} $\dim_x(X)$ is the infimum of $\dim(V)$ for $V\subset X$
an affinoid neighborhood of $x$.  One has $\dim(X) = \max_{x\in X}\dim_x(X)$,
and $X$ is equidimensional if and only if $\dim_x(X) = \dim(X)$ for all $x\in X$.

Let $X = \sM(\cA)$ be a $K$-affinoid space.  The \emph{irreducible components}
of $X$ are the reduced Zariski-closed subspaces of $X$ defined by the minimal
prime ideals of $\cA$.  
Each irreducible component is equidimensional, and $X$ is equidimensional if and
only if its irreducible components have the same
dimension by~\cite[Proposition~2.3.5]{berkovic90:analytic_geometry}.
If $X$ is $K$-affinoid then $\dim_x(X)$ is the maximal dimension of an
irreducible component containing $x$.

See~\cite{conrad99:irred_comps} for a global theory of irreducible components.
We need to extend the following result, found in \textit{ibid.}, to
the case of non-strict $K$-affinoid domains.

\begin{prop}\label{prop:affinoid.union.comps}
  Let $X$ be a finite-type $K$-scheme, let $Y\subset X$ be an irreducible
  component, and let $U\subset X^\an$ be a (possibly non-strict) $K$-affinoid
  domain.  Then $Y^\an\cap U$ is a union of irreducible components of $U$.
\end{prop}

\begin{proof}
  Let $K'\supset K$ be a valued extension field of $K$ which is non-trivially
  valued and such that $U_{K'}$ is strictly $K'$-affinoid.  Let
  $X' = X_{K'}$, $Y' = Y_{K'}$, and $U' = U_{K'}$.  Then the reduced space
  underlying $Y'^\an\cap U'$ is a union of irreducible components of $U'$
  by~\cite[Corollary~2.2.9, Theorem~2.3.1]{conrad99:irred_comps}.  Let
  $U = \sM(\cA)$ and let $\cA'=\cA\hat\tensor_K K'$, so $U' = \sM(\cA')$.  Then
  $\cA'$ is a faithfully flat $\cA$-algebra
  by~\cite[Lemma~2.1.2]{berkovic93:etale_cohomology}.  
Suppose that $Y\cap U$ is defined by the ideal $\fa\subset\cA$, so $Y'\cap U'$ is defined by  $\fa' = \fa\cA'$.  The question is now purely one of commutative algebra: if $\cA\to\cA'$ is a faithfully flat homomorphism of Noetherian rings and $\sqrt{\fa'}$ is an intersection of minimal prime ideals $\wp_i'$ of $\cA'$, then $\sqrt\fa$ is the intersection of the minimal prime ideals $\wp_i' \cap \ cA$ of $\cA$. This follows from~\cite[Proposition~2.3.4]{egaIV_2}.
\end{proof}

\subsection{Admissible formal schemes}
Suppose now that the valuation on $K$ is non-trivial.  An
\emph{admissible $K^\circ$-algebra} in the sense
of~\cite{bosch_lutkeboh93:formal_rigid_geometry_I} is a $K^\circ$-algebra $A$
which is topologically finitely generated and flat (i.e.\ torsionfree) over
$K^\circ$.  We will generally use Roman letters to denote admissible
$K^\circ$-algebras.  An \emph{admissible $K^\circ$-formal scheme} is a formal
scheme $\fX$ which has a cover by formal affine opens of the form $\Spf(A)$ for
$A$ an admissible $K^\circ$-algebra.  The special fiber of $\fX$ is denoted
\label{notn:special.fiber}
$\fX_s = \fX\tensor_{K^\circ} \td K$; this is a $\td K$-scheme locally of finite
type.  If $\fX$ has a locally finite atlas, the \emph{analytic generic fiber}
$\fX_\eta$ \label{notn:generic.fiber} of $\fX$ is the strictly $K$-analytic space defined locally by
$\Spf(A)_\eta = \sM(A\tensor_{K^\circ} K)$.  Here we recall several facts about
admissible $K^\circ$-algebras and affine admissible $K^\circ$-formal schemes.

\begin{prop}\label{prop:adm.fs.facts}
  Let $f\colon A\to B$ be a homomorphism of admissible $K^\circ$-algebras, let
  $\fX=\Spf(A)$ and $\fY=\Spf(B)$, and let $\phi\colon\fY\to\fX$ be the induced
  morphism.
  \begin{enumerate}
  \item $f$ is finite if and only if $f_K\colon A_K\to B_K$ is finite.
  \item Suppose that $f_K\colon A_K\to B_K$ is finite and dominant, i.e.\ that
    $\ker(f_K)$ is nilpotent.  Then $\phi_s\colon\fY_s\to\fX_s$ is finite and
    surjective.
  \item If $\fX_\eta$ has dimension $d$ (resp.\ is equidimensional of
    dimension $d$), then $\fX_s$ has dimension $d$ (resp.\ is equidimensional of
    dimension $d$).
  \item Suppose that $f_K\colon A_K\to B_K$ is finite and dominant, and that
    $\fX_\eta$ and $\fY_\eta$ are equidimensional (necessarily of the
    same dimension).  Then $\phi_s\colon\fY_s\to\fX_s$ maps generic points to
    generic points.
  \end{enumerate}
\end{prop}

\begin{proof}
  These are all found in~\cite[\S3]{baker_payne_rabinoff16:tropical_curves},
  except for the ``dimension $d$'' part of~(3), which uses a similar argument to
  the ``equidimension $d$'' part
  of~\cite[Proposition~3.23]{baker_payne_rabinoff16:tropical_curves}.  Note
  that the proofs in \textit{loc.\ cit.}\ do not use the standing assumption
  there that $K$ is algebraically closed.
\end{proof}

\subsection{} \label{par:notation.affinoid.algebras}
Assume that the valuation on $K$ is non-trivial.
Let $\cA$ be a strictly $K$-affinoid algebra and let $X = \sM(\cA)$.  We let
$\cA^\circ\subset\cA$ \label{notn:power.bounded} denote the subring of
power-bounded elements and we let \label{notn:top.nilpotent}
$\cA^\ccirc\subset\cA^\circ$ denote the ideal of topologically nilpotent
elements.  The ring $\cA^\circ$ is flat over $K^\circ$, but it may not be
topologically of finite type,    
{see ~\cite[Theorem~3.17]{baker_payne_rabinoff16:tropical_curves}.} 
If $|\scdot|_{\sup}$ \label{notn:sup.norm} is the supremum seminorm on $\cA$, then
$\cA^\circ = \{f\in\cA\mid|f|_{\sup}\leq 1\}$ and
$\cA^\ccirc = \{f\in\cA\mid|f|_{\sup} < 1\}$.  We set \label{notn:reduction.alg}
$\td\cA=\cA^\circ/\cA^\ccirc$.

\subsection{The canonical reduction} \label{par:canonical-reduction}
With the notation in \S\ref{par:canonical-reduction} 
the \emph{canonical model} of $X= \sM(\cA)$ is the affine formal
$K^\circ$-scheme \label{notn:canonical.model} $\fX^\can \coloneq \Spf(\cA^\circ)$.  The
\emph{canonical reduction} of $X$ is \label{notn:canonical.reduction} $\td X\coloneq\Spec(\td\cA)$.  This is a
reduced affine $\td K$-scheme of finite type, and the association
$X\mapsto\td X$ is functorial.  Since the radical of the ideal
$K^\ccirc\cA^\circ$ is equal to $\cA^\ccirc$, the canonical reduction $\td X$ is
the reduced scheme underlying the special fiber of the canonical model
$\fX^\can_s$.  If $\Gamma = \val(K^\times)$ is divisible (e.g.\ if $K$ is algebraically closed)
then $\td X = \fX^\can_s$.

We have the following analogue of Proposition~\ref{prop:adm.fs.facts} for the
canonical reduction.

\begin{prop}\label{prop:canon.red.facts}
  Let $f\colon\cA\to\cB$ be a homomorphism of strictly $K$-affinoid algebras,
  let $X = \sM(\cA)$ and $Y = \sM(\cB)$, and let $\phi\colon Y\to X$ be the
  induced morphism.
  \begin{enumerate}
  \item The following are equivalent: $f$ is finite;
    $f^\circ\colon\cA^\circ\to\cB^\circ$ is integral;
    $\td f\colon\td\cA\to\td\cB$ is finite.
  \item Suppose that $f$ is finite and dominant, i.e.\ that $\ker(f)$ is
    nilpotent.  Then $\td\phi\colon\td Y\to\td X$ is finite and surjective.
  \item If $X$ has dimension $d$ (resp.\ is equidimensional of
    dimension $d$), then $\td X$ has dimension $d$ (resp.\ is equidimensional of
    dimension $d$).
  \item Suppose that $f$ is finite and dominant, and that $X$ and $Y$ are
    equidimensional (necessarily of the same dimension).  Then
    $\td\phi\colon\td Y\to\td X$ maps generic points to generic points.
  \end{enumerate}
\end{prop}

\begin{proof}
  Part~(1)
  is~\cite[Theorem~6.3.5/1]{bosch_guntzer_remmert84:non_archimed_analysis}.  In
  the situation of~(2), we know from~(1) that $\td\phi$ is finite.  Let
  $\fX^\can = \Spf(\cA^\circ)$ and $\fY^\can = \Spf(\cB^\circ)$.  Then
  $\td\phi\colon\td Y\to\td X$ is the morphism of reduced schemes underlying
  $\phi_s\colon\fY^\can_s\to\fX^\can_s$, so $\phi_s$ is a closed map.  Hence to
  show $\td\phi$ is surjective, it suffices to show that
  $\ker(f^\circ_{\td K}\colon\cA^\circ_{\td K}\to\cB^\circ_{\td K})$ is nilpotent.
  Let $a\in\cA^\circ$ have zero image in $\cB^\circ_{\td K}$.  Then
  $f(a) = \varpi b$ for some $b\in\cB^\circ$ and $\varpi\in K^\ccirc$.  Let
  \[ b^n + f(c_{n-1})b^{n-1} + \cdots + f(c_1) b + f(c_0) = 0 \qquad
  (c_i\in\cA^\circ) \]
  be an equation of integral dependence for $b$ over $\cA^\circ$.  Then
  \[ a^n + \varpi\big( c_{n-1}a^{n-1} + \cdots + \varpi^{n-2}c_1a
  + \varpi^{n-1}c_0 \big) \in \ker(f), \]
  so the image of $a^n$ in $\cA^\circ_{\td K}$ is nilpotent, as desired.

  Now suppose that $X$ has dimension $d$.  By Noether
  normalization~\cite[Theorem~6.1.2/1]{bosch_guntzer_remmert84:non_archimed_analysis},
  there exists a finite injection $K\angles{x_1,\ldots,x_d}\inject\cA$, which
  by~(2) yields a finite, surjective morphism $\td X\to\bA^d_{\td K}$.  Hence
  $\td X$ has dimension $d$.  Suppose that $X$ is equidimensional of dimension
  $d$.  Choose a generic point $\td x\in\td X$, and choose $a\in\cA^\circ$ such
  that $\td a(\td x)\neq 0$ and $\td a$ vanishes on all other generic
  points of $\td X$.  Let $X'\subset X$ be the Laurent domain
  $\sM(\cA\angles{a\inv})$.  Then $\dim(X') = \dim(X)$ and
  $\td X' =
  \Spec(\td\cA[\td a\inv])$
  by~\cite[Proposition~7.2.6/3]{bosch_guntzer_remmert84:non_archimed_analysis}.
  By the above, $\td X'$ has dimension $d$, and hence $\td X$ has
  equidimension $d$.

  Part~(4) follows immediately from parts~(1)--(3).
\end{proof}

\subsection{Relating the two reductions}
We continue to assume the valuation on $K$ is non-trivial.
Let $A$ be an admissible $K^\circ$-algebra and let $\cA = A_K$, a strictly
$K$-affinoid algebra.  Put $\fX=\Spf(A)$, $X=\sM(\cA)$, and
$\fX^\can=\Spf(\cA^\circ)$, as above.  Then $A\subset\cA^\circ$ since by
definition $A$ is generated by power-bounded elements, so we obtain morphisms
\begin{equation}\label{eq:map.between.reductions}
  \fX^\can\to\fX \sptxt{and} \td X\inject\fX^\can_s\to\fX_s.
\end{equation}
These morphisms are functorial in $\fX$.  The next Proposition relates the
two finite-type $\td K$-schemes canonically associated with $A$.

\begin{prop}\label{prop:two.reductions}
  With the above notation, the natural inclusion $A\inject\cA^\circ$ is an
  integral homomorphism, and the morphism $\td X\to\fX_s$ is finite and
  surjective.
\end{prop}

\begin{proof}
  Choose a surjection $K^\circ\angles{x_1,\ldots,x_n}\surject A$.  Tensoring
  with $K$, we get a surjection $K\angles{x_1,\ldots,x_n}\surject\cA$, so
  by~\cite[Theorem~6.3.5/1]{bosch_guntzer_remmert84:non_archimed_analysis}, the
  composition 
  \[ K^\circ\angles{x_1,\ldots,x_n}\to A\to \cA^\circ \]
  is an integral homomorphism.  Hence $A\to\cA^\circ$ is integral.  It follows
  that $A_{\td K}\to\cA^\circ_{\td K}\surject\td\cA$ is an integral homomorphism of
  finitely generated $\td K$-algebras, so $A_{\td K}\to\td\cA$ and
  $\td X\to\fX_s$ are finite.  In particular, $\td X\to\fX_s$ is closed, and
  hence $\fX^\can_s\to\fX_s$ is closed, as $\td X$ and $\fX^\can_s$ have the
  same underlying topological space.  Thus to show surjectivity it is enough to
  prove that $\ker(A_{\td K}\to\cA^\circ_{\td K})$ is nilpotent.
  This is done exactly as in the proof of
  Proposition~\ref{prop:canon.red.facts}(2). 
\end{proof}

\subsection{The Shilov boundary and the reduction map}
Here the valuation on $K$ is allowed to be trivial.  The \emph{Shilov boundary}
of a $K$-affinoid space $X = \sM(\cA)$ is the unique minimal subset \label{notn:shilov.boundary}
$B(X)\subset X$ on which every $f\in\cA$ achieves its maximum.
It exists and is finite and nonempty for any $K$-affinoid space
by~\cite[Corollary~2.4.5]{berkovic90:analytic_geometry}.  The Shilov boundary is
insensitive to nilpotent elements of $\cA$.

We postpone the proof of the following Lemma until after
Proposition~\ref{prop:fs.red.surj}.

\begin{lem}\label{lem:bdy.kernel}
  Let $X = \sM(\cA)$ be a $K$-affinoid space, let $x\in B(X)$ be a Shilov
  boundary point, and let $|\scdot|_x\colon\cA\to\R_{\geq 0}$ be the
  corresponding  seminorm.  Then $\ker|\scdot|_x$ is a minimal prime
  ideal of $\cA$.
\end{lem}

\begin{prop}\label{prop:bdy.irred.comps}
  Let $X = \sM(\cA)$ be a $K$-affinoid space and let $X = X_1\cup\cdots\cup X_n$
  be its decomposition into irreducible components.  Then each Shilov boundary
  point of $X$ is contained in a unique irreducible component $X_i$. Moreover,
  we have $B(X) = B(X_1)\djunion\cdots\djunion B(X_n)$, and if $x\in B(X_i)$
  then $\dim_x(X) = \dim(X_i)$.
\end{prop}

\begin{proof}
  It follows from Lemma~\ref{lem:bdy.kernel} that a Shilov boundary point
  $x\in B(X)$ is contained in a unique irreducible component $X_i$, namely, the
  one defined by the prime ideal $\ker|\scdot|_x\subset\cA$.  For $f\in\cA$, by
  definition the restriction of $f$ to $X_i$ achieves its maximum value on
  $B(X_i)$, so it is clear that $B(X)\subset\bigcup_{i=1}^n B(X_i)$.
  
  Let $x\in B(X_i)$ for some $i=1,\ldots,n$.  Choose $f_i\in\cA$ which vanishes
  identically on $\bigcup_{j\neq i} X_j$ but not on $X_i$.  By
  Lemma~\ref{lem:bdy.kernel}, $|f_i(x)|\neq 0$.  Choose also $g\in\cA$ such that
  $|g|$ attains its maximum value on $X_i$ only at $x$, i.e.\ such that
  $|g(x)| > |g(x')|$ for all $x'\in B_i\setminus\{x\}$.  Using that $B(X)$ is
  finite, $g^n f_i\in\cA$ achieves its maximum only on $x$ for $n\gg 0$, so
  $x\in B$.  Thus $B(X) = \bigcup_{i=1}^n B(X_i)$, and this union is disjoint
  because any $x\in B(X)$ lies on only one $X_i$.  The final assertion is clear
  because $x\in B(X_i)$ admits an equidimensional $K$-affinoid neighborhood
  contained in $X_i$, namely, $\{|f_i|\geq\epsilon\}$ for $f_i$ as above and
  $\epsilon$ small.
\end{proof}

Now we assume that the valuation on $K$ is non-trivial.
Let $X = \sM(\cA)$ be a strictly $K$-affinoid space. 
{By \cite[\S 2.4]{berkovic90:analytic_geometry}, there is a canonical reduction map} 
\begin{equation}\label{eq:red.can}
  \red\colon X\To\td X.
 \end{equation}
We have the following relationship between the reduction map and the Shilov
boundary, proved in~\cite[Proposition~2.4.4]{berkovic90:analytic_geometry}.

\begin{prop}\label{prop:can.red.surj}
  Let $X = \sM(\cA)$ be a strictly $K$-affinoid space, and let
  $\red\colon X\to\td X$ be the reduction map.
  \begin{enumerate}
  \item $\red$ is surjective, anti-continuous,%
    \footnote{The inverse image of an open set is closed.}
    and functorial in $X$.
  \item If $\td x\in\td X$ is a generic point, then $\red\inv(\td x)$ consists
    of a single point.
  \item The inverse image under $\red$ of the set of generic points of $\td X$
    is equal to the Shilov boundary of $X$.
  \end{enumerate}
\end{prop}

Now let $\fX = \Spf(A)$ be an affine admissible $K^\circ$-formal scheme, and let
$X = \fX_\eta$.  Recall from~\eqref{eq:map.between.reductions} that we have a
natural finite, surjective morphism $\td X\to\fX_s$.  We define the
\emph{reduction map}
\begin{equation}\label{eq:red.model}
  \red\colon X\to\fX_s
\end{equation}
to be the composition of $\red\colon X\to\td X$ with the morphism
$\td X\to\fX_s$.  This construction globalizes: if $\fX$ is any admissible
$K^\circ$-formal scheme with generic fiber $X = \fX_\eta$, then one obtains a
reduction map $\red\colon X\to\fX_s$ by working on formal affines and gluing. 

\begin{prop}\label{prop:fs.red.surj}
  Let $\fX$ be an admissible $K^\circ$-formal scheme with generic fiber
  $X = \fX_\eta$.  Then the reduction map $\red\colon X\to\fX_s$ is
  surjective, anti-continuous, and functorial in $\fX$.
\end{prop}

\begin{proof}
  We reduce immediately to the case of an affine formal scheme, where the result
  follows from Propositions~\ref{prop:two.reductions} and~\ref{prop:can.red.surj}.
\end{proof}

\begin{proof}[Proof of Lemma~\ref{lem:bdy.kernel}]
  By passing to an irreducible component of $X$ containing $x$, we may assume
  that $\cA$ is an integral domain.  First we suppose that the valuation on $K$
  is non-trivial and that $X$ is strictly $K$-affinoid.  Then $\red(x)$ is the
  generic point of the canonical reduction $\td X$ by
  Proposition~\ref{prop:can.red.surj}(3).  The canonical reduction is an
  equidimensional scheme of the same dimension as $X$ by
  Proposition~\ref{prop:canon.red.facts}(3), so $x$ cannot be contained in a
  smaller-dimensional Zariski-closed subspace of $X$ by functoriality of the
  reduction map.

  Now we suppose that $X$ is not strictly $K$-affinoid or that the valuation on
  $K$ is trivial (or both).  There exists a non-trivially-valued non-Archimedean
  extension field $K'\supset K$, namely, the field $K'=K_{r_1,\ldots,r_n}$ for
  suitable $r_1,\ldots,r_n\in\R_{>0}$ as
  in~\cite[\S2.1]{berkovic90:analytic_geometry}, such that
  $X' = X\hat\tensor_{K}K'$ is strictly $K'$-affinoid.  Let
  $\cA'=\cA\hat\tensor_KK'$, so $X' = \sM(\cA')$.  Then $\cA'$ is an integral
  domain by~\cite[Proposition~2.1.4(iii)]{berkovic90:analytic_geometry} and the
  proof of~\cite[Proposition~2.1.4(ii)]{berkovic90:analytic_geometry}.  Let
  $\pi\colon X'\to X$ be the structural morphism.  We have $\pi(B(X')) = B(X)$
  by the proof of~\cite[Proposition~2.4.5]{berkovic90:analytic_geometry}, so
  there exists $x'\in B(X')$ with $\pi(x') = x$.  The above argument in the
  strictly $K$-affinoid case shows that
  the seminorm $|\scdot|_{x'}\colon\cA'\to\R_{\geq 0}$ is a norm, i.e., has
  trivial kernel.  But $\cA\to\cA'$ is injective
  by~\cite[Proposition~2.1.2(i)]{berkovic90:analytic_geometry}, so $|\scdot|_x$
  is a norm.
\end{proof}

\subsection{Extension of scalars}
In the sequel it will be important to understand the behavior of extension of
the ground field with respect to the underlying topological space and the Shilov
boundary.  First we make a simple topological observation.  Here the valuation
on $K$ is allowed to be trivial.

\begin{lem}\label{lem:extend.closed}
  Let $K'\supset K$ be a valued extension of $K$ and let $X$ be a good
  $K$-analytic space.  Then the structural morphism $\pi\colon X_{K'}\to X$ is a
  proper and closed map on underlying topological spaces.
\end{lem}

\begin{proof}
  Since $X$ is good, every point $\xi\in X$ has a $K$-affinoid
  neighborhood $U$, which is compact.  The inverse image
  $U' = \pi\inv(U) = U_{K'}$ is also affinoid, hence compact.  Thus
  every point of $X^\an$ admits a compact neighborhood whose inverse
  image under $\pi$ is compact.  It is clear that any such map is both proper
  and closed.
\end{proof}

The following Lemma is an analogue
of~\cite[Proposition~3.1(v)]{temkin04:local_properties_II}.

\begin{lem}\label{lem:canon.red.extn.scalars}
  Suppose that the valuation on $K$ is non-trivial.  Let $X = \sM(\cA)$ be a
  strictly $K$-affinoid space, let $K'$ be a valued field extension of $K$, and
  let $X' = X_{K'}$.  Then the canonical morphism
  \[ \td X'\to\td X\tensor_{\td K}\td K' \]
  is finite.
\end{lem}

\begin{proof}
  There exists an admissible $K^\circ$-algebra $A\subset\cA^\circ$ such that
  $A_K = \cA$ (let $A$ be the image of $K\angles{x_1,\ldots,x_n}^\circ$ under a
  surjection $K\angles{x_1,\ldots,x_n}\surject\cA$).  Let
  $A' = A\hat\tensor_{K^\circ}K'^\circ$, let $\fX=\Spf(A)$, and let
  $\fX'=\Spf(A')$.  Then $A'$ is an admissible $K'^\circ$-algebra and
  $\fX'_\eta = X'$.  We have a commutative square of affine $\td K'$-schemes
  \[\xymatrix @=.2in{
    {\td X'} \ar[r] \ar[d] & {\td X\tensor_{\td K}\td K'} \ar[d] \\
    {\fX_s'} \ar[r] & {\fX_s\tensor_{\td K}\td K'} }\]
  where the left and right arrows are the morphisms coming
  from~\eqref{eq:map.between.reductions}. They are finite by Proposition~\ref{prop:two.reductions}. Since the bottom arrow is an
  isomorphism, the top morphism is also finite.
\end{proof}

Recall that the Shilov boundary of a $K$-affinoid space $X$ is denoted $B(X)$.

\begin{prop}\label{prop:shilov.surj}
  Let $K'\supset K$ be valued extension of $K$ and let $X$ be a (possibly
  non-strict) $K$-affinoid space, let $X' = X_{K'}$, and let $\pi\colon X'\to X$
  be the structural map.  Then $\pi(B(X')) = B(X)$.
\end{prop}

\begin{proof}
  It is clear \textit{a priori} that $\pi(B(X'))\supset B(X)$
  since there are more analytic functions on $X'$ than on $X$.  Therefore we are
  free to replace $K'$ by a valued extension field.
  
  Suppose that $X$ is not strictly $K$-affinoid or that the valuation on $K$ is
  trivial (or both).  There exists a non-trivially-valued non-Archimedean
  extension field $K_\br\coloneq K_{r_1,\ldots,r_n}\supset K$ for suitable
  $r_1,\ldots,r_n\in\R_{>0}$ as in~\cite[\S2.1]{berkovic90:analytic_geometry},
  such that $X\hat\tensor_{K}K_\br$ is strictly $K_\br$-affinoid and such that
  the image of $B(X\hat\tensor_{K}K_\br)$ under the structural morphism
  $X\hat\tensor_{K}K_\br\to X$ is equal to $B(X)$.
  By~\cite[(0.3.2)]{ducros09:excellent}, there exists a valued extension
  $K''\supset K$ which is simultaneously a valued extension field of $K'$ and
  $K_\br$.  Replacing $K$ by $K_\br$ and $K'$ by $K''$, we may assume that $X$
  is strictly $K$-affinoid and that the valuation on $K$ is non-trivial.

  Let $X = X_1\cup\cdots\cup X_n$ be the irreducible decomposition of $X$.  Then
  $B(X) = \bigcup_{i=1}^n B(X_i)$ and $B(X') = \bigcup_{i=1}^n B(\pi\inv(X_i))$
  by Proposition~\ref{prop:bdy.irred.comps}, since $\pi\inv(X_i)$
  is a union of irreducible components of $X'$.  Replacing $X$ by $X_i$ and $X'$
  by $\pi\inv(X_i) = (X_i)_{K'}$, we may assume $X$ is irreducible.

  Consider the commutative diagram
  \[\xymatrix @R=.2in{
    {X'} \ar[r]^\red \ar[d]_(.4)\pi &
    {\td X'} \ar[r] \ar[d]_(.4){\td\pi} &
    {\td X\tensor_{\td K}\td K'} \ar[dl] \\ 
    {X} \ar[r]^\red & {\td X} &
  }\]
  By~\cite[Lemma~2.1.5]{conrad99:irred_comps}, $X'$ is equidimensional of
  dimension $d = \dim(X)$, so the same is true of $\td X'$ by
  Proposition~\ref{prop:canon.red.facts}(3).  Also $\td X\tensor_{\td K}\td K'$
  is equidimensional of dimension $d$, and $\td X\tensor_{\td K}\td K'\to\td X$
  sends generic points to generic points.  As the morphism
  $\td X'\to\td X\tensor_{\td K}\td K'$ is finite by
  Lemma~\ref{lem:canon.red.extn.scalars}, it takes generic points to generic
  points, so $\td\pi\colon\td X'\to\td X$ takes generic points to generic
  points.  It follows from proposition~\ref{prop:can.red.surj} that 
  $\pi\colon X'\to X$ takes Shilov boundary points to Shilov boundary points.
\end{proof}

\section{Toric varieties and tropicalizations}\label{sec:toric and tropical}

In this section we present the notions and notations that we will use for toric
varieties and their tropicalizations.  We generally
follow~\cite{gubler13:guide_tropical,rabinoff12:newton_polygons}.
See Appendix~A for a list of notations.

\subsection{Toric varieties}\label{subsection:toric_varieties}
Fix a finitely generated free abelian group \label{notn:char.group} $M\cong\Z^n$, and let
$N = \Hom(M,\Z)$.  We write \label{notn:eval.pairing} $\angles{\scdot,\scdot}\colon M\times N\to\Z$ for
the evaluation pairing.  For an additive subgroup $G\subset\R$ we set \label{notn:bigger.char.group}
$M_G = M\tensor_\Z G$ and $N_G = N\tensor_\Z G = \Hom(M,G)$, and we extend
$\angles{\scdot,\scdot}$ to a pairing $M_G\times N_G\to\R$.

For a ring $R$ and a monoid $S$ we write $R[S]$ \label{notn:monoid.ring} for the monoid ring on $S$; for
$u\in S$ we let $\chi^u\in R[S]$ \label{notn:character} denote the corresponding
character.  We set \label{notn:the.torus}
$T = \Spec(K[M])\cong\bG_{m,K}^n$, the split $K$-torus with character lattice
$M$ and cocharacter lattice $N$, where $K$ is our fixed non-Archimedean field.
For a rational cone $\sigma\subset N_\R$ we let $S_\sigma = \sigma^\vee\cap M$
\label{notn:affine.toric}
and $Y_\sigma = \Spec(K[S_\sigma])$, 
an affine toric variety.  Given a rational pointed fan $\Delta$ in $N_\R$ we let
\label{notn:general.toric}
$Y_\Delta = \bigcup_{\sigma\in\Delta} Y_\sigma$ denote the corresponding
$K$-toric variety with dense torus $T$.  For $\sigma\in\Delta$ we write
$M(\sigma) = \sigma^\perp\cap M$ and $N(\sigma) = N/\angles\sigma\cap N$, where
$\angles\sigma\subset N_\R$ is the linear span of $\sigma$.  For $G\subset\R$ as
above we write \label{notn:MGsigma} $M_G(\sigma) = M(\sigma)\tensor_\Z G$, etc., and by abuse of
notation we use $\angles{\scdot,\scdot}$ to denote the pairing
$M_G(\sigma)\times N_G(\sigma)\to\R$.   The torus orbit~\cite[\S3.2]{cox_little_schenck11:toric_varieties}
corresponding to $\sigma$ is \label{notn:torus.orbit} $O(\sigma) \coloneq \Spec(K[M(\sigma)])$, and
$Y_\Delta = \Djunion_{\sigma\in\Delta} O(\sigma)$ as sets.

We denote by \label{notn:pisigma}
$\pi_\sigma\colon N_\R\to N_\R(\sigma) = N_\R/\angles\sigma$  the natural 
projection. 

For $\tau\in\Delta$, the closure of
$O(\tau)$ in $Y_\Delta$ is the toric variety $Y_{\Delta_\tau}$ with dense
torus $O(\tau)$, where $\Delta_\tau$ is the fan
$\{\pi_\tau(\sigma)\mid\sigma\in\Delta,\,\tau\prec\sigma \}$.
We have 
\[ Y_{\Delta_\tau} = \Djunion_{\substack{\sigma\in\Delta\\\tau\prec\sigma}}
O(\sigma) = \Djunion_{\pi_\tau(\sigma) \in\Delta_\tau} O(\pi_\tau(\sigma)), \]
where $O(\sigma) = O(\pi_\tau(\sigma))$ for $\sigma\in\Delta$ such that
$\tau\prec\sigma$.  

\subsection{Kajiwara--Payne compactifications}
Much of this paper will be concerned with extended tropicalizations, which take
place in a Kajiwara--Payne partial compactification of $N_\R$.  We briefly
introduce these partial compactifications here;
see~\cite{payne09:analytif_tropical,rabinoff12:newton_polygons} for details.
Put \label{notn:barR} $\bar\R = \R\cup\{\infty\}$, and for a rational pointed cone
$\sigma\subset N_\R$ we let $\bar N{}_\R^\sigma$ \label{notn:barNRs} denote the space of monoid
homomorphisms $\Hom(S_\sigma,\bar\R)$.  As usual we let $\angles{\scdot,\scdot}$
denote the evaluation pairing $S_\sigma\times\bar N{}_\R^\sigma\to\bar\R$.  For a face
$\tau\prec\sigma$, a point $\omega\in N_\R(\tau)$ gives rise to a point of
$\bar N{}_\R^\sigma$ by the rule
\[ u \mapsto
\begin{cases}
  \angles{u,\omega} & \text{if } u\in\tau^\perp\\
  \infty       & \text{if not.}
\end{cases}\]
This yields a decomposition
$\bar N{}_\R^\sigma = \Djunion_{\tau\prec\sigma} N_\R(\tau)$ as sets, but not as
topological spaces.  In particular, $N_\R$ is a subset of $\bar N{}_\R^\sigma$.  For a
rational pointed fan $\Delta$ in $N_\R$, the spaces $\bar N{}_\R^\sigma$ glue to give a
partial compactification \label{notn:barNRD} $\bar N{}_\R^\Delta$, which is to $N_\R$ as $Y_\Delta$ is to
$T$.  We have a decomposition
$\bar N{}_\R^\Delta = \Djunion_{\sigma\in\Delta} N_\R(\sigma)$.  For an additive
subgroup $G\subset\R$ we let
$\bar N{}_G^\sigma = \Djunion_{\tau\prec\sigma} N_G(\tau)\subset\bar N{}_\R^\sigma$ and
$\bar N{}_G^\Delta = \Djunion_{\sigma\in\Delta} N_G(\sigma)\subset\bar N{}_\R^\Delta$.

If we pass to the fan $\Delta_\tau$ defined above, the associated partial  compactification $\bar{N(\tau)}{}_\R^{\Delta_\tau}$ of $N_\R(\tau)$ is the disjoint union of all $N(\sigma)$ for $\tau \prec \sigma$. 

\subsection{Tropicalization}\label{par:tropicalization}
Recall that $K$ is a non-Archimedean field whose valuation is allowed to be trivial.
For a rational cone $\sigma\subset N_\R$ we define
the \emph{tropicalization map} \label{notn:trop.map} $\trop\colon Y_\sigma^\an\to\bar N{}_\R^\sigma$ by
\[ \angles{u,\,\trop(\xi)} = -\log|\chi^u(\xi)|, \]
where we interpret $-\log(0) = \infty$.  This map is continuous, surjective,
closed, and proper, in the sense that the inverse image of a compact subset is
compact.  Choosing a basis for $M$ gives an isomorphism
$K[M]\cong K[x_1^{\pm1},\ldots,x_n^{\pm1}]$; when $\sigma=\{0\}$ the
tropicalization map $\trop\colon\bG_m^n = T\to N_\R = \R^n$ is given by
\[ \trop(\xi) = \big(-\log|x_1(\xi)|,\ldots,-\log|x_n(\xi)|\big). \]
The tropicalization map on $Y_\sigma^\an$ has a natural continuous section
$s\colon\bar N{}_\R^\sigma\to Y_\sigma^\an$ given by \label{notn:trop.section} $s(\omega) = |\scdot|_\omega$, where
$\omega\in\Hom(S_\sigma,\bar\R)$ and $|\scdot|_\omega\colon K[S_\sigma]\to\R$ \label{notn:norm.omega}
is the
multiplicative seminorm defined by
\begin{equation}\label{eq:section.seminorm}
  \bigg|\sum_{u\in S_\sigma}a_u\chi^u\bigg|_\omega =
  \max_u\big\{|a_u|\,\exp(-\angles{u,\omega})\big\},
\end{equation}
where we put $\exp(- \infty) = 0$.

The image $s(\bar N{}_\R^\sigma)$
is by definition the \emph{skeleton} \label{notn:skeleton.toric} $S(Y_\sigma^\an)\subset Y_\sigma^\an$ of
the affine toric variety $Y_\sigma^\an$.  The image of a continuous section of a
continuous map between Hausdorff spaces is closed, so
$S(Y_\sigma^\an)$ is closed in $Y_\sigma^\an$.

If $\Delta$ is a rational pointed fan in $N_\R$, then the maps $\trop$ glue to
give a tropicalization map $\trop\colon Y_\Delta^\an\to\bar N{}_\R^\Delta$.
This map is again continuous and proper.  The sections also glue to a continuous
section $s\colon\bar N{}_\R^\Delta\to
Y_\Delta^\an$,
whose image is by definition the \emph{skeleton} \label{notn:skeleton.toric2} $S(Y_\Delta^\an)$ of the toric
variety $Y_\Delta^\an$.  Again the skeleton $S(Y_\Delta^\an)$ is closed in
$Y_\Delta^\an$.  The tropicalization and section are compatible with the
decompositions $Y_\Delta = \Djunion_{\sigma\in\Delta} O(\sigma)$ and
$\bar N{}_\R^\Delta
= \Djunion_{\sigma\in\Delta} N_\R(\sigma)$,
in that $\trop$ (resp.\ $s$) restricts to the tropicalization map
$O(\sigma)^\an\to N_\R(\sigma)$ (resp.\ the section
$N_\R(\sigma)\to O(\sigma)^\an$) defined on the torus
$O(\sigma) = \Spec(K[M(\sigma)])$ with cocharacter lattice $N(\sigma)$.

For a closed subscheme $X\subset Y_\Delta$, the \emph{tropicalization} of $X$ is
\label{notn:TropX}
\[ \Trop(X) \coloneq \trop(X^\an)\subset\bar N{}_\R^\Delta. \]
For any $\sigma\in\Delta$ we have
$\Trop(X)\cap N_\R(\sigma) = \Trop(X\cap O(\sigma))$, so that $\Trop(X)$ may be
defined on each torus orbit separately.  The tropicalization is insensitive to
extension of scalars: if $K'\supset K$ is a valued field extension then
$\Trop(X_{K'}) = \Trop(X)$ as subsets of $\barNRD$.

\subsection{Fibers of tropicalization}
For $\omega\in N_\R$ the subset \label{notn:Uomega}
$U_\omega\coloneq \trop\inv(\omega)\subset T^\an$ is a $K$-affinoid domain,
which is strictly $K$-affinoid when $\omega\in N_{\sqrt\Gamma}$.  
See~\cite[Proposition~4.1]{gubler07:tropical_varieties} and its proof.
The ring of analytic functions on
$U_\omega$ is
\[ K\angles{U_\omega} = \bigg\{ \sum_{u\in M}a_u\chi^u\mid a_u\in K,\,
\val(a_u)+\angles{u,\omega}\to\infty \bigg\}, \]
where the sums are infinite and the limit is taken on the complement of finite
subsets {of $M$}. 
The supremum semi-norm $|\scdot|_{\sup}$ on $K\angles{U_\omega}$ is
multiplicative, and is given by the formula
\[ -\log\bigg|\sum_{u\in M} a_u\chi^u\bigg|_{\sup} =
\min\big\{\val(a_u)+\angles{u,\omega}\mid a_u\neq 0\big\}. \]
Compare~\eqref{eq:section.seminorm}.  Suppose now that the valuation on $K$ is
non-trivial and that $\omega\in N_{\sqrt\Gamma}$, so that $U_\omega$ is strictly
affinoid.  Then the ring of power-bounded elements in $K\angles{U_\omega}$ is
\[ K\angles{U_\omega}^\circ = \bigg\{ \sum_{u\in M}a_u\chi^u\in
K\angles{U_\omega}\mid \val(a_u)+\angles{u,\omega}\geq 0\text{ for all } u\in M
\bigg\}. \]
{If the value group $\Gamma$ is discrete, then $K\angles{U_\omega}^\circ$ is an admissible $K^\circ$-algebra (i.e., it is topologically of finite
presentation) for all $\omega \in N_{\sqrt\Gamma}$ by~\cite[Proposition~6.7]{gubler13:guide_tropical}. If $\Gamma$ is not discrete, then $K\angles{U_\omega}^\circ$ is an admissible $K^\circ$-algebra if and only if $\omega \in N_{\Gamma}$. This follows from~\cite[Proposition~6.9]{gubler13:guide_tropical} by noting that $K\angles{U_\omega}^\circ$ is the completion of what is denoted
$K[M]^{\{\omega\}}$ in \textit{ibid}. 

If  $K\angles{U_\omega}^\circ$  is an admissible $K^\circ$-algebra, then we set \label{notn:fUomega}
$\fU_\omega = \Spf(K\angles{U_\omega}^\circ)$, which is in our terminology the canonical model of
$U_\omega$. }

Let $\Delta$ be a pointed rational fan in $N_\R$ and assume that 
$\omega\in\bar N{}_\R^\Delta$
is contained in $N_\R(\sigma)$ for $\sigma\in\Delta$. Then we define $U_\omega$,
$K\angles{U_\omega}$, and $\fU_\omega$ as above, with the torus orbit
$O(\sigma)$ replacing the torus $T$.

\subsection{Initial degeneration}\label{sec:initial.degen}
Let $\Delta$ be a pointed rational fan in $N_\R$ and let $X\subset Y_\Delta$ be
a closed subscheme.  For $\omega\in\bar N{}_\R^\Delta$
we set \label{notn:Xomega} $X_\omega = U_\omega\cap X^\an$.  This is an affinoid domain in
$(X\cap O(\sigma))^\an$, where $\omega\in N_\R(\sigma)$; moreover, $X_\omega$ is
strictly $K$-affinoid when $\omega\in\bar N{}_{\sqrt\Gamma}^\Delta$,
and is a Zariski-closed subspace of $U_\omega$ in any case.  Suppose now that
the valuation on $K$ is non-trivial and that
$\omega\in\bar N{}_{\Gamma}^\Delta$.
Let $\fa_\omega\subset K\angles{U_\omega}$ be the ideal defining $X_\omega$.
The \emph{tropical formal model} of $X_\omega$ is the admissible formal scheme
\label{notn:fXomega}
$\fX_\omega$ defined as the closed formal subscheme of $\fU_\omega$ given by the
ideal $\fa_\omega\cap K\angles{U_\omega}^\circ$.  The
\emph{initial degeneration} of $X$ at $\omega$ is by definition the special
fiber of $\fX_\omega$: \label{notn:initial.degen}
\[ \inn_\omega(X) \coloneq (\fX_\omega)_s. \]
This closed subscheme of $(\fU_\omega)_s$ is defined by the $\omega$-initial
forms of the elements of $\fa_\omega$.
See~\cite[\S4.13]{baker_payne_rabinoff16:tropical_curves}
and~\cite[\S5]{gubler13:guide_tropical} for details.

Now let $\omega\in\barNRD$ be any point and let $K'\supset K$ be an
algebraically closed valued field extension whose value group
$\Gamma'=\val(K'^\times)$ is non-trivial and large enough that $\omega\in\barNGpD$.  Let
$X' = X_{K'}$.  Let $Z\subset\inn_\omega(X')$ be an irreducible component with
generic point $\zeta$ and let $m_Z$ \label{notn:comp.mult} be the multiplicity of $Z$, i.e.\ the length
of the local ring $\sO_{\inn_\omega(X'),\zeta}$.  The
\emph{tropical multiplicity} of $X$ at $\omega$ is \label{notn:mTrop}
\[ m_{\Trop}(\omega) = m_{\Trop}(X,\omega) \coloneq \sum_Z m_Z, \]
where the sum is taken over all irreducible components $Z$ of $\inn_\omega(X')$.
This quantity is independent of the choice of $K'$:
see~\cite[\S13]{gubler13:guide_tropical}.

\begin{rem} \label{finite surjective morphism}
Assuming the valuation on $K$ is non-trivial,
for $\omega\in\barNGD$
we let
\[ \fX_\omega^\can = \Spf\big((K\angles{U_\omega}/\fa_\omega)^\circ\big) \]
be the canonical model of $X_\omega$.  Then~\eqref{eq:map.between.reductions}
gives an integral morphism $\fX_\omega^\can\to\fX_\omega$ and a finite,
surjective morphism 
\[\td X_\omega\to(\fX_\omega)_s = \inn_\omega(X) \] 
by
Proposition~\ref{prop:two.reductions}.  Suppose now that $K$ is algebraically
closed and that $X$ is reduced, so that $\fX_\omega^\can$ is an admissible
affine formal scheme 
with
special fiber $\td X_\omega$  
{(see \cite[Theorem~3.17]{baker_payne_rabinoff16:tropical_curves} and \S\ref{par:canonical-reduction})}.  
The \emph{projection formula} in this
case~\cite[(3.34.2)]{baker_payne_rabinoff16:tropical_curves} says that for every
irreducible component $Z$ of $\inn_\omega(X)$, we have
\begin{equation}\label{eq:projection.formula}
  m_Z = \sum_{Z'\surject Z} [Z':Z],
\end{equation}
where the sum is taken over all irreducible components $Z'$ of
$\td X_\omega$ surjecting onto $Z$, and $[Z':Z]$ is the degree of the
finite dominant morphism of integral schemes $Z'\to Z$.
\end{rem}

\section{The tropical skeleton}\label{sec:tropical.skeleton}

Let $X$ be a closed subscheme of a $K$-toric scheme $Y_\Delta$.  In this section
we define a canonical locally closed subset $\STrop(X)\subset X^\an$ which we
call the \emph{tropical skeleton}.  We prove that $\STrop(X)$ is closed in
every torus orbit $O(\sigma)^\an$. For a polyhedron in $\Trop(X)$ and $\xi \in \STrop(X)$, we  introduce the 
notions of $d$-maximality  at $\xi$ and relevance for $\xi$  which will be crucial for studying limit points of  $\STrop(X)$ in $X^\an$ in the next section.

\subsection{Definition and basic properties}

For the rest of this section, we fix a pointed rational fan $\Delta$ in $N_\R$.
Let $X\subset Y_\Delta$ be a closed subscheme.
Recall that $\trop$ maps $Y_\Delta^\an$ to $\barNRD$ and that
$X_\omega\coloneq\trop\inv(\omega)\cap X^\an$ for $\omega\in\barNRD$.

\begin{defn}
  The~\emph{tropical skeleton} of a closed subscheme $X\subset Y_\Delta$ is the set
  \[ \STrop(X) \coloneq
  \big\{ \xi\in X^\an\mid\xi\text{ is a Shilov boundary point of } 
  X_{\trop(\xi)}\big\}. \]
\end{defn}

In other words, the tropical skeleton is the set of all Shilov boundary points
of fibers of tropicalization on $X^\an$.  It is clear that for $\sigma\in\Delta$ we
have $\STrop(X\cap O(\sigma)) = \STrop(X)\cap O(\sigma)^\an$, where the
left side of the equation is defined with respect to the closed subscheme
$X\cap O(\sigma)$ of the torus $O(\sigma)$.  In other words, the
tropical skeleton is defined independently on each torus orbit.  It is also
clear that $\trop$ maps $\STrop(X)$ surjectively onto
$\Trop(X) = \trop(X^\an)$.  

\begin{lem}\label{lem:skeleton.reduction}
  Let $X\subset Y_\Delta$ be a closed subscheme and let $X_{\red}$ be the
  underlying reduced subscheme.  Then $\STrop(X) = \STrop(X_{\red})$.
\end{lem}

\begin{proof}
  The Shilov boundary of an affinoid space $\sM(\cA)$ is obviously equal to the
  Shilov boundary of $\sM(\cA_{\red})$.
\end{proof}

Hence when discussing the tropical skeleton, we may always assume $X$ is
reduced.  The remainder of this subsection is  devoted to showing that
$\STrop(X)\cap O(\sigma)^\an$ is closed in $X^\an \cap O(\sigma)^\an$ for each $\sigma\in\Delta$.  This is
clear when $X = Y_\Delta^\an$, as in this case $\STrop(X)$ coincides with the
usual skeleton $S(Y_\Delta^\an)$, which is closed in $Y_\Delta^\an$, as we saw
in \S\ref{par:tropicalization}.

\begin{lem}\label{lem:Strop.extend.scalars}
  Let $X\subset Y_\Delta$ be a closed subscheme, let $K'\supset K$ be a valued
  extension field, and let $\pi\colon X_{K'}^\an\to X^\an$ be the structural
  map.  Then $\pi(\STrop(X_{K'})) = \STrop(X)$.
\end{lem}

\begin{proof}
  This is an immediate consequence of Proposition~\ref{prop:shilov.surj}.
\end{proof}

\begin{rem}
  In the situation of Lemma~\ref{lem:Strop.extend.scalars}, it is certainly not
  true in general that $\pi\inv(\STrop(X)) = \STrop(X_{K'})$.  Indeed, take
  $K = \C_p$ and $X = T = \bG_{m}$.  Then for $\omega\in\R\setminus\Q$ the
  non-strict affinoid domain
  $X_\omega = \sM(K\angles{\omega\inv T,\omega T\inv})$ is a single point
  (see~\cite[\S2.1]{berkovic90:analytic_geometry}), and if $K'\supset\C_p$ is an
  extension containing $\omega$ in its value group, then $X'_\omega$ is a
  ``modulus-zero'' closed annulus, where $X'=X_{K'}$.  We conclude that
  $\pi\inv(\STrop(X))\cap X'_\omega = X'_\omega$ while
  $\STrop(X')\cap X'_\omega$ contains only one point.
\end{rem}

\begin{prop}\label{prop:equidim.closed}
  Let $X\subset T$ be an equidimensional closed subscheme.  Then the tropical
  skeleton $\STrop(X)$ is closed in $X^\an$.
\end{prop}

\begin{proof}
  This is a minor variation of the argument used
  in~\cite[Theorem~10.6]{gubler_rabinoff_werner16:skeleton_tropical}, so we only
  provide a sketch.  By Lemmas~\ref{lem:Strop.extend.scalars}
  and~\ref{lem:extend.closed}, we may assume that $K$ is algebraically closed
  and that $\Gamma = \R$, i.e.\ that $\val\colon K^\times\to\R$ is surjective.
  By Lemma~\ref{lem:skeleton.reduction} we may also assume $X$ is reduced.  Let
  $d = \dim(X)$.  A generic homomorphism $\psi\colon T\to\bG_m^d$ has the
  property that the induced homomorphism $f\colon N_\R\to\R^d$ is
  finite-to-one on $\Trop(X)$. Using this and properness of the tropicalization maps, 
one deduces easily that the analytification of   the composition
  $\phi\colon X\inject T\to\bG_m^d$ is proper as a map of topological spaces. Since it also boundaryless 
\cite[Theorem~3.4.1]{berkovic90:analytic_geometry}, it is a proper morphism of analytic spaces.
We conclude that $\varphi$ is proper and hence finite as it is an affine morphism. 
 Let $\omega\in N_\R$,
  let $\omega' = f(\omega)\in\R^d$, and let
  $U'_{\omega'} = \trop\inv(\omega')\subset\bG_m^{d,\an}$.  Then
  $\psi^\an\colon X_\omega\to U'_{\omega'}$ is finite, so the map on canonical
  reductions $\td X_\omega\to\td U'_{\omega'}$ is finite as well
  by~\cite[Theorem~6.3.4/2]{bosch_guntzer_remmert84:non_archimed_analysis}.
  Since $X$ is equidimensional of dimension $d$, the same is true of $X_\omega$,
  and hence of $\td X_\omega$ by Proposition~\ref{prop:canon.red.facts}, so
  generic points of $\td X_\omega$ map to the generic point of
  $\td U'_{\omega'}$.  By functoriality of the reduction map, this implies that
  the $\phi$-inverse image of the Shilov boundary point of $U'_{\omega'}$ is the
  Shilov boundary of $X_\omega$.  Therefore
  $\phi\inv(\STrop(\bG_m^{d})) = \STrop(X)$.  Since the skeleton
  $S(\bG_m^{d,\an}) = \STrop(\bG_m^{d})$ is closed in $\bG_m^{d,\an}$ as we have
  seen in~\S\ref{par:tropicalization}, this proves that $\STrop(X)$
  is closed.
\end{proof}

In Corollary~\ref{cor:proper.int.ST.closed} we will prove a more general statement about closed subschemes in toric varietes.

\begin{prop}\label{prop:union.of.comps}
  Let $X\subset T$ be a closed subscheme and let $X = X_1\cup\cdots\cup X_n$ be
  its decomposition into irreducible components.  Then
  \[ \STrop(X) = \STrop(X_1)\djunion\cdots\djunion \STrop(X_n) \]
  as topological spaces,
  i.e., the $\STrop(X_i)$ are disjoint open and closed subsets of $\STrop(X)$.
\end{prop}

\begin{proof}
  For each $i$, $(X_i)_\omega$ is a union of irreducible components of
  $X_\omega$ by Proposition~\ref{prop:affinoid.union.comps}.  We claim that for
  $i\neq j$, $(X_i)_\omega$ and $(X_j)_\omega$ do not share any irreducible
  components.  Let $X = \Spec(A)$ and $X_\omega = \sM(\cA_\omega)$.  Then
  $A\to\cA_\omega$ is flat, so $\Spec(\cA_\omega)\to X$ is flat.  Thus the image
  of a generic point of $\Spec(\cA_\omega)$ is a generic point of $X$, 
  so we have shown that any irreducible component of
  $(X_i)_\omega$ is dense in $X_i$. This proves our claim.

  By Proposition~\ref{prop:bdy.irred.comps}, the Shilov
  boundary of $X_\omega$ is the disjoint union of the Shilov boundaries of
  $(X_1)_\omega,\ldots,(X_n)_\omega$.  Hence $\STrop(X)$ is the disjoint union
  of the $\STrop(X_i)$.  Each $\STrop(X_i)$ is closed in $X_i^{\an}$ (and hence in $X^{\an}$) by
  Proposition~\ref{prop:equidim.closed}. Hence for all $i$ the finite union
  $\Djunion_{j\neq i}\STrop(X_j)$ is closed, so the complement of $\STrop(X_i)$ in $\STrop(X)$ is also closed. This implies our claim.
\end{proof}

\begin{cor}\label{cor:closed.on.orbits}
  Let $X\subset Y_\Delta$ be a closed subscheme.  Then
  $\STrop(X)\cap O(\sigma)^\an$ is closed in $X^\an \cap O(\sigma)^\an$ for every
  $\sigma\in\Delta$.  Therefore, $\STrop(X)$ is locally closed in $X^\an$.
\end{cor}

\begin{proof}
  Since $\STrop(X)$ is defined independently on each torus orbit, we may assume
  $O(\sigma) = T$.  Now apply Propositions~\ref{prop:equidim.closed}
  and~\ref{prop:union.of.comps}.
\end{proof}

In general, $\STrop(X)$ is not closed in $X^\an$ when the ambient toric variety
$Y_\Delta$ is not a torus, even when $X$ itself is irreducible, as the following
example shows. 

\begin{eg}\label{eg:counter.eg}
  For simplicity, we let $K$ be an algebraically closed non-Archimedean field
  with value group $\R$.  Let $Y_\Delta = Y_\sigma$ be the affine toric variety
  $\bA^3$, so $\sigma = \R_+^3$.  Let $x_1,x_2,x_3$ be coordinates on $\bA^3$ and
  let $X\subset\bA^3$ be the closed subscheme defined by the equation
  $(x_1-1)x_2 + x_3 = 0$.  This is an irreducible hypersurface.

  The partial compactification $\barNRD$ in this case is
  $\bar\R{}^3 = (\R\cup\{\infty\})^3$,
  and the tropicalization map $\trop\colon\bA^{3,\an}\to\bar\R{}^3$
  takes $\xi$ to $-(\log|\xi(x_1)|,\log|\xi(x_2)|,\log|\xi(x_3)|)$.
  Let $\omega_1,\omega_2,\omega_3$ be coordinates on $\R^3$.  Then $\Trop(X)\cap\R^3$ is the
  union of the three cones
  \[\begin{split} 
    P_1 &= \{\omega \in \R^3 \mid  \omega_1 \geq 0   , \,  \omega_3 = \omega_2\} \\
    P_2 &= \{\omega \in \R^3 \mid   \omega_1= 0  , \,   \omega_3 \geq \omega_2 \} \\
    P_3 &= \{\omega \in \R^3 \mid  \omega_1 \leq 0 , \, \omega_3= \omega_1 + \omega_2 \}.
  \end{split}\]
  Let $\sigma_1 = \{\omega\in\R_+^3\mid \omega_2=\omega_3=0\}$, so
  $O(\sigma_1) = \{x_1=0,x_2x_3\neq 0\}$ and we identify $N_\R(\sigma_1)$ with
  $\{\infty\}\times\R^2$.  Then $X\cap O(\sigma_1)$ is defined by the equation
  $x_3 = x_2$ in $\{0\}\times\bG_m^2$, so
  \[ \Trop(X)\cap N_\R(\sigma_1) = \{\infty\} \times \{\omega_3=\omega_2\}. \]
  Let $\sigma_2 = \{\omega\in\R_+^3\mid \omega_1=\omega_3=0\}$, so
  $O(\sigma_2) = \{x_2=0,x_1x_3\neq 0\}$ and we identify $N_\R(\sigma_2)$ with
  $\R\times\{\infty\}\times\R$.  Then $X\cap O(\sigma_2) =\emptyset$ and
  therefore
  \[ \Trop(X)\cap N_\R(\sigma_2) = \emptyset. \]
  Let $\sigma_3 = \{\omega\in\R_+^3\mid \omega_1=\omega_2=0\}$, so
  $O(\sigma_3) = \{x_3=0,x_1x_2\neq 0\}$ and we identify $N_\R(\sigma_3)$ with
  $\R^2\times\{\infty\}$.  Then $X\cap O(\sigma_3)$ is defined by the equation
  $x_1=1$ in $\bG_m^2\times\{0\}$, so
  \[ \Trop(X)\cap N_\R(\sigma_3) = \{0\} \times \R \times\{\infty\}. \]
  Let $\sigma_{12} = \{\omega\in\R_+^3\mid \omega_3=0\}$, 
  let $\sigma_{13} = \{\omega\in\R_+^3\mid \omega_2=0\}$, and
  let $\sigma_{23} = \{\omega\in\R_+^3\mid \omega_1=0\}$.
  Then $X\cap O(\sigma_{12}) = \emptyset$ and
  $X\cap O(\sigma_{13}) = \emptyset$, so
  \[ \Trop(X)\cap N_\R(\sigma_{12}) = \emptyset \sptxt{and}
  \Trop(X)\cap N_\R(\sigma_{13}) = \emptyset. \]
  However, $X$ contains $O(\sigma_{23}) = \{x_1\neq 0, x_2=x_3=0\}$, so
  \[ \Trop(X)\cap N_\R(\sigma_{23}) = \R\times\{\infty\}\times\{\infty\}. \]
  Finally, $O(\sigma) = \{(0,0,0)\}$ is contained in $X$, so
  \[ \Trop(X)\cap N_\R(\sigma) = \{(\infty,\infty,\infty)\}. \]

  One checks that the initial form of the defining equation $(x_1-1)x_2+x_3$ at
  every $\omega\in\bar\R{}^3$
  is irreducible, and hence that each $X_\omega$ contains a unique Shilov
  boundary point $s_X(\omega)$
  by~\cite[Lemma~10.3]{gubler_rabinoff_werner16:skeleton_tropical} or
  Proposition~\ref{prop:shilov.section} below.  
  {In the following, $\bB(r)$ denotes the closed disk with center $0$ and radius $r$ in $\bA^{3,\an}$.} 
  For
  $\omega = (0,\omega_2,\infty)\in\Trop(X)\cap N_\R(\sigma_3)$, the tropical
  fiber $X_\omega$ is the modulus-zero annulus
  $\{(1,\xi,0)\mid\val(\xi)=\omega_2\}$, so the Shilov point $s_X(\omega)$ is
  the Gauss point of the {disk} $\{1\}\times\bB(\exp(-\omega_2))\times\{0\}$.  
  Taking $\omega_2\to\infty$, the Shilov points
  $s_X(\omega)$ converge to $(1,0,0)$.  However, for
  $\omega = (0,\infty,\infty)\in N_\R(\sigma_{23})$ the tropical fiber
  $X_\omega$ is all of $\{(\xi,0,0)\mid v(\xi)=0\}$, so the Shilov point
  $s_X(\omega)$ is the Gauss point of the  {disk}
  $\bB(1)\times\{0\}\times\{0\}$.
  Hence $s_X(\omega)\neq(1,0,0)$, so $\STrop(X)$ does not contain the limit
  point $(1,0,0)$.

  With a bit more work, one can show that if $\omega(r) = (0,r,r)\in P_2$ for
  $r\geq 0$ then $s_X(\omega(r))\to s_X(0,\infty,\infty)$ as
  $r\to\infty$, but that if $\omega'(r) = (0,r,2r)$ then
  $s_X(\omega'(r))\to (1,0,0)$.
\end{eg}

\section{Relevant polyhedra and $d$-maximality}
\label{sec:relevant.polyhedra}

In this section we introduce two technical notions, those of relevant and
$d$-maximal polyhedra, which will be used to prove our main results
in~\S\ref{sec:limit.points}.

\subsection{The local dimension} 
In Example~\ref{eg:counter.eg}, the ``problem'' with the ``incorrect'' limit point
$(1,0,0)\in O(\sigma_{23})$ of $\STrop(X)$ is that there exists a maximal polyhedron $P_2$ in the tropicalization of $X$ intersected with the dense torus orbit, such that the closure of $P_2$ intersects the tropicalization of $X \cap O(\sigma_{23})$ in a non-maximal polyhedron. Note that in our example, $X$ has codimension one in $\bA^3$, but $X \cap O(\sigma_{23})$ has codimension zero in $O(\sigma_{23})$ The main theorem of the next section, Theorem~\ref{thm:limit.points}, says that, in a precise sense, this dimensional incompatibility of polyhedra across torus orbits is  the only possible reason for a limit point of $\STrop(X)$ not to be contained in
$\STrop(X)$.

In what follows, we fix a rational pointed fan $\Delta$ in $N_\R$ and a closed
subscheme $X\subset Y_\Delta$.  For $\omega\in\Trop(X)$ we let
\label{notn:local.cone}
$\LC_\omega(\Trop(X))$ denote the local cone of $\omega$ in
$\Trop(X)\cap N_\R(\sigma)$, where $N_\R(\sigma)$ is the orbit containing
$\omega$.  See~\cite[\S{A.6}]{gubler13:guide_tropical}.

\begin{defn}\label{defn:local.dim}
  The \emph{local dimension of $\Trop(X)$ at $\omega\in\Trop(X)$} is defined as
  \[ d(\omega) \coloneq \dim(\LC_\omega(\Trop(X))). \]
\end{defn}

Note that $d(\omega)$ only depends on $X\cap O(\sigma)$ and
$\Trop(X)\cap N_\R(\sigma)$.  If $X\cap O(\sigma)$ is equidimensional of
dimension $d$ then $d(\omega) = d$ for all $\omega\in\Trop(X)\cap N_\R(\sigma)$
by the Bieri--Groves theorem.

\begin{lem}\label{lem:equiv.d.omega}
  Let $\sigma\in\Delta$ and let $\omega\in\Trop(X)\cap N_\R(\sigma)$.  Let
  $K'\supset K$ be a valued field extension whose value group
  $\Gamma' = \val(K'^\times)$ is non-trivial and large enough that
  $\omega\in N_{\Gamma'}(\sigma)$, and let $X' = X_{K'}$.
  The following numbers coincide:
  \[ d(\omega) = \dim(\inn_\omega(X')) = \dim(\td X'_\omega) = 
  \dim(X'_\omega) = \dim(X_\omega). \]
\end{lem}

\begin{proof}
  The final equality holds by definition of $\dim(X_\omega)$.
  As all numbers in question depend only on the torus orbit whose
  tropicalization is $N_\R(\sigma)$, we may assume $Y_\Delta = T$, and as
  $\Trop(X) = \Trop(X')$, we may assume $K=K'$ and $X = X'$.
  By~\cite[Proposition~10.15]{gubler13:guide_tropical}, the local cone of
  $\omega$ in $\Trop(X)$ is equal to the tropicalization of $\inn_\omega(X)$,
  considered as a scheme over the trivially valued field $\td K'$, so the first
  equality follows from the Bieri--Groves theorem.  The other two equalities are
  a result of Proposition~\ref{prop:canon.red.facts}(3) and Remark \ref{finite surjective morphism}. 
\end{proof}

\begin{defn}\label{def:d.maxmality}
  For $\sigma\in\Delta$,  we call a polyhedron
  $P\subset\Trop(X)\cap N_\R(\sigma)$ \emph{$d$-maximal at $\omega\in N_\R(\sigma)$}
  provided that $\omega\in P$ and $d(\omega) = \dim(P)$.
\end{defn}

Let $\omega\in\Trop(X)\cap N_\R(\sigma)$ for $\sigma\in\Delta$.  Choose a
polyhedral complex structure $\Sigma$ on $\Trop(X)\cap N_\R(\sigma)$.  Then
there is a $P\in\Sigma$ which is $d$-maximal at $\omega$, and any such $P$ is a
maximal element of $\Sigma$ with respect to inclusion.

We will be concerned with limit points of $\STrop(X)$ contained in the boundary
$Y_\Delta^\an\setminus T^\an$.  First we recall what the limit points are in $\barNRD$
of a polyhedron in $N_\R$.  The following lemma is~\cite[Lemma~3.9]{osserman_rabinoff13:non_proper}.

\begin{lem}\label{lem:closure.polyhedron}
  Let $P\subset N_\R$ be a polyhedron and let $\bar P$ be its closure in
  $\barNRD$.  Let $\sigma\in\Delta$, and recall that
  $\pi_\sigma\colon N_\R\to N_\R(\sigma)$ denotes the projection map.
  \begin{enumerate}
  \item We have $\bar P\cap N_\R(\sigma) \neq\emptyset$ if and only if the
    recession cone \label{notn:recession.cone} $\rho(P)$ of $P$ intersects the relative interior of $\sigma$.
  \item If $\bar P\cap N_\R(\sigma)\neq\emptyset$, then
    $\bar P\cap N_\R(\sigma) = \pi_\sigma(P)$.
  \end{enumerate}
\end{lem}

\begin{prop} \label{tropicalization and closure}
Let $X$ be a closed subscheme of the toric variety $Y_\Delta$. Suppose that 
$X\cap T$ is dense in $X$ and write $\Trop(X)\cap N_\R$ as union of a finite set $\Sigma$ of polyhedra in $N_\R$. Then $\Trop(X)$ is the closure of $\Trop(X)\cap N_\R$ in $\barNRD$ and for  any $\sigma \in \Delta$ we have
$$\Trop(X)\cap N_\R(\sigma) =
  \bigcup_{\rho(P)\cap\relint(\sigma)\neq\emptyset}\pi_\sigma(P),$$
where $P$ ranges over all polyhedra in $\Sigma$ with $\rho(P)\cap\relint(\sigma)\neq\emptyset$.
\end{prop}

\begin{proof}
  By~\cite[Lemma~3.1.1]{osserman_payne13:lifting_tropical}, we have that
  $\Trop(X)$ is the closure of $\Trop(X)\cap N_\R$ in $\barNRD$, and by
  Lemma~\ref{lem:closure.polyhedron}, for $P\in\Sigma$ such that
  $\rho(P)\cap\relint(\sigma)\neq\emptyset$, the closure $\bar P$ of $P$ in
  $\barNRD$ satisfies $\bar P\cap N_\R(\sigma) = \pi_\sigma(P)$.  Hence we get
  the last claim.
 \end{proof}

\begin{eg}\label{eg:counter.eg.2}
  We continue with Example~\ref{eg:counter.eg}.  It is clear that $P_i$ is
  $d$-maximal at all $\omega\in P_i$ for $i=1,2,3$.  Consider now the orbit
  $O(\sigma_1) = \{x_1=0, x_2x_3\neq0\}$. By Lemma~\ref{lem:closure.polyhedron}
  we have
  $\bar P_2\cap N_\R(\sigma_1) = \bar P_3\cap N_\R(\sigma_1) = \emptyset$, and
  \[ \bar P_1\cap N_\R(\sigma_1) = \Trop(X)\cap N_\R(\sigma_1) =
  \{\infty\}\times\{\omega_3=\omega_2\}. \]
  Therefore $\bar P_1\cap N_\R(\sigma_1)$ is $d$-maximal at all of its points.
  Similarly,
  $\bar P_1\cap N_\R(\sigma_3) = \bar P_3\cap N_\R(\sigma_3) = \emptyset$ and
  \[ \bar P_2\cap N_\R(\sigma_3) = \Trop(X)\cap N_\R(\sigma_3) =
  \{0\}\times\R\times\{\infty\}, \]
  which is $d$-maximal at all of its points.  Clearly
  \[ \Trop(X)\cap N_\R(\sigma) = N_\R(\sigma) = \{(\infty,\infty,\infty)\} \] 
  is $d$-maximal at its only point.  We have $X\cap O(\sigma_2) = X\cap
  O(\sigma_{12}) = X\cap O(\sigma_{13}) = \emptyset$.

  Now consider
  $\Trop(X)\cap N_\R(\sigma_{23}) = \R\times\{\infty\}\times\{\infty\}$, which
  we identify with $\R$.  We have
  \[\begin{split}
    \pi_{\sigma_{23}}(P_1) &= \bar P_1\cap N_\R(\sigma_{23})
    = [0,\infty) \\
    \pi_{\sigma_{23}}(P_2) &= \bar P_2\cap N_\R(\sigma_{23})
    = \{0\} \\
    \pi_{\sigma_{23}}(P_3) &= \bar P_3\cap N_\R(\sigma_{23})
    = (-\infty,0]. \\
  \end{split}\]
  Let $Q_i = \pi_{\sigma_{23}}(P_i) = \bar P_i\cap N_\R(\sigma_{23})$ for
  $i=1,2,3$.  Then $Q_1$ and $Q_3$ are $d$-maximal at all of their points, but
  $Q_2$ is not $d$-maximal at its point.  We will see in
  Theorem~\ref{thm:limit.points} that this is related to the fact in
  Example~\ref{eg:counter.eg} than $s_X(\omega'(r))$ does not approach a
  point of $\STrop(X)$ as $r\to\infty$ (note that $\omega'(r)\in\relint(P_2)$
  for $r>0$).
\end{eg}

\subsection{Relevant polyhedra for Shilov boundary points}

Let us consider a point $\xi \in \STrop(X)$ with tropicalization
$\omega \in \Trop(X) \cap N_\R(\sigma)$.  We will see that not every polyhedron
$P \subset \Trop(X) \cap N_\R(\sigma)$ is ``relevant'' for $\xi$ for the
purposes of checking limit points.

First, we assume that $\omega \in N_\Gamma(\sigma)$ and that the valuation is non-trivial. Then $X_\omega$ is a strictly affinoid domain of dimension $d(\omega)$ (see Lemma~\ref{lem:equiv.d.omega}) and we have a finite  surjective morphism $\iota:\td X_\omega\to\inn_\omega(X)$ of $d(\omega)$-dimensional affine schemes of finite type over ${\td K}$ (see~\ref{finite surjective morphism}). 
It is a fact of tropical geometry~\cite[Proposition~10.15]{gubler13:guide_tropical} that the local cone of the tropical variety at $\omega$ decomposes as 
\begin{equation} \label{localization of tropical variety} 
\LC_\omega(\Trop(X)) = \Trop(\inn_\omega(X))=     \bigcup_Z \Trop(Z),
\end{equation}
where $Z$ ranges over the irreducible components of $\inn_\omega(X)$ and where
the tropical varieties $\Trop(\inn_\omega(X))$ and $\Trop(Z)$ are taken with
respect to the trivial valuation on the residue field.

\begin{defn} \label{relevant definiton} 
  ~
  \begin{enumerate}
  \item Assume that the valuation on $K$ is non-trivial and that $\omega \in
    N_\Gamma(\sigma)$.
    Let $\xi$ be a Shilov boundary point of $X_\omega$. Then the reduction
    $\red(\xi)$ is a generic point of the canonical reduction $\td X_\omega$ and
    hence $\iota(\red(\xi))$ is the generic point of an irreducible component
    $Z$ of $\inn_\omega(X)$.  Put $\trop(\xi)=\omega \in N_\Gamma$.  We say that
    a polyhedron $P \subset \Trop(X)\cap N_\R(\sigma)$ is
    {\it relevant for $\xi$} if $\omega \in P$ and $P \subset \Trop(Z)+ \omega$. 
    Equivalently, $P$ has a non-empty local cone at $\omega$ contained in
    $\Trop(Z)$ (see Lemma~\ref{irreducible component and polyhedral complex}). 

  \item For general $\omega\in N_\R(\sigma)$, we choose a valued extension field $K' \supset K$
    whose value group $\Gamma'$ is non-trivial and large enough that
    $\omega \in N_{\Gamma'}(\sigma)$. Let $X'=X_{K'}$ and let $\pi:X' \to X$ be the
    canonical morphism. Then it follows from Proposition~\ref{prop:shilov.surj}
    that $\pi(B(X_\omega')) = B(X_\omega)$. We conclude that there is
    $\xi' \in B(X_\omega') $ with $\pi(\xi')=\xi$.  We say that the polyhedron
    $P\subset \Trop(X)\cap N_\R(\sigma)$ is {\it relevant for $\xi$} if there is
    such a $\xi'$ with $P$ relevant for $\xi'$ in the sense of (1).
  \end{enumerate}
\end{defn}

\begin{lem} \label{relevant is well-defined}
The above definition of $\xi$-relevance for a polyhedron $P \subset \Trop(X) \cap N_\R(\sigma)$     does not depend on the choice of $K'$.
\end{lem}

\begin{proof}
Suppose that $P$ is relevant for $\xi$ with respect to $\xi'\in B(X_\omega')$
as in Definition~\ref{relevant definiton}(2).
We choose another  valued extension field $K'' \supset K$ with value group
$\Gamma''$  non-trivial and with $\omega \in N_{\Gamma''}(\sigma)$. Setting $X''\coloneq
X_{K''}$, we have to show that there is a $\xi''\in B(X_\omega'')$ over $\xi$
such that $P$ is relevant for $\xi''$ in the sense of Definition~\ref{relevant
  definiton}(1). By~\cite[(0.3.2)]{ducros09:excellent}, there is a complete
valued extension field $K'''$ of $K'$ and $K''$ simultaneously; we let $X'''=X_{K'''}$. It follows from Proposition~\ref{prop:shilov.surj} that the Shilov boundary of $X_\omega'''$ maps onto $B(X_\omega')$ and also onto $B(X_\omega'')$. We choose a preimage $\xi''' \in B(X_\omega''')$ of $\xi'$ and let $Z'''$ be the irreducible component of $\inn_\omega(X''')$ with generic point $\iota(\red(\xi'''))$. It follows from functoriality of the reduction and the initial degenerations that the canonical map 
$$\inn_\omega(X''')= \inn_\omega(X') \otimes_{{\td K}'} {\td K}''' \to \inn_\omega(X')$$ 
maps $Z'''$ onto $Z'$, the closure of $\iota(\xi')$. 
As $Z'''\to Z'$ is surjective, we have $\Trop(Z''') = \Trop(Z')$. By assumption, $P$ is contained in $\Trop(Z')$. 
We conclude that $P$ is relevant for $\xi'''$. Let $\xi'' \in B(X_\omega'')$ be
the image of $\xi'''$ under the analytification of the canonical map $X''' \to
X''$. A similar argument as above shows that $Z'''$ surjects onto an irreducible component $Z''$ of  $\inn_\omega(X'')$ and that $\red(\xi'')$ maps to the generic point of $Z''$. Since we have $\Trop(Z'')=\Trop(Z''')$, we conclude that $P$ is relevant for $\xi''$. Since $\xi''$ is lying over $\xi$, this proves the claim. 
\end{proof}

\begin{lem} \label{irreducible component and polyhedral complex}
Let $U$ be a closed $d$-dimensional subscheme of the torus $T$ and let $\Sigma$ be a polyhedral complex with support equal to $\Trop(U)$. If $Y$ is a $d$-dimensional irreducible component of $U$, then there is a subcomplex of $\Sigma$ with support equal to $\Trop(Y)$.  
\end{lem}

\begin{proof}
By the Bieri--Groves theorem, $\Trop(Y)$ is a finite union of $d$-dimensional polyhedra and hence $\Trop(Y)$ is covered by the $d$-dimensional polyhedra in $\Sigma$. It is enough to show that any $d$-dimensional $P \in \Sigma$ with $\relint(P) \cap \Trop(Y) \neq \emptyset$ is necessarily contained in $\Trop(Y)$. If not, then 
$P \cap \Trop(Y)$ has a boundary point $\omega$ with respect to  $P$. We may assume that $\omega$ is in $\relint(P)$ and also in the relative interior of a $(d-1)$-dimensional face $\tau$ of a $d$-dimensional polyhedron $Q \subset P \cap \Trop(Y)$. Let $H$ be a supporting hyperplane of the face $\tau$ of $Q$. By construction, there is a neighbourhood $\Omega$ of $\omega$ in $N_\R$ such that $\Omega \cap P \cap \Trop(Y)= \Omega \cap Q$  is contained in a half space $H^+$ bounded by $H$.

Bieri--Groves~\cite[Theorem~D]{bieri_groves84:geometry_set_character_induced_valuation} proved that $\Trop(Y)$ is totally concave in $\omega$ which means that there is a  polyhedron $R \subset \Trop(Y)$ through $\omega$ and intersecting the complement of $H^+$. This can also be deduced from the balancing condition. Since  $\Omega \cap P \cap \Trop(Y) \subset H^+$, we conclude that $R \cap P \subset H^+$. Using that $R$ expands from $\omega$ into the complement of $H^+$ and that $\Trop(Y) \subset \Trop(U)$, we deduce that there is a  face $S \in \Sigma$ through $\omega$ with $S \neq P$. Since $\Sigma$ is a polyhedral complex, this contradicts $\omega \in \relint(P)$. This proves $P \subset \Trop(Y)$.
\end{proof}

\begin{prop} \label{properties of relevant}
Let $\xi \in \STrop(X)$ with tropicalization $\omega \in N_\R(\sigma)$. Let $K' \supset K$ be a  valued extension field with non-trivial value group $\Gamma'$ and with $\omega \in N_{\Gamma'}$. 
\begin{enumerate}
\item  A polyhedron $P \subset \Trop(X) \cap N_\R(\sigma)$ is relevant for $\xi$ with respect to the closed subscheme $X$ of the toric variety $Y_\Delta$ if and only if $P$ is relevant for $\xi$ with respect to the affine closed subscheme $X \cap O(\sigma)$ of the orbit $O(\sigma)$. 
\item We set $X' \coloneq X_{K'}$. The local cone at $\omega$ of the union of all $\xi$-relevant polyhedra is the tropical variety of an  irreducible component $Z$ of $\inn_\omega(X')$ with respect to the trivial valuation and with $\dim(Z)=\dim_\xi(X)$. 
\item Let $\Sigma$ be a polyhedral complex in $N_\R(\sigma)$ with support equal to $\Trop(X) \cap N_\R(\sigma)$. If $\dim_\xi(X^\an)=d(\omega)$, then in (2) it is enough to consider
  $\xi$-relevant $d(\omega)$-dimensional polyhedra in $\Sigma$. Moreover, the above
  local cone  and $Z$ are both of dimension $d(\omega)$. In particular, there is a  polyhedron in $\Sigma$ which is $d$-maximal at $\omega$ and which is relevant for $\xi$.
\item If $m_{\Trop}(\omega)=1$, then every polyhedron in $\Trop(X) \cap N_\R(\sigma)$ containing $\omega$ is relevant for $\xi$. 
\end{enumerate}
\end{prop}

\begin{proof} 
Property  (1) is obvious from the definition. 
By Lemma~\ref{lem:Strop.extend.scalars}, there is $\xi' \in \STrop(X')$ over $\xi$ for the given valued extension field $K' \supset K$. Then the canonical reduction $\red(\xi') \in \td X_\omega'$ maps to the generic point of an irreducible component $Z$ of $\inn_\omega(X')$. By Proposition~\ref{prop:union.of.comps}, $\xi$ is contained in a unique irreducible component and hence $X'$ is equidimensional in a neighbourhood of  $\xi'$ of dimension $\dim_\xi(X)$. By Proposition~\ref{prop:canon.red.facts} and Proposition~\ref{prop:two.reductions}, we deduce that $\dim(Z)=\dim_\xi(X)$. 
It follows from~\eqref{localization of tropical variety} that 
\begin{equation} \label{localization of tropical variety2} 
 \Trop(Z) = \bigcup_{ P \subset \Trop(Z) + \omega} \LC_\omega(P),
\end{equation}
 where $P$ ranges over all polyhedra contained in $\Trop(X)$ with  $ \omega \in P \subset \Trop(Z)+ \omega$.
  This proves (2). 

 If $\dim_\xi(X^\an)=d(\omega)$, 
then~(2) shows that $Z$ is $d(\omega)$-dimensional and the same is true for $\inn_\omega(X')$ by Lemma~\ref{lem:equiv.d.omega}. 
We apply Lemma~\ref{irreducible component and polyhedral complex} to the fan $\{\LC_\omega(P)\mid P \in \Sigma\}$ which has support $\Trop(\inn_\omega(X'))$ by~\eqref{localization of tropical variety}. This shows that $\Trop(Z)$ is the union of the $d(\omega)$-dimensional $(\LC_\omega(P))_{P \in \Sigma}$ contained in $\Trop(Z)$, which proves~(3).

Let $P$ be a polyhedron in $\Trop(X)$ containing $\omega$. Then $\LC_\omega(P)$ is non-empty and contained in $\Trop(\inn_\omega(X))=\LC_\omega(\Trop(X))$ by  \eqref{localization of tropical variety}.     
If $m_{\Trop}(\omega)=1$, then $Z\coloneq\inn_\omega(X)$ is irreducible, and~(4) follows.
\end{proof}

In the following important lemma, $X$ is any closed subscheme of the
multiplicative torus $T$ over $K$ with character lattice $N$.  Recall that
$s(\omega)$ denotes the point of $S(T^\an) = \STrop(T)$ lying above
$\omega\in N_\R$, that $X_\omega = \trop\inv(\omega)\cap X^\an$, and that
$B(X_\omega)$ denotes the Shilov boundary of the affinoid space $X_\omega$.

\begin{lem}\label{lem:main.lemma}
  Choose a polyhedral
  complex structure $\Sigma$ on $\Trop(X)$.  Let $\omega\in\Trop(X)$ and let
  $d \coloneq d(\omega)$.  We consider a homomorphism
  $\psi\colon T\to\bG_m^d$, let $f\colon N_\R\to\R^d$ be the induced linear map, let
  $\phi = \psi|_X$, and let $\omega' = f(\omega)$.  Then  $\phi\inv(s(\omega'))\cap X_\omega$ is equal to the set of Shilov boundary points $\xi \in B(X_\omega)$ for which there is a $P \in \Sigma$ satisfying the following three conditions:
	\begin{enumerate}
	\item $P$ is $d$-maximal at $\omega$;
	\item $P$ is relevant for $\xi$;
	\item $f$ is injective on $P$.
	\end{enumerate}
For such a $\xi$, we have always $\dim_\xi(X_\omega) = d$.
\end{lem}

\begin{proof}
  Assume to begin that the valuation on $K$ is non-trivial and that
  $\omega\in N_\Gamma$.  Let
  $U'_{\omega'} = \trop\inv(\omega')\subset\bG_m^{d,\an}$.  We have
  $\phi(X_\omega)\subset U'_{\omega'}$ due to the commutativity of the square
  \[\xymatrix @=.2in{
    {X^\an} \ar[r]^{\phi} \ar[d]_{\trop} & {\bG_m^{d,\an}} \ar[d]^{\trop} \\
    {\Trop(X)} \ar[r]^(.6){f} & {\R^d} }\]
  Consider the morphism $\phi\colon X_\omega\to U'_{\omega'}$ and the induced
  morphism on canonical reductions
  $\td\phi\colon\td X_\omega\to\td U'_{\omega'}\cong\bG_{m,\td K}^d$.  If
  $\xi\in X_\omega$ maps to $s(\omega')$ then by functoriality of the reduction
  map, we have that $\red(\xi)\in\td X_\omega$ maps to the generic point of
  $\bG_{m,\td K}^d$.  Since $d(\omega) = \dim(\td X_\omega) = d$ by
  Lemma~\ref{lem:equiv.d.omega}, it follows that $\red(\xi)$ is a generic point
  of $\td X_\omega$, so $\xi$ is a Shilov boundary point of $X_\omega$ by
  Proposition~\ref{prop:can.red.surj}.  This proves that
  $\phi\inv(s(\omega'))\cap X_\omega$ is contained in the Shilov boundary
  $B(X_\omega)$.  Since $\red(\xi)$ is the generic point of an irreducible
  component $Z'$ of $\td X_\omega$ of dimension $d$, the irreducible component  of
  $X_\omega$ containing $\xi$ must also have dimension $d$.  Hence $\dim_\xi(X_\omega) = d$.

Next, we show that $\xi$ satisfies properties (1)--(3). Note that $\td\phi\colon\td X_\omega\to\td U'_{\omega'}\cong\bG_{m,\td K}^d$ has a  canonical factorization 
\begin{equation} \label{canonical factorization}
\td X_\omega \To \inn_\omega(X)  \xrightarrow{\inn_\omega(\phi)}
  \inn_{\omega'}(\bG_{m,K}^d) \cong \bG_{m,\td K}^d.
	\end{equation}
We have seen in  Remark \ref{finite surjective morphism} that the first map is surjective and finite. We conclude that $Z'$ maps onto a $d$-dimensional irreducible component $Z$ of $\inn_\omega(X)$. Since $\red(\xi)$ maps to the generic point of $\bG_{m,\td K}^d$,  the restriction of $\inn_\omega(\phi)$ to $Z$ is dominant. 
By functoriality of the tropicalization, we get a commutative square 
\begin{equation} \label{commutative square}
\xymatrix @R=.2in{
  {Z^\an} \ar[r]^{\inn_\omega(\phi)} \ar[d]_{\trop} &
  {\bG_{m,\td K}^{d,\an}} \ar[d]^{\trop} \\
  {\Trop(Z)} \ar[r]_(.6)f & {\R^d}
}
\end{equation}
using that the tropicalization of $\inn_\omega(\phi)$ agrees with the restriction of $f$. The Bieri--Groves theorem shows that $\Trop(Z)$ is a $d$-dimensional polyhedral fan. Since the tropicalization maps are surjective and since $\inn_\omega(\phi)(Z)$ is dense in $\bG_{m,\td K}^d$, commutativity of the diagram~\eqref{commutative square} shows that $f(\Trop(Z))=\R^d$. By Proposition~\ref{properties of relevant} (3), there is a $d$-dimensional polyhedron $P\subset\Trop(Z)$ in $\Sigma$ containing $\omega$ such that $f(P)$ is $d$-dimensional. In other words, $P$ is $d$-maximal at $\omega$ and relevant for $\xi$. Moreover, $f|_P$ is injective. This proves~(1)--(3).

Conversely, we assume that $\xi \in B(X_\omega)$ has a polyhedron $P \in \Sigma$ satisfying (1)--(3). 
	We must
  show that  $\td\phi(\red(\xi))$ is the generic point of
  $\bG_{m,\td K}^d$.  
  The first map in \eqref{canonical factorization} is finite and surjective, hence it maps $Z'$ onto an irreducible component $Z$ of $\inn_\omega(X)$. Since $P$ is relevant for $\xi$ by (2), we conclude that $P \subset \Trop(Z) + \omega$. Property (1) says that $P$ is $d$-maximal at $\omega$ which means that $\dim(P)=d(\omega)=d$. It follows from Lemma \ref{lem:equiv.d.omega} and the Bieri--Groves theorem that $\dim(Z)=d$. Using again the commutative diagram 
\eqref{commutative square}, we deduce from property (3) that the tropicalization of $\inn_\omega(\phi)(Z)$ contains a $d$-dimensional polyhedron. This is only possible if the generic point of $Z$ maps to the generic point of $\bG_{m,\td K}^d$. Using the factorization \eqref{canonical factorization} of $\td \varphi$ and that the first map takes $\red(\xi)$ to the generic point of $Z$, we deduce that $\td \varphi$ maps $\red(\xi)$
to the generic point of $\bG_{m,\td K}^d$. This proves $\xi \in \phi\inv(s(\omega'))\cap X_\omega$ and hence we have shown the Lemma in case of  a non-trivial valuation on $K$ with
  $\omega\in N_\Gamma$. 
	
To deal with the general case, we choose a  valued extension field 
  $K'\supset K$  whose value group
  $\Gamma'=\val(K'^\times)$ is non-trivial and large enough that
  $\omega\in N_{\Gamma'}$, and let $X' = X_{K'}$.  Let
  $\xi\in\phi\inv(s(\omega'))\cap X_\omega$.  We have a commutative
  square
  \[\xymatrix @R=.2in @C=.5in{
    {X'_\omega} \ar[r]^(.4){\phi_{K'}} \ar[d]_\pi &
    {\bG_{m,K'}^{d,\an}} \ar[d]^\pi \\
    {X_\omega}
    \ar[r]_(.4){\phi} & {\bG_{m,K}^{d,\an}} }\]
  where the vertical arrows are the structural morphisms.
  By~\cite[Theorem 1.2.1]{berkovic90:analytic_geometry} and \cite[(0.3.2)]{ducros09:excellent}, the spectrum 
  $F = \sM\big(\sH(\xi)\hat\tensor_{\sH(s(\omega'))}\sH(s'(\omega'))\big)$ of the Banach ring $\sH(\xi)\hat\tensor_{\sH(s(\omega'))}\sH(s'(\omega'))$ is
  nonempty, where $s'\colon\R^d\to\bG_{m,K'}^{d,\an}$ is the section of
  tropicalization.  If $\xi'$ is in the image of the natural morphism
  $F\to X_\omega'$, then $\phi_{K'}(\xi') = s'(\omega')$ and $\pi(\xi') = \xi$.
  By the case handled above, $\xi'\in B(X'_\omega)$ and there is $P\in \Sigma$ satisfying (1)--(3). Then $\xi\in B(X_\omega)$ by
  Proposition~\ref{prop:shilov.surj} and $P$ is also relevant for $\xi$ by definition.  Moreover, one always
  has $\dim_{\xi'}(X'_\omega) \leq \dim_\xi(X_\omega)$, so
  $\dim_{\xi'}(X'_\omega) = d(\omega) = \dim(X_\omega)$ implies
  $\dim_\xi(X_\omega) = d(\omega)$.  This proves one inclusion of the Lemma in the general case and the last claim.
	
	Conversely, let $\xi\in B(X_\omega)$ with $P \in \Sigma$ satisfying (1)--(3).  Then there
  exists $\xi'\in B(X_\omega')$ mapping to $\xi$ by
  Lemma \ref{relevant is well-defined} 
  such that $P$ is also relevant for $\xi'$. By the special case considered above, we have
  $\phi_{K'}(\xi') = s'(\omega')$, so
  $\phi(\xi) = \pi(s'(\omega')) = s(\omega')$. This proves $\xi \in \phi\inv(s(\omega'))\cap X_\omega$.
\end{proof}

\section{Limit points of the tropical skeleton}
\label{sec:limit.points}

We have seen in Example~\ref{eg:counter.eg}  that the tropical skeleton $\STrop(X)$ is not necessarily closed in the toric variety {$Y_\Delta^\an$}. 
In this section, we give conditions under which a limit of a sequence of
points of $\STrop(X)$ is contained in $\STrop(X)$.
 Our goal is to study the accumulation points 
of $\STrop(X)$ for the closed subscheme $X$ of the toric variety $Y_\Delta$. 
By~\cite[Th\'eor\`eme~5.3]{poineau13:angeliques}, every limit point of $\STrop(X)$ is the
limit of a sequence of points.%
\footnote{This is mainly for convenience, as one could work with nets.}
We use the notation from~\S\ref{subsection:toric_varieties}.

\begin{thm}\label{thm:limit.points}
Let $\Delta$ be a rational pointed fan in $N_\R$ and let $X\subset Y_\Delta$
  be a closed subscheme.  By $T$ we denote the dense torus in $Y_\Delta$. We assume that $X \cap T$ is equidimensional of dimension $d$. Let
  $(\xi_i)_{i \in \N}$ be a sequence in $\STrop(X)$ which is contained in $(X\cap T)^{\an}$ such that $\xi_i$ converges to a point $\xi\in (X \cap O(\sigma))^\an$ for some torus orbit $O(\sigma) \subset Y_\Delta$. Put $\omega_i = \trop(\xi_i) \in N_\R$ and $\omega = \trop (\xi) \in N_\R(\sigma)$. 
  
 If there exists a polyhedron $P \subset \Trop(X) \cap N_\R$ 
 of dimension $d$ containing all $\omega_i$ 
and relevant for all $\xi_i$ such that $\pi_\sigma(P)$ is d-maximal at $\omega$, we have 
  $\xi\in\STrop(X)$ and $\dim_\xi(X_\omega)=d(\omega)$. 
\end{thm}

\begin{rem}
  The hypotheses in Theorem~\ref{thm:limit.points} are sufficient but not
  necessary.  For example, let the notation be as in
  Examples~\ref{eg:counter.eg} and~\ref{eg:counter.eg.2}, and for $r \geq 1$ let
  $\omega''(r) = (0, r, r + r\inv)\in\relint(P_2)$.  One can show that
  $s_X(\omega''(r))\to s_X(0,\infty,\infty)$ as $r\to\infty$ even though
  $\omega''(r)$ is contained in $P_2\setminus(P_1\cup P_3)$ and
  $\pi_{\sigma_{23}}(P_2)$ is not $d$-maximal at $(0,\infty,\infty)$.
\end{rem}

At a basic level, the proof of Theorem~\ref{thm:limit.points} uses a similar
idea to~\cite[Theorem~10.6]{gubler_rabinoff_werner16:skeleton_tropical}
(and Proposition~\ref{prop:equidim.closed}), in that we compare $\STrop(X)$ with
the inverse image of the skeleton of a smaller-dimensional toric
variety.  However, one must be much more careful in constructing the smaller
toric variety.

\begin{proof}[Proof of Theorem~\ref{thm:limit.points}]
If $O(\sigma) = T$, our claim follows from Corollary \ref{cor:closed.on.orbits}. Hence we may assume that $\sigma \neq 0$. 

  Since $\omega = \lim\omega_i\in N_\R(\sigma)$, we have
  $\bar P\cap N_\R(\sigma)\neq\emptyset$, where $\bar P$ is the
  closure of $P$ in $\barNRD$.  Hence by Lemma~\ref{lem:closure.polyhedron},
  the recession cone of $P$ intersects the relative interior of $\sigma$, and
  $\bar P\cap N_\R(\sigma) = \pi_\sigma(P)$.  This set contains $\omega$.
  Let $H$ be a rational supporting hyperplane of the face $\{0\}\prec\sigma$:
  that is, $H$ is rational and $H\cap\sigma = \{0\}$, so $\sigma$ is contained in one
  of the half-spaces bounded by $H$.  Let $\angles\sigma\subset N_\R$ be the
  linear span of $\sigma$, and let $\angles{P}_0\subset N_\R$ be the linear
  (as opposed to the affine) span of $P$, i.e., $\angles{P}_0$ is the linear
  span of $P-\eta$ for any $\eta\in P$.  The subspaces $\angles\sigma$ and
  $\angles{P}_0$ are rational.
  Note that $H+(\angles\sigma\cap\angles{P}_0) = N_\R$ since the recession
  cone of $P$ intersects $\relint(\sigma)$.

  We claim that there exists a rational subspace $L\subset N_\R$ such that
  \[ \text{(a) } L\subset H \qquad \text{(b) } N_\R = L\dsum\angles{P}_0
  \qquad \text{(c) } N_\R(\sigma) =
  \pi_\sigma(L)\dsum\pi_\sigma(\angles{P}_0). \]
  This is pure linear algebra.
  Let $n = \dim(N_\R)$,  recall $d = \dim(P)$ and let $d' = d(\omega)$.  By hypothesis,
  $d' = \dim(\pi_\sigma(P))$, and hence
  $\dim(\angles{\sigma}\cap\angles{P}_0) = d-d'$.  Since
  $\angles{\sigma}\not\subset H$, we have that $H$ surjects onto $N_\R(\sigma)$.
  Choose a rational subspace $V\subset H$ mapping isomorphically onto
  $N_\R(\sigma)$, so $H = (\angles{\sigma}\cap H)\dsum V$ and
  $N_\R = \angles{\sigma}\dsum V$.  As
  $H+(\angles{\sigma}\cap\angles{P}_0)=N_\R$, we have
  $(\angles{\sigma}\cap H)+(\angles{\sigma}\cap\angles{P}_0) = \angles{\sigma}$,
  so $\dim(\angles{\sigma}\cap H\cap\angles{P}_0) = d-d'-1$, and therefore,
  \[\begin{split}
    \dim(\angles{\sigma}\cap H) - \dim(\angles{\sigma}\cap H\cap\angles{P}_0)
    &= (\dim(\angles{\sigma})-1) - (d-d'-1) \\
    &= \dim(\angles{\sigma}) - (d-d') \\
    &= \dim(\angles{\sigma}) - \dim(\angles{\sigma}\cap\angles{P}_0).
  \end{split}\]
  We conclude that 
  \[\begin{split}
    \codim(\angles{\sigma}\cap H\cap\angles{P}_0,\, \angles{\sigma}\cap H) &=
    \codim(\angles{\sigma}\cap\angles{P}_0,\, \angles{\sigma}) = \dim(\angles{\sigma}) - (d-d'), \\
    \codim(\pi_\sigma(\angles{P}_0),\,N_\R(\sigma)) &= n - \dim(\angles{\sigma}) - d'.
  \end{split}\]
  It is possible to choose (generic)
  rational subspaces $L_1\subset\angles{\sigma}\cap H$ of dimension $\dim(\angles{\sigma})-(d-d')$ and
  $L_2\subset V$ of dimension $n-\dim(\angles{\sigma})-d'$ such that
  $L_1\cap(\angles{\sigma}\cap\angles{P}_0) = \{0\}$ and
  $\pi_\sigma(L_2)\cap\pi_\sigma(\angles{P}_0) = \{0\}$.  The subspace $L = L_1\dsum L_2$
  satisfies our requirements~(a)--(c).
  
  \newcommand*\barNpRs{\bar N\p^{\sigma'}_\R}

  To prove the theorem, we may replace $Y_\Delta$ by $Y_\sigma$ and $X$ by $X \cap Y_\sigma$ without loss of
  generality.  Let $N' = N/(N\cap L)$ and let
  $f\colon N\to N'$ be the quotient homomorphism.  Let
  $M' = \Hom(N',\Z)$ and let $T' = \Spec(K[M'])\cong\bG_m^d$ be the
  torus with cocharacter lattice $N'$.  The map $f$ induces a homomorphism
  $\psi\colon T\to T'$.  Let $\sigma' = f(\sigma)$.  This is a pointed
  cone in $N'_\R$ since the supporting hyperplane $H$ contains $L$.  Thus
  $\psi\colon T\to T'$ extends to a toric morphism of affine toric varieties
  $\psi\colon Y_\sigma\to Y'_{\sigma'}$, $f\colon N_\R\to N'_\R$
  extends to a continuous map $f\colon\barNRs\to\barNpRs$, and the following
  squares commute:
  \[\xymatrix @R=.2in{
    {T} \ar[r]^\trop \ar[d]_{\psi} & {N_\R} \ar[d]^{f} & &
    {Y_\sigma} \ar[r]^\trop \ar[d]_{\psi} & {\barNRs} \ar[d]^{f} \\
    {T'} \ar[r]_(.45)\trop & {N'_\R} & & {Y'_{\sigma'}} \ar[r]_(.45)\trop &
    {\barNpRs} }\]
  Note that
  $N'_\R(\sigma') = N_\R/(\angles\sigma + L) = N_\R(\sigma)/\pi_\sigma(L)$. Condition~(b)
  for $L$ implies $f$ is injective on $P$ with $\dim(P) = \dim(N_\R')$, and
  condition~(c) implies $f\colon N_\R(\sigma)\to N'_\R(\sigma')$ is injective on
  $\pi_\sigma(P)$ with $d' = \dim(\pi_\sigma(P)) = \dim(N_\R'(\sigma'))$.

  Let $\phi = \psi|_X$,
  and define $S = \phi\inv\big(S(Y'^\an_{\sigma'})\big)$, the inverse image of
  the skeleton of the affine toric variety $Y'_{\sigma'}$.  This is a closed
  subset of $X^\an$ because $S(Y'^\an_{\sigma'})$ is closed in
  $Y'^\an_{\sigma'}$, as we saw in~\S\ref{par:tropicalization}.  
  Lemma~\ref{lem:main.lemma} implies that
  $\xi_i\in S \cap X_{\omega_i}$ for all $i$. Hence the limit point $\xi$ lies in $ S\cap X_\omega$.  By construction,
  $d' = d(\omega) = \dim(N'_\R(\sigma'))$, which is equal to the dimension of
  the torus orbit $O'(\sigma')\subset Y'_{\sigma'}$.
 
  Again by Lemma~\ref{lem:main.lemma}, this time applied to $T = O(\sigma)$,
  $X = X\cap O(\sigma)$, and the map $\psi\colon O(\sigma)\to O'(\sigma')$ on
  torus orbits, we see that $S\cap X_\omega$ is contained in the Shilov boundary
  of $X_\omega$, and therefore that $\xi\in\STrop(X)$.  Moreover, the local
  dimension of $X_\omega$ at any point of $S\cap X_\omega$ is $d(\omega)$ by
  Lemma~\ref{lem:main.lemma}.
\end{proof}

If $X$ intersects all torus orbits equidimensionally, we can deduce the following result. 

\begin{thm} \label{thm:skeleton_equidimensional}
Let $X$ be a closed subscheme of the toric variety $Y_\Delta$ such that $X \cap O(\sigma)$ 
is equidimensional of dimension $d_\sigma$ for any $\sigma \in \Delta$.
We suppose that for all faces $\tau \prec \sigma$ of $\Delta$ there exists a finite polyhedral complex structure $\Sigma$ on $\Trop(X) \cap N_\R(\tau)$
 with the following property: For every $d_\tau$-dimensional 
polyhedron $P$ in $\Sigma$ such that its recession cone intersects the relative interior of $\pi_\tau(\sigma)$ in $N_\R(\tau)$, put $\overline{\sigma} = \pi_\tau (\sigma)$ and assume that the canonical projection $\pi_{\overline{\sigma}}(P)$ has dimension $d_\sigma$ in $N_\R(\sigma)$.  Then $\STrop(X)$ is closed.
\end{thm}

\begin{proof}
Let  $(\xi_i)_{i \in \N}$ be a sequence in $\STrop(X)$ converging to
 some point $\xi\in X^\an$. We have to show that $\xi \in \STrop(X)$.  We may always pass to a subsequence and so we may assume 
	that the sequence $\{\xi_i\}$ is contained in
 a single torus orbit $O(\tau)^\an$. 
	Let $\sigma\in\Delta$ be the cone such that
  $\xi\in O(\sigma)^\an$. Then $\tau \prec \sigma$. 
 Let
 $\omega_i = \trop(\xi_i)\in N_\R(\tau)$ and let
 $\omega = \trop(\xi)\in N_\R(\sigma)$, so $\omega = \lim\omega_i$ by
 continuity of $\trop$. 
 
Since $X \cap O(\tau)$ is equidimensional of dimension $d_\tau$, we have $\dim_{\xi_i}(X)= d_\tau = d(\omega_i)$ for every $i$.  
By Proposition \ref{properties of relevant}(3), there is a polyhedron $P_i \subset \Trop(X) \cap N_\R(\tau)$ in $\Sigma$ which is $d$-maximal at $\omega_i$ and which is relevant for $\xi_i$. After passing to a subsequence, we may assume that all $P_i = P$ for 
 a single $d_\tau$-dimensional polyhedron $P$ in $\Sigma$. Since $\omega$ lies in the closure of $P$, the recession cone of $P$ meets the relative interior of $\pi_\tau(\sigma)$ by Lemma~\ref{lem:closure.polyhedron}.
Therefore the  projection of $P$ to $N_\R(\sigma)$ is $d$-maximal by assumption.  Hence our claim follows from Theorem~\ref{thm:limit.points}.
\end{proof}

\section{Proper intersection with orbits}\label{sec:proper intersection}
In this section we discuss common dimensionality conditions under which the
hypotheses of Theorem~\ref{thm:limit.points} are automatically satisfied.  We consider a closed subscheme $X$ of a toric variety $Y_\Delta$ with dense torus
$T$.  We assume throughout that $X\cap T$ is equidimensional of dimension $d$
and that $X\cap T$ is dense in $X$.

\begin{defn}
  Let $\sigma\in\Delta$.  We say that \emph{$X$ intersects $O(\sigma)$ properly}
  provided that $\dim(X\cap O(\sigma)) = \dim(X) - \dim(\sigma)$.
\end{defn}

Note that if $X$ intersects $O(\sigma)$ properly then
$X\cap O(\sigma)\neq\emptyset$ when $\dim(\sigma) \leq \dim(X)$, and
$X\cap O(\sigma) = \emptyset$ when $\dim(\sigma) > \dim(X)$.

\begin{lem}\label{lem:prop.int.codim}
If
  $\dim(\sigma) \leq \dim(X)$ and $X$ intersects $O(\sigma)$ properly, then
  $X\cap O(\sigma)$ is equidimensional and
  \[ \codim(X\cap O(\sigma),\,O(\sigma)) = \codim(X\cap T,\,T). \]
  Therefore, $\Trop(X\cap O(\sigma)) = \Trop(X)\cap N_\R(\sigma)$ has pure
  dimension $\dim(X)-\dim(\sigma)$.
\end{lem}

\begin{proof}
  This follows from~\cite[Proposition~14.7]{gubler13:guide_tropical} and the
  fact that the dimension of $\sigma$ is $\codim(O(\sigma),T)$.
  The last statement is a consequence of the Bieri--Groves theorem.
\end{proof}

\begin{rem}\label{recession cones}
For a polyhedron $P\subset N_\R$, we let $\rho(P)\subset N_\R$ denote the
recession cone of $P$. Let $U$ be a closed subscheme of the torus $T$ over $K$. 
Then the Bieri--Groves theorem shows that we may write $\Trop(U)$ as a finite union of 
integral $\Gamma$-affine polyhedra $P$ (see \cite[2.2]{gubler_rabinoff_werner16:skeleton_tropical} for the definition of integral $\Gamma$-affine polyhedra 
for a subgroup $\Gamma$ of $\R$). If $U$ is of pure dimension $d$, then we can choose all $P$ $d$-dimensional. 
Let $\Sigma$ be the collection of these polyhedra and let $\Trop_0(U)$ be the tropical variety of $X$ with respect to the trivial valuation. Then we recall from \cite[Corollary 11.13]{gubler13:guide_tropical} the non-trivial fact that 
\begin{equation} \label{trivial tropicalization}
\Trop_0(U)=\bigcup_{P \in \Sigma} \rho(P).
\end{equation}
\end{rem}

The next proposition shows that the condition that $X$
intersects $O(\sigma)$ properly can be checked on tropicalizations.

\begin{prop}\label{prop:equiv.int.proper}
  Choose a finite collection $\Sigma$ of integral $\R$-affine
  $d$-dimensional polyhedra whose union is $\Trop(X)\cap N_\R$.  Fix
  $\sigma\in\Delta$ with $\dim(\sigma)\leq d$.  Then the following are
  equivalent:
  \begin{enumerate}
  \item $X$ intersects $O(\sigma)$ properly.
  \item There exists $P\in\Sigma$ such that
    $\rho(P)\cap\relint(\sigma)\neq\emptyset$, and for all such $P$,
    \[ \dim(\rho(P)\cap\sigma) = \dim(\sigma). \]
  \item There exists $P\in\Sigma$ such that
    $\rho(P)\cap\relint(\sigma)\neq\emptyset$, and for all such $P$,
    \[ \dim(\pi_\sigma(P)) = d - \dim(\sigma). \]
 
\end{enumerate}
\end{prop}

\begin{proof}
  The equivalence of~(1) and~(2)
  is~\cite[Corollary~14.4, Remark~14.5]{gubler13:guide_tropical}, using \eqref{trivial tropicalization} and noting that
  $\dim(\rho(P)\cap\sigma) = \dim(\sigma)$ if and only if
  $\dim(\rho(P)\cap\relint(\sigma)) = \dim(\sigma)$.  For $P\in\Sigma$ such that
  $\rho(P)\cap\relint(\sigma)\neq\emptyset$, condition~(3) is equivalent to
  having $\angles\sigma\subset\angles{P}_0$, where $\angles\sigma$ is the span
  of $\sigma$ and $\angles{P}_0$ is the linear span of $P$, as in the proof of
  Theorem~\ref{thm:limit.points}.  As $\rho(P)\subset\angles{P}_0$, it is clear
  that~(2) implies~(3).  For~(3) implies~(1), we have 
  $$\Trop(X)\cap N_\R(\sigma) =
  \bigcup_{\rho(P)\cap\relint(\sigma)\neq\emptyset}\pi_\sigma(P)$$
by Proposition \ref{tropicalization and closure}, so
  \[ \dim(X\cap O(\sigma))
  = \dim(\Trop(X)\cap N_\R(\sigma)) = d - \dim(\sigma) \]
proving~(1).
\end{proof}

The following Corollary is a special case of Theorem~\ref{thm:limit.points}.

\begin{cor}\label{cor:limit.pt.proper.int}
Let $\sigma\in\Delta$ be a cone such
  that $X$ intersects $O(\sigma)$ properly.  If $(\xi_i)_{i \in \N}$ is a
  sequence in $\STrop(X)\cap T^\an$ converging to a point
  $\xi\in O(\sigma)^\an$, then $\xi\in\STrop(X)$.
\end{cor}

\begin{proof}
  Since $X^\an$ is closed in $Y_\Delta^\an$, we have
  $\xi\in X^\an\cap O(\sigma)^\an$.  In particular,
  $X\cap O(\sigma)\neq\emptyset$, so $\dim(\sigma)\leq d$.  Let
  $\omega_i = \trop(\xi_i)\in N_\R$ and let
  $\omega = \trop(\xi) \in N_\R(\sigma)$.  Choose a finite collection $\Sigma$
  of integral $\R$-affine $d$-dimensional polyhedra whose union is
  $\Trop(X)\cap N_\R$.  By Proposition~\ref{properties of relevant}, for every $\xi_i$ there exists a polyhedron in $\Sigma$, which has dimension $d$ and is relevant for $\xi_i$. After passing to a subsequence, we may assume that the same polyhedron $P$ works for all 
  $\xi_i$.  By Lemma~\ref{lem:closure.polyhedron}, this implies that
  $\rho(P)\cap\relint(\sigma)\neq\emptyset$, so by
  Proposition~\ref{prop:equiv.int.proper}, the dimension of
  $\pi_\sigma(P)$ is
  \[ \dim(\pi_\sigma(P)) = d - \dim(\sigma) =
  \dim(X\cap O(\sigma)) = \dim(\Trop(X)\cap N_\R(\sigma)). \]
  It follows that $\pi_\sigma(P)$ is $d$-maximal at all of its points, so we can
  apply Theorem~\ref{thm:limit.points}.
\end{proof}

The case when $X$ intersects \emph{all} torus orbits properly is even more
special.  As above we assume that $X\cap T$ is equidimensional of dimension $d$
and that $X\cap T$ is dense in $X$.

\begin{lem}\label{lem:smaller.orbit}
Suppose that for all $\sigma\in\Delta$, either $X\cap O(\sigma)=\emptyset$ or $X$
  intersects $O(\sigma)$ properly.  Fix $\tau\in\Delta$, and let
  $X_\tau\subset Y_{\Delta_\tau}$ be the closure of $X\cap O(\tau)$.  Then
  for all $\bar\sigma = \pi_\tau(\sigma) \in\Delta_\tau$, either
  $X_\tau\cap O(\bar\sigma) = \emptyset$ or $X_\tau$ intersects $O(\bar\sigma)$
  properly.
\end{lem}

\begin{proof}
  If $X \cap O(\tau) = \emptyset$ then the assertion is trivial.
  Otherwise, $X \cap O(\tau)$ is equidimensional of dimension
  $d-\dim(\tau)$ by Lemma~\ref{lem:prop.int.codim}.  Choose a finite
  collection $\Sigma$ of integral $\R$-affine $d$-dimensional polyhedra
  whose union is $\Trop(X)\cap N_\R$, so
  \[ \Trop(X\cap O(\tau)) =
  \bigcup_{\rho(P)\cap\relint(\tau)\neq\emptyset}\pi_\tau(P) \]
  as seen in Proposition \ref{tropicalization and closure}. By Proposition~\ref{prop:equiv.int.proper},  we have
  $\dim(\pi_\tau(P)) = d-\dim(\tau)$ for all such $P$.  Fix $\sigma\in\Delta$
  with $\tau\prec\sigma$, let $\bar\sigma=\pi_\tau(\sigma)\in\Delta_\tau$, and
  suppose that $X_\tau\cap O(\bar\sigma)\neq\emptyset$.  This implies
  $X\cap O(\sigma)\neq\emptyset$, so $\dim(\sigma)\leq d$.  If $P\in\Sigma$ 
  with $\rho(P) \cap \relint(\tau) \neq \emptyset$ 
  has
  $\rho(\pi_\tau(P))\cap\relint(\bar\sigma)\neq\emptyset$ then the closure 
  $\bar{\pi_\tau(P)}$ of
  $\pi_\tau(P)$ in $\bar{N(\tau)}{}_\R^{\Delta_\tau}$ satisfies
  \[ \bar{\pi_\tau(P)}\cap N(\tau)_\R(\bar\sigma) = 
  \bar{\pi_\tau(P)}\cap N_\R(\sigma) =
  \pi_{\bar\sigma}(\pi_\tau(P)) = \pi_\sigma(P) \]
  by Lemma \ref{lem:closure.polyhedron}. In particular, the closure $\bar P$ of $P$ in $\barNRD$ intersects
  $N_\R(\sigma)$, so Lemma \ref{lem:closure.polyhedron} again shows  $\rho(P)\cap\relint(\sigma)\neq\emptyset$.  Hence
  $\dim(\pi_\sigma(P)) = d - \dim(\sigma)$ by
  Proposition~\ref{prop:equiv.int.proper}, since $X$ intersects $O(\sigma)$
  properly, so
  $$\dim(\pi_{\bar\sigma}(\pi_\tau(P))) = 
\dim(\pi_\sigma(P)) = d - \dim(\sigma)=
(d-\dim(\tau)) - \dim(\bar\sigma).$$
  This is true for all $P$ with
  $\rho(\pi_\tau(P))\cap\relint(\bar\sigma)\neq\emptyset$, so
  again by Proposition~\ref{prop:equiv.int.proper}, $X_\tau$ intersects
  $O(\bar\sigma)$ properly.
\end{proof}

\begin{cor}\label{cor:proper.int.ST.closed}
If, for all $\sigma\in\Delta$, either
  $X\cap O(\sigma)=\emptyset$ or $X$ intersects $O(\sigma)$ properly, then
  $\STrop(X)$ is closed in $X^\an$.
\end{cor}

\begin{proof}
  Let $(\xi_i)_{i \in \N}$ be a sequence in $\STrop(X)$ converging to some
  $\xi\in X^\an$.  We wish to show $\xi\in\STrop(X)$.  Passing to a subsequence,
  we may assume that the sequence is contained in a single torus orbit
  $O(\tau)$.  By Lemma~\ref{lem:smaller.orbit}, the closure $X_\tau$ of
  $X\cap O(\tau)$ in the toric variety $Y_{\Delta_\tau}$ satisfies the same
  hypotheses as $X$.  Hence for dimension reasons, if $\sigma\in\Delta$,
  $\tau\prec\sigma$, and $\bar\sigma=\pi_\tau(\sigma)$, then
  $X_\tau\cap O(\bar\sigma)$ is a union of irreducible components of
  $X\cap O(\sigma)$.  By Proposition~\ref{prop:union.of.comps} as applied to
  $X\cap O(\sigma)$, then, we may replace $X$ by $X_\tau$ to assume
  $\{\xi_i\}\subset T^\an$.  Now we apply
  Corollary~\ref{cor:limit.pt.proper.int} to conclude $\xi\in\STrop(X)$.
\end{proof}

\begin{rem}\label{rem:why.one}
  In the situation of Corollary~\ref{cor:proper.int.ST.closed}, suppose that
  $Y_\Delta = Y_\sigma$ is an affine toric variety.  In this case it is possible
  using the techniques of Theorem~\ref{thm:limit.points} to produce a morphism
  $\psi\colon Y_\sigma\to Y'_{\sigma'}$ of affine toric varieties such that the
  composite morphism $\phi\colon X\inject Y_\sigma\to Y'_{\sigma'}$ is finite
  and surjective over every torus orbit $O(\tau')\subset Y'_{\sigma'}$ such that
  $\phi\inv(O(\tau'))\neq\emptyset$.  
From
  this it follows exactly as in Proposition~\ref{prop:equidim.closed} that
  $\phi\inv(S(Y'_{\sigma'})) = \STrop(X)$.  This gives another proof of
  Corollary~\ref{cor:proper.int.ST.closed}, and also shows that $\STrop(X)$ is a
  kind of generalization of a $c$-skeleton in the sense
  of~\cite{ducros03:image_reciproque, ducros12:squelettes_modeles}.
  Compare~\cite[Remark~10.7]{gubler_rabinoff_werner16:skeleton_tropical}.

  In particular, one should be able to use the results of
    Ducros to prove that each torus orbit in the tropical skeleton has a natural
    $\Q$-affine structure.  Forthcoming work of Ducros--Thuillier may allow for
    stronger assertions.
\end{rem}

\begin{rem}
  The hypotheses of Corollary~\ref{cor:proper.int.ST.closed} are commonly
  satisfied in the context of tropical compactifications.  Let $\Trop_0(X)$
  denote the tropicalization of $X$, considered as a subscheme of $Y_\Delta$
  over the field $K$ endowed with the trivial valuation.  If $\Sigma$ is
  a finite collection of integral $\R$-affine $d$-dimensional polyhedra
  whose union is $\Trop(X\cap T)$, then $\Trop_0(X \cap T)$ is the union of
  the recession cones of the polyhedra in $\Sigma$  as we have seen in Remark \ref{recession cones}.

  Suppose that the support of the fan $\Delta$ is equal to $\Trop_0(X \cap T)$.  This
  happens for instance if $\Delta$ is a \emph{tropical fan} for $X\cap T$ as defined by Tevelev~\cite{tevelev07:compactifications} for integral $X\cap T$ and generalized in~\cite[\S 12]{gubler13:guide_tropical} to arbitrary closed subschemes of $T$.  
Then $X$ is proper over $K$ 
by~\cite[Proposition 2.3]{tevelev07:compactifications}, 
and $X$
  intersects each torus orbit $O(\sigma)$ properly by
   \cite[Theorem~14.9]{gubler13:guide_tropical}.
   In this case, $\STrop(X)$ is
  closed by Corollary~\ref{cor:proper.int.ST.closed}, so it is even compact.
\end{rem}

\section{Section of tropicalization}\label{sec:section.trop}

In this section we prove that there is a section of the tropicalization map on
the locus of tropical multiplicity one, and we use the results
of~\S\ref{sec:tropical.skeleton} to examine when this section is continuous.

\subsection{Existence of the section}
Let $\Delta$ be a rational pointed fan in $N_\R$ and let $X\subset Y_\Delta$ be
a closed subscheme.  Suppose that $\omega\in\Trop(X)$ has $m_{\Trop}(\omega)=1$.
We will show that in this case, there is a distinguished Shilov boundary point
of $X_\omega = \trop\inv(\omega)\cap X^\an$, which will be the image of the
section evaluated at $\omega$.  However, as the following example 
shows,
$X_\omega$ may still have multiple Shilov boundary points.

\begin{eg}
  Suppose that the valuation on $K$ is non-trivial.
  Let $X$ be the closed subscheme of $T = \Spec(K[x_1^{\pm1}, x_2^{\pm 1}])$
  given by the ideal $\fa = \big( x_1-1,x_2-1 \big) \cap \big( x_1-1-\varpi\big)$ for
  $\varpi\in K^\times$ with $|\varpi| < 1$.  Then $X$ is the disjoint union of
  the line $\{x_1 = 1+\varpi\}$ with the point $(1, 1)$.  The initial
  degeneration at $\omega=0$ is defined by the ideal
  $\inn_w(\fa) = \big( (x_1-1)^2,(x_1-1)(x_2-1) \big)$ over $\td K$.  This is a
  generically reduced line with an associated point at $(1,1)$.  It has tropical
  multiplicity $1$, but the canonical reduction is the disjoint union of a point
  and a line, so that $X_\omega$ has \emph{two} Shilov boundary points.  Note
  however that one of these points is contained in an irreducible component of
  dimension one, and the other in a component of dimension zero.
\end{eg}

Recall from Definition~\ref{defn:local.dim} that for $\omega\in\Trop(X)$, the
local dimension of $\Trop(X)$ at $\omega$ is denoted
$d(\omega)=\dim\big(\LC_\omega(\Trop(X))\big)$.

\begin{prop}\label{prop:shilov.section}
  Let $X\subset Y_\Delta$ be a closed subscheme and let $\omega\in\Trop(X)$ be a
  point with $m_{\Trop}(\omega)=1$.  Let $\sigma\in\Delta$ be the cone such that
  $\omega\in N_\R(\sigma)$.  Then there is a unique irreducible component $C$ of
  $X\cap O(\sigma)$ of dimension $d(\omega)$ such that $\omega\in\Trop(C)$, and
  there is a unique Shilov boundary point of $C_\omega$.  Moreover,
  $m_{\Trop}(C,\omega) = 1$.
\end{prop}

\begin{proof}
  We immediately reduce to the case $Y_\Delta = O(\sigma) = T$.  First suppose
  that $K$ is algebraically closed, non-trivially valued and
  $\omega\in N_\Gamma$.  By hypothesis, $\inn_\omega(X)$ is irreducible and
  generically reduced, and its dimension is
  $d\coloneq d(\omega) = \dim(X_\omega) = \dim(\td X_\omega)$ by
  Lemma~\ref{lem:equiv.d.omega}.  Replacing $X$ by its underlying reduced
  subscheme $X_{\red}$ has the effect of replacing $\inn_\omega(X)$ by the
  closed subscheme defined by a nilpotent ideal sheaf, so the same is true of
  $\inn_\omega(X_{\red})$.  The conclusions of the Proposition (except the last one) are insensitive
  to nilpotents, so we may assume without loss of generality that $X$ is
  reduced (coming back to the last claim at the end).

  The projection formula of~\S\ref{finite surjective morphism} gives
  \[ 1 = \sum_{Z\surject\inn_\omega(X)} [Z:\inn_\omega(X)], \]
  where the sum is taken over all irreducible components $Z$ of $\td X_\omega$
  dominating $\inn_\omega(X)$.  Since $\td X_\omega\to(\fX_\omega)_s$ is finite
  in any case by Proposition~\ref{prop:two.reductions}, this proves that there
  is a unique irreducible component $Z$ of $\td X_\omega$ of dimension $d$.  Let
  $\xi\in X_\omega$ be the Shilov boundary point reducing to the generic point
  of $Z$.

  Let $D$ be an irreducible component of $X_\omega$ of dimension
  $d=\dim(X_\omega)$.  The inclusion $D\inject X_\omega$ gives a finite morphism
  of canonical reductions $\td D\to\td X_\omega$.  As $\td D$ is equidimensional
  of dimension $d$ by Proposition~\ref{prop:canon.red.facts}(3), every generic
  point of $\td D$ maps to the generic point of $Z$, so every Shilov boundary
  point of $D$ maps to $\xi$ by Proposition~\ref{prop:can.red.surj}.  Since
  $D\to X_\omega$ is injective, its Shilov boundary is $\{\xi\}$. By
  Proposition~\ref{prop:bdy.irred.comps}, $D$ is the unique irreducible component
  of $X_\omega$ containing $\xi$, thus is the unique irreducible component of
  $X_\omega$ of dimension $d$.

  Let $C$ be an irreducible component of $X$ containing the Shilov point $\xi$
  in its analytification.  Then $C_\omega$ is a union of irreducible components
  of $X_\omega$ of the same dimension as $C$
  by Proposition~\ref{prop:affinoid.union.comps}.  Since
  $\xi\in C_\omega$ we have $D\subset C_\omega$, so $\dim(C) = \dim(D) = d$, and
  therefore $C_\omega = D$.  Finally, if $C'$ is another irreducible component
  of $X$ of dimension $d$ such that $C'_\omega\neq\emptyset$, then
  $D\subset C'_\omega$, which is impossible since $\dim(C\cap C') < d$.  Thus
  $C$ is the unique irreducible component of $X$ of dimension $d$ such that
  $\omega\in\Trop(C)$, and $\xi$ is the unique Shilov boundary point of
  $C_\omega = D$.

  Now we allow $K$ and $\omega\in N_\R$ to be arbitrary.  Let $K'\supset K$ be
  an algebraically closed valued extension field whose value group
  $\Gamma'=\val(K^\times)$ is non-trivial and large enough that $\omega\in N_{\Gamma'}$, let
  $X' = X_{K'}$, and let $\pi\colon X'\to X$ be the structural morphism.  By the
  above, there is a unique irreducible component $C'$ of $X'$ of dimension
  $d = d(\omega)$ such that $\omega\in\Trop(C')$, and there is a unique Shilov
  boundary point of $C'_\omega$.  Then $C = \pi(C')$ is the unique irreducible
  component of $X$ of dimension $d$ containing $\omega$ in its tropicalization,
  and $C' = (\pi\inv(C))_{\rm red} = C_{K',\rm red}$. 
Hence
  $C'_\omega = ((C_\omega)_{K'})_{\rm red}$, so $B(C_\omega) = \pi(B(C'_\omega))$ by
  Proposition~\ref{prop:shilov.surj}, so $B(C_\omega)$ has only one point.
  
  We come now to the last claim no longer assuming that $X$ is reduced. By definition, $m_{\Trop}(C,\omega) = 1$ provided that $\inn_\omega(C_{K'})$
  is irreducible and generically reduced.  The map on initial degenerations
  $\inn_\omega(C_{K'})\to\inn_\omega(X')$ is a closed immersion, as both are
  closed subschemes of $\inn_\omega(T_{K'})\cong\bG_{m,\td K'}^n$.  Since
  $\inn_\omega(X')$ is irreducible and generically reduced, and
  $\dim(\inn_\omega(C_{K'})) = \dim(\inn_\omega(X'))$, this shows that
  $\inn_\omega(C_{K'})$ is also irreducible and generically reduced.
\end{proof}

\begin{defn}\label{def:section}
  Let $X\subset Y_\Delta$ be a closed subscheme.  Write
  \[ \multone \coloneq \big\{\omega\in\Trop(X)\mid m_{\Trop}(\omega) = 1
  \big\} \]
  for the tropical multiplicity-$1$ locus in $\Trop(X)$.  If
  $\omega\in\multone\cap N_\R(\sigma)$ for $\sigma\in\Delta$, we let $C(\omega)$ be
  the unique irreducible component of $X\cap O(\sigma)$ of dimension $d(\omega)$
  with $\omega\in\Trop(C(\omega))$, and we define
  \[\begin{split}
    s_X(\omega) &= \text{the unique Shilov boundary point of }
    C(\omega)_\omega \\
    &= \text{the unique Shilov boundary point $\xi$ of } X_\omega \text{ such that }
    \dim_{\xi}(X_\omega) = d(\omega).
  \end{split} \]
  We regard $s_X$ as a map $\multone\to\STrop(X)\subset X^\an$.
\end{defn}

It follows immediately from the above definition that the image of $s_X$ is  contained in $\STrop(X)$. 
By construction, $\trop\circ s_X$ is the
identity, so $s_X$ is a section of $\trop$ defined on $\multone$.  If
$X\cap O(\sigma)$ is equidimensional of dimension $d$ then $d(\omega) = d$ for
all $\omega\in\Trop(X)\cap N_\R(\sigma)$ by the Bieri--Groves theorem, so
$s_X(\omega)$ is the unique Shilov boundary point of $X_\omega$ in this case:
$\STrop(X)\cap\trop\inv(\omega) = \{s_X(\omega)\}$.  Therefore our $s_X$
coincides with the section considered
in~\cite[\S10]{gubler_rabinoff_werner16:skeleton_tropical} for $X\subset T$
irreducible.  It also coincides with the section
$s\colon\barNRD\to Y_\Delta^\an$ introduced in~\S\ref{par:tropicalization} when
$X = Y_\Delta$.

\begin{rem}[Behavior with respect to extension of scalars]
  \label{rem:section.extend.scalars}
  Let $K'$ be a valued extension field of $K$, let $X' = X_{K'}$, and let
  $\pi\colon X'^\an\to X^\an$ be the structural map.  Then
  $\multone[X'] = \multone$ by the definition of $m_{\Trop}$.  It is clear from
  the proof of Proposition~\ref{prop:shilov.section} that
  $\pi\circ s_{X'} = s_X$.
\end{rem}

The uniqueness of $C(\omega)$ for $\omega\in\multone$ gives us the following
decomposition of $\multone$.  Supposing for simplicity that $X$ is a closed
subscheme of $T$, for an irreducible component $C$ of $X$ let
\begin{equation}\label{eq:mult1.decomp}
  Z(C) = \big\{\omega\in\multone\mid C(\omega) = C \big\}.
\end{equation}
In other words, $Z(C)$ is the set of all multiplicity-$1$ points $\omega$ such
that $C$ is the unique irreducible component of $X$ of dimension $d(\omega)$
with $\omega\in\Trop(C)$. Hence $s_X(\omega)$ is the Shilov boundary point of
$C_\omega$.  Then by definition, $\multone$ is the disjoint union
$\Djunion_C Z(C)$, and $s_X = s_C$ on $Z(C)$ (which makes sense by the final
assertion of Proposition~\ref{prop:shilov.section}).
This observation, along with the following Lemma, will
allow us to reduce topological questions about $\multone$ to the case when $X$
is irreducible.

\begin{lem}\label{lem:decomp.multone}
  Let $X\subset T$ be a closed subscheme, let $C$ be an irreducible component of $X$, and define $Z(C)$ as
  in~\eqref{eq:mult1.decomp}.  Then $Z(C)$ is open and closed in $\multone$.
\end{lem}

\begin{proof}
  Since $\multone = \Djunion_C Z(C)$, it is enough to prove that $Z(C)$ is
  closed.  Let $(\omega_i)_{i\in \N}$ be a sequence in $Z(C)$ converging to
  a point $\omega\in\multone$.  By passing to a subsequence, we may assume that
  all $\omega_i$ are contained in a single polyhedron $P\subset\Trop(C)$ of
  dimension $d = \dim(C)$.  Then $\omega\in P$ as well.  We claim that
  $d(\omega) = d$.  If not, then there exists a polyhedron $P'\subset\Trop(X)$
  of dimension $d'>d$ also containing $\omega$. We note that $\LC_{\omega}(P)$ is not included in $\LC_{\omega}(P')$ as we have $d(\omega_i)=d$ for all $i$. But then
  $\LC_\omega(\Trop(X)) = \Trop(\inn_\omega(X))$ is not equidimensional, so
  $\inn_\omega(X)$ is not irreducible, which contradicts $m_{\Trop}(\omega)=1$.
  This proves the claim.  
  We have $\omega\in P \subset \Trop(C)$,
  so $C$ is the unique irreducible component of $X$ of dimension $d = d(\omega)$
  containing $\omega$ in its tropicalization, and therefore $\omega\in Z(C)$.
\end{proof}

\subsection{Continuity on torus orbits}
The following analogue of Corollary~\ref{cor:closed.on.orbits} is a
generalization
of~\cite[Theorem~10.6]{gubler_rabinoff_werner16:skeleton_tropical} to
reducible $X$ and a general non-Archimedean ground field $K$.

\begin{prop}\label{prop:section.cont}
  For $\sigma\in\Delta$, the section of tropicalization $s_X$ is continuous on
  the multiplicity-$1$ locus
  $\multone\cap N_\R(\sigma)$.  Moreover, if $Z\subset\multone\cap N_\R(\sigma)$ is
  contained in the closure of its interior in $\Trop(X)\cap N_\R(\sigma)$, then
  $s_X$ is the unique continuous partial section of $\trop\colon X^\an\to\Trop(X)$
  defined on $Z$.
\end{prop}

\begin{proof}
  The statement of the Proposition is intrinsic to the torus orbit $O(\sigma)$,
  so we immediately reduce to the case $Y_\Delta = O(\sigma) = T$.  Since
  $\multone$ is the disjoint union (as a topological space) of the subspaces
  $Z(C)$ by Lemma~\ref{lem:decomp.multone}, it is enough to prove continuity
  and uniqueness on $Z(C)$ for a fixed irreducible component $C$ of $X$.  Since
  $s_X = s_C$ on $Z(C)$, we may replace $X$ by $C$ to assume $X$ irreducible.

  Let $Z = \multone$.  It suffices to show that $s_X(Z)$ is closed in
  $\trop\inv(Z)\cap X^\an$ (endowed with its relative topology), since
  $\trop\colon\trop\inv(Z)\cap X^\an \to Z$ is a proper map to a first-countable
  topological space, thus is a closed map
  by~\cite{palais70:when_proper_maps_are_closed}.  In this case, $s_X(\omega)$
  is the unique Shilov boundary point of $X_\omega$ for all $\omega\in Z$.  Thus
  $s_X(Z) = \trop\inv(Z)\cap\STrop(X)$, which is closed in
  $\trop\inv(Z)\cap X^\an$ by Proposition~\ref{prop:equidim.closed}.  This
  settles the continuity assertion.

  Now let $Z\subset\multone$ be a subset which is contained in the closure of
  its interior in $\Trop(X)$, still assuming (as we may) that $X$ is
  irreducible.  The proof of uniqueness of the section $s_X$ on $Z$ goes through
  exactly as in~\cite[Theorem~10.6]{gubler_rabinoff_werner16:skeleton_tropical},
  which only uses that in the situation of
  Proposition~\ref{prop:equidim.closed}, we have
  \[ \phi\inv(\STrop(\bG_m^d)) = \STrop(X), \] 
  and that
  $\STrop(X)\cap\trop\inv(\omega) = \{s_X(\omega)\}$ for $\omega\in\multone$. 
\end{proof}

\begin{rem}
  In the proof of Proposition~\ref{prop:section.cont}, after reducing to the
  case of irreducible $X$, we could have
  applied~\cite[Theorem~10.6]{gubler_rabinoff_werner16:skeleton_tropical}
  after an extension of scalars to prove continuity.  However, it is instructive
  to see why continuity follows from the more general results
  of~\S\ref{sec:tropical.skeleton}. 
\end{rem}

\subsection{A sequential continuity criterion}
The section $s_X$ need not be continuous on all of $\Trop(X)$ when $X$ is a
closed subscheme of a toric variety $Y_\Delta$.

\begin{eg}\label{eg:counter.eg.3}
  In Example~\ref{eg:counter.eg}, $s_X(0,r,2r)$ does not tend to
  $s_X(0,\infty,\infty)$ as $r\to\infty$.
\end{eg}

It is easy to see that $\barNRD$ is a  metric space.
Hence $s_X$
is continuous if and only if it is sequentially continuous.  Given a sequence
$(\omega_i)_{i\in \N}$ in $\multone$ converging to a point
$\omega\in\multone$, if $\omega$ and all $\omega_i$ are contained in the same
torus orbit then $s_X(\omega_i)\to s_X(\omega)$ by
Proposition~\ref{prop:section.cont}.  If $(\omega_i)_{i \in \N}$ is contained in several
different torus orbits, one checks sequential continuity on each torus orbit
separately.  Hence verifying sequential continuity amounts to showing that if
$(\omega_i)_{i\in \N}$ is a sequence in $\multone\cap N_\R$ converging to
$\omega\in\multone\cap N_\R(\sigma)$, then $s_X(\omega_i)$
converges to $s_X(\omega)$.  Decomposing by local dimension, one breaks
$(\omega_i)_{i \in \N}$ into several subsequences, each one of which is contained in a
single polyhedron in $\Trop(X)$ which is $d$-maximal
(Definition~\ref{def:d.maxmality}) at each point of the subsequence.

The main ingredient in the proof of the following result is Theorem~\ref{thm:limit.points}.

\begin{thm}\label{thm:continuity}
  Let $\Delta$ be a pointed rational fan in $N_\R$ and let $X\subset Y_\Delta$
  be a closed subscheme.  Let $(\omega_i)_{i \in \N}$ be a sequence in
  $\multone\cap N_\R$ converging to a point $\omega\in\multone\cap N_\R(\sigma)$
  for $\sigma\in\Delta$.  Suppose that there exists a
  polyhedron $P\subset\Trop(X)\cap N_\R$ which is $d$-maximal at each
  $\omega_i$.  If $\pi_\sigma(P)$ is $d$-maximal at $\omega$, then
  $s_X(\omega_i)\to s_X(\omega)$.
\end{thm}

\begin{proof}
  First we reduce to the case where $X\cap T$ is irreducible and dense in $X$. In fact, equidimensionality would be enough to apply Theorem~\ref{thm:limit.points} later.

  Let $d = \dim(P)$, so $d(\omega_i) = d$ for all $i$.  By Proposition~\ref{prop:shilov.section}, there exists a unique irreducible component $C_i = C(\omega_i)$
   of $X\cap T$ of dimension $d$ with
  $\omega_i\in\Trop(C_i)$.  As all other irreducible components of $X$ meeting
  $X_{\omega_i}$ have smaller dimension, we must have
  $\LC_{\omega_i}(P)\subset\LC_{\omega_i}(\Trop(C_i))$.  Let $C$ be a $d$-dimensional
  irreducible component of $X$ occurring as $C_i$ for infinitely many $i$. Replacing $P$ by $\Trop(C) \cap P$, we may asssume that 
   $P\subset\Trop(C)$. Note that $d$-maximality of $\pi_\sigma(P)$ at $\omega$ is preserved. 
  The limit point $\omega$ lies in the  closure $\overline{P}$
  of $P$ intersected with $N_\R(\sigma)$, and hence in $ \pi_\sigma(P)$ by Lemma~\ref{lem:closure.polyhedron}, which is $d$-maximal at $\omega$.  Letting
  $\bar C\subset X$ be the closure of $C$ in $Y_\Delta$, we have
  $\bar P\subset\Trop(\bar C)$, so Lemma \ref{lem:equiv.d.omega} yields
  \[ \dim(X_\omega) = d(\omega) = \dim(\LC_\omega(\Trop(\bar C))) = \dim(\bar
  C_\omega). \]
  Therefore using Proposition~\ref{prop:shilov.section}, the unique irreducible component of $X\cap O(\sigma)$ containing
  $s_X(\omega)$ is an irreducible component of $\bar C\cap O(\sigma)$.  
  This is
  true for all irreducible components $C$ occurring as $C_i$ for
  infinitely many $i$, so we may assume 
   that $X\cap T = C$ is irreducible and dense in
  $X=\bar C$.

  Let $W =\{\omega_i \mid i \in \N\}$ and hence 
  $\bar W = \{\omega_i \mid i \in \N\}\cup\{\omega\}$.  Proving the Theorem amounts to
  showing that $s_X$ is continuous on the subspace $\bar W$ of $\Trop(X)$.  As
  $\bar W$ is a first-countable Hausdorff space, as in the proof of
  Proposition~\ref{prop:section.cont} it suffices to show that $s_X(\bar W)$ is
  closed in $\trop\inv(\bar W)\cap X^\an$.
  By~\cite[Th\'eor\`eme~5.3]{poineau13:angeliques}, $X^\an$ is a
  Fr\'echet--Urysohn space, so for every subspace, closure is the same as
  sequential closure.  Since $s_X(W)$ is discrete, we only have to prove that
  $s_X(\omega)$ is the unique (sequential) limit point of $s_X(W)$ not contained
  in $s_X(W)$.  Any such is contained in $X_\omega$, so let $\xi\in X_\omega$ be
  a limit point of $s_X(W)$.  By Theorem~\ref{thm:limit.points},
  $\xi\in\STrop(X)$ and $\dim_\xi(X_\omega) = d(\omega)$.  The only point with
  these properties is $s_X(\omega)$.
\end{proof}

We can apply the preceeding result in the following situation.

\begin{cor}\label{cor:continuity.proper.int}
  Let $\Delta$ be a pointed rational fan in $N_\R$ and let $X\subset Y_\Delta$
  be a closed subscheme.  Suppose that $X\cap T$ is equidimensional 
and let $\sigma\in\Delta$ be a cone such
  that 
  $X$ intersects $O(\sigma)$ properly. Then every sequence $(\omega_i)_{i\in \N}$ in $\multone\cap N_\R$ converging to a point $\omega\in\multone\cap N_\R(\sigma)$ has the property that $s_X(\omega_i)$ converges to $s_X(\omega)$. 
  \end{cor}

\begin{proof} 
We choose a finite collection $\Sigma$ of integral $\R$-affine $d$-dimensional polyhedra whose union is $\Trop(X) \cap N_\R$. For every $P$ in $\Sigma$ such that $\rho(P)$ meets the relative interior of $\sigma$,   Proposition~\ref{prop:equiv.int.proper} implies that $\pi_\sigma(P)$ has dimension $\dim(X)-\dim(\sigma)=d(\omega)$. Hence we can apply Theorem~\ref{thm:continuity} to every $P$ containing  infinitely many  $\omega_i$. 
\end{proof}

\begin{thm}\label{thm:continuity_equi}
  Let $\Delta$ be a rational pointed fan in $N_\R$ and let $X \subset Y_\Delta$
  be a closed subscheme.  Suppose that $X \cap T$ is dense in $T$, and that for all $\sigma \in \Delta$ the subscheme $X \cap O(\sigma)$ of $O(\sigma)$ is either empty or equidimensional of  dimension 
  $d_\sigma$, where we put $d_0 = d$. Assume that $\Trop(X) \cap N_\R$ can be covered by finitely many $d$-dimensional polyhedra $P$ with the following property: If the recession cone of $P$ meets the relative interior of $\sigma$, the projection $\pi_\sigma(P)$ has dimension $d_\sigma$. Then 
  $s_X\colon\multone\to X^\an$ is continuous.
\end{thm}

\begin{proof} Let $(\omega_i)_{i \in \N}$ be any sequence in $\multone$ converging to $\omega \in\barNRD$. We have to show that $s_X(\omega_i)$ converges to $s_X(\omega)$. By passing to subsequences we may assume that all $\omega_i$ are contained in $N_\R(\tau)$ for some face $\tau \in \Delta$. Then $\omega$ lies in $N_\R(\sigma)$ for a face $\sigma$ with $\tau \prec \sigma$. 
 According to Proposition~\ref{tropicalization and closure}, $\Trop(X) \cap
 N_\R(\tau)$ is covered by the polyhedra $\pi_\tau(P)$, where $P$ runs over all polyhedra in our finite covering of $\Trop(X) \cap N_\R$ with $\rho(P) \cap \relint(\tau) \neq \emptyset$. Assume that infinitely many $\omega_i$ are contained in the same $\pi_\tau(P)$. By hypothesis, $\pi_\tau(P)$ has dimension $d_\tau$.  Now we apply Theorem~\ref{thm:continuity} to the subscheme $X \cap Y_{\Delta_\tau}$ of $Y_{\Delta_\tau}$: We consider the face $\overline{\sigma} = \pi_\tau(\sigma)$ of $\Delta_\tau$. The projection $\pi_{ \overline{\sigma}}(\pi_\tau 
(P)) = \pi_\sigma(P)$ has dimension $d_\sigma$, which implies that $s_X$ applied to our subsequence converges to $s_X(\omega)$. This proves our claim.
\end{proof}

\begin{thm}\label{cor:proper.int.continuity}
 Let $X$ be a closed subscheme of $Y_\Delta$ such that $X \cap T$ is {equidimensional and} dense in $X$.  Assume additionally that  for all $\sigma\in\Delta$, either
  $X\cap O(\sigma)=\emptyset$ or $X$ intersects $O(\sigma)$ properly.  Then
  $s_X\colon\multone\to X^\an$ is continuous.
\end{thm}

\begin{proof}
  This follows from Theorem~\ref{thm:continuity_equi} using Lemma~\ref{lem:prop.int.codim} and Proposition~\ref{prop:equiv.int.proper}. 
\end{proof}

\begin{eg}\label{eg:grassmannians}
  Let us now briefly discuss the example of the Grassmannian of planes
  $X = \mbox{Gr}(2,n)$ in $n$-space with its Pl\"ucker embedding
  $\varphi: \mbox{Gr}(2,n) \hookrightarrow \bP^{\binom{n}{2} -1}$. The
  toric variety in this example is given by projective space with projective
  coordinates $p_{kl}$ indexed by pairs $k< l$ in $\{1, \ldots, n\}$. In this
  case, continuity for the section map on tropical Grassmannians was shown
  directly in \cite{cueto_habich_werner13:grassmannians}. See also
  \cite{werner16:analytif_tropical} for an expository account of this
  construction. Note that in \cite{cueto_habich_werner13:grassmannians}, we use
  $\log$ instead of $-\log$ for tropicalization.

  Put $\mathcal{T}\mbox{Gr}(2,n) = \Trop( \mbox{Gr}(2,n))$. Let $N_\R$ be the
  cocharacter space of the dense torus in $\bP^{\binom{n}{2} -1}$.  Let $J$ be a
  proper subset of the set of all projective coordinates $\{p_{kl} \}$, and let
  $E_J$ be the subvariety of projective space, where precisely the coordinates
  in $J$ vanish. These are the torus orbits in projective space.  Hence the
  locally closed subvariety $\mbox{Gr}_J(2,n)$ from
  \cite{cueto_habich_werner13:grassmannians} is the intersection of
  $ \mbox{Gr}(2,n)$ with a torus orbit. By
  \cite{cueto_habich_werner13:grassmannians}, Lemma 5.2, $\mbox{Gr}_J(2,n)$ is
  irreducible. In particular, the intersection of $\mbox{Gr}(2,n)$ with every
  torus orbit is equidimensional.

  Now let $C_T$ be a maximal cone in $\mathcal{T}\mbox{Gr}(2,n) \cap N_\R$.  It
  has dimension $2(n-2)$. By \cite{speyer_sturmfel04:tropical_grassman}, it
  corresponds to all phylogenetic trees on the trivalent combinatorial tree
  $T$. With the help of the results in \cite{cueto_habich_werner13:grassmannians}, section 4 and 5,   one can show that
  the projection of $C_T$ to the cocharacter space of the torus orbit $E_J$ is $d$-maximal. Moreover,
  it is shown in \cite{cueto_habich_werner13:grassmannians}, Corollary 6.5 that
  the tropical Grassmannian has tropical multiplicity one everywhere. Therefore
  we can apply Theorem~\ref{thm:continuity_equi} to deduce the existence of a
  continous section to the tropicalization map.

  This provides a conceptual explanation for the existence of a section for the tropical Grassmannian. Note however that we need the combinatorial arguments developed 
  in \cite{cueto_habich_werner13:grassmannians} in order to show that the tropical Grassmannian satisfies
  the prerequisites of Theorem~\ref{thm:continuity_equi}.
\end{eg}

\section{Sch\"on subvarieties}
\label{sec:helm-katz}

In this section, we will prove that when $X$ is a so-called sch\"on subvariety
of a torus $T$, the tropical skeleton $\STrop(X)$ can be identified with the
parameterizing complex of Helm--Katz~\cite{helm_katz12:monodrom_filtration}.  It
also coincides with the skeleton of a strictly semistable pair in the sense
of~\cite{gubler_rabinoff_werner16:skeleton_tropical}.  Since the latter is a
deformation retract of $X^\an$, this answers a question of Helm--Katz.

In this section, $K_0$ is a discretely valued field with value group $\Gamma_0$.
We write $\bar K_0$ for an algebraic closure, and $K$ for the completion of
$K_0$.  We assume that the value group $\Gamma$ of $\bar K_0$ and $K$ is equal
to $\Q$, so that $\Gamma_0 = v(K_0^\times) = r\Z$ for some $r\in\Q^\times$.

\subsection{Tropical compactifications}
Here we recall several standard facts about tropical compactifications.  For
more details, see~\cite[\S7, \S12]{gubler13:guide_tropical}.

Let
$\R_+=[0,\infty)\subset\R$.  Let $\Sigma$ be a pointed rational  fan in
$N_\R\times\R_+$.  For
$r\in\R_+$ we let $\Sigma_r = \{\sigma\cap N_\R\times\{r\}\mid\sigma\in\Sigma\}$; this is a
polyhedral complex in $N_\R$, which is a fan when $r=0$.  The $K^\circ$-toric
scheme associated to $\Sigma$ will be denoted $\sY_\Sigma$.  This is a finitely
presented, flat, normal, separated $K^\circ$-scheme.  The generic fiber of $\sY_\Sigma$ is
the toric variety $Y_{\Sigma_0}$.  The torus $T=\Spec(K[M])$ is dense in $\sY_\Sigma$, and
the integral torus $\sT\coloneq\Spec(K^\circ[M])$ acts on $\sY_\Sigma$.  The
torus orbits on the generic (resp.\ special) fiber correspond to cones (resp.\
polyhedra) in $\Sigma_0$ (resp.\ $\Sigma_1$); for $\sigma\in\Sigma_0$ (resp.\
$P\in\Sigma_1$) we let $O(\sigma)\subset Y_{\Sigma_0}$ (resp.\
$O(P)\subset(\sY_\Sigma)_s$) denote the corresponding orbit.  For
$\sigma\in\Sigma_0$ and $P\in\Sigma_1$, we have
\begin{equation}\label{eq:complementary.dims}
  \dim(\sigma) + \dim(O(\sigma)) = \dim(N_\R) \sptxt{and}
  \dim(P) + \dim(O(P)) = \dim(N_\R).
\end{equation}
By a \emph{variety} we mean a geometrically integral, separated, finite-type
scheme over a field. In the following, $X$ is always a closed subvariety of the torus $T$. 

\begin{defn}\label{def:tropical.compactification}
  Let $\Sigma$
  be a pointed rational  tropical fan in $N_\R\times\R_+$ and let $\sX$ be the
  closure of $X$ in $\sY_\Sigma$.  We say that $\sX$ is a
  \emph{tropical compactification} of $X$ provided that $\sX$ is proper over
  $K^\circ$ and the multiplication map
  \[ \mu\colon \sT\times_{K^\circ} \sX\To \sY_\Sigma \]
  is faithfully flat.  In this case we call $\Sigma$ a
  \emph{tropical fan for $X$}.
\end{defn}

We refer the reader to~\cite[\S12]{gubler13:guide_tropical} for proofs of the
following facts in this context.  We write $|\Sigma_r|$ for the support of
$\Sigma_r$ in $N_\R$.

\begin{thm}\label{thm:trop.fan.facts}
  Let $\Sigma$ be a  tropical fan for $X$.
  \begin{enumerate}
  \item If $\Sigma'$ is a rational fan which subdivides $\Sigma$, then
    $\Sigma'$ is a  tropical fan for $X$.
  \item The support $|\Sigma_1|$ is equal to $\Trop(X)$.
  \item If $O\subset\sY_\Sigma$ is any torus orbit, then $\sX\cap O$ is a
    non-empty pure dimensional scheme and
    $\codim(\sX\cap O,\,O) = \codim(X,\,T)$.
  \end{enumerate}
\end{thm}

It follows from Theorem~\ref{thm:trop.fan.facts}(3) that $\sX\setminus X$ is a
closed subscheme of pure codimension one, which we regard as a reduced Weil
divisor on $\sX$.

\begin{defn}\label{def:boundary.divisor}
  Let $\sX\subset\sY_\Sigma$ be a tropical compactification of a variety
  $X\subset T$.  The \emph{boundary divisor} of $\sX$ is the Weil divisor
  $\sX\setminus X$.  The \emph{horizontal part} of the boundary divisor is the
  closure of $\sX_K\setminus X$ in $\sX$, and the \emph{vertical part} is the
  special fiber $\sX_s$.
\end{defn}

In case of a tropical compactification as above, replacing $K_0$ be a finite
extension, we may always assume that all vertices of $\Sigma_1$ are in
$N_{\Gamma_0}$. Then the toric scheme $\sY_\Sigma$ is defined over $K_0^\circ$
(see~\cite[Proposition 7.11]{gubler13:guide_tropical}).  If $X \subset T$ is
defined over $\bar K_0$ (that is, $X$ is the extension of scalars of a
subvariety of $\Spec(\bar K_0[M])$), then after passing to a larger finite
extension, we may assume that $X$ is defined over $K_0$, so $\sX$ is defined
over $K_0^\circ$. This means that the tropical compactification
$\sX\subset\sY_\Sigma$ is obtained by base change from the corresponding toric
compactification over $K_0^\circ$. This will be used to quote results proved
over discrete valuation rings.

We have the following relationship between torus orbits and initial
degenerations.  Let $P\in\Sigma_1$, let $\sT_P\subset\sT$ be the subtorus that
acts trivially on $O(P)\subset\sY_\Sigma$, and let
$\omega\in\relint(P)\cap N_\Gamma$.
Then by~\cite[\S3]{helm_katz12:monodrom_filtration}, there is a natural map
$\inn_\omega(X)\to\sX\cap O(P)$, and an isomorphism
\begin{equation}\label{eq:orbit.inn}
  \inn_\omega(X) \isom \sT_P\times(\sX\cap O(P)).
\end{equation}
In particular, $\inn_\omega(X)$ is smooth if and only if $\sX\cap O(P)$ is smooth.

\subsection{Sch\"on subvarieties}
The following class of subvarieties of tori are sometimes called ``tropically
smooth''.

\begin{defn}\label{def:schon}
  Let $X\subset T$ be a subvariety.  We say that
  $X$ (or more precisely, the embedding $X\inject T$) is \emph{sch\"on} provided
  that there exists a tropical compactification $\sX\subset\sY_\Sigma$ of $X$
  such that the multiplication map
  $\mu\colon \sT\times_{K^\circ} \sX\to \sY_\Sigma$
  is smooth.
\end{defn}

Note that if $X$ is sch\"on then it is smooth. The following result is due to Luxton and Qu~\cite[Proposition~7.6]{luxton_qu11:tropical_compactification}.

\begin{lem}\label{lem:schon.compactifications}
  Let $X\subset T$ be a sch\"on subvariety defined over $\bar K_0$ and let
  $\sX\subset\sY_\Sigma$ be any tropical compactification.  Then
  $\mu\colon \sT\times_{K^\circ} \sX\to \sY_\Sigma$ is smooth.
\end{lem}

\begin{prop}[Helm--Katz]\label{prop:schon.equiv}
  Let $X\subset T$ be a subvariety defined over $\bar K_0$.  The following are
  equivalent:
  \begin{enumerate}
  \item $X$ is sch\"on.
  \item $\inn_\omega(X)$ is smooth for all $\omega\in\Trop(X)$.
  \item For any tropical compactification $\sX\subset\sY_\Sigma$ and any
    polyhedron $P\in\Sigma_1$, the intersection $\sX\cap O(P)$ is smooth.
  \end{enumerate}
\end{prop}

\begin{proof}
  See~\cite[Proposition~3.9]{helm_katz12:monodrom_filtration}. 
\end{proof}

De Jong defined in~\cite[\S6]{dejong96:alterations} the notion of  a strictly semistable pair $(\sX,H)$ over a complete discrete valuation ring $R$. Roughly, this consists of a pair $(\sX,H)$, where $\sX$ is an irreducible, proper, flat, separated $R$-scheme, $H$ is an effective ``horizontal'' Cartier divisor on $\sX$, and $H + \sX_s$ is a divisor with strict normal crossings. Note that de Jong denotes such a strictly semistable pair by $(\sX,H + \sX_s)$, but we will omit the vertical part in the notation.

\begin{prop}[Helm--Katz]\label{prop:schon.ss.pair}
  Let $X\subset T$ be a sch\"on subvariety defined over $\bar K_0$.  Then there
  exists a tropical compactification $\sX\subset\sY_\Sigma$ such that, letting
  $H$ be the horizontal part of the boundary divisor of $\sX\subset\sY_\Sigma$
  (Definition~\ref{def:boundary.divisor}), then $H$ is Cartier and the pair
  $(\sX,H)$ is the base change a strictly semistable pair defined over the
  valuation ring of a finite subextension of $K/K_0$.
\end{prop}

\begin{proof}
  Proposition~3.10 in~\cite{helm_katz12:monodrom_filtration} says that $\sX$ is
  strictly semistable over $K^\circ$ in the sense of
  de~Jong~\cite[\S2]{dejong96:alterations}.  For the stronger fact that
  $(\sX,H)$ is a strictly semistable pair, one has to refer back to the proof
  of~\cite[Proposition~2.3]{helm_katz12:monodrom_filtration}.
\end{proof}

\begin{rem}
  The desingularization
  process~\cite[Proposition~2.3]{helm_katz12:monodrom_filtration} used in the
  proof of Proposition~\ref{prop:schon.ss.pair} is the primary reason for the
  assumption that $X \subset T$ is defined over $\bar K_0$.  It involves
  subdividing $\Trop(X)$ into unimodular simplicial polyhedra, which is not
  possible in general when the value group is too large, e.g.\ if $\Gamma=\R$.
  On the other hand, the proofs of Lemma~\ref{lem:schon.compactifications} and
  Proposition~\ref{prop:schon.equiv} can be extended to any valued field.
\end{rem}

\subsection{Skeletons and the parameterizing complex}
Following Helm--Katz, the pair consisting of a sch\"on subvariety $X\subset T$
and a tropical compactification $\sX\subset\sY_\Sigma$ satisfying the
conclusions of Proposition~\ref{prop:schon.ss.pair} is called a
\emph{normal crossings pair}.  We associate to a normal crossings pair
$(X,\sY_\Sigma)$ a piecewise linear set $\HK(X,\sY_\Sigma)$ (denoted
$\Gamma_{(X,\bP)}$ in~\cite[\S4]{helm_katz12:monodrom_filtration}), defined as
follows.  For $P\in\Sigma_1$ we let $\sX_P$ be the closure of $\sX\cap O(P)$.
The $k$-cells of $\HK(X,\sY_\Sigma)$ are pairs $(P,C)$, where $P$ is a
$k$-dimensional polyhedron in $\Sigma_1$ and $C$ is an irreducible component
(equivalently, a connected component) of $\sX_P$.  The cells on the boundary of
$(P,C)$ are the cells of the form $(P',C')$, where $P'$ is a facet of $P$ and
$C'$ is the irreducible component of $\sX_{P'}$ containing $C$; there is only
one such component as $\sX_{P'}$ is smooth.  The piecewise linear set
$\HK(X,\sY_\Sigma)$ is called the \emph{Helm--Katz parameterizing complex}; it
maps naturally to $\Sigma_1$ by sending $(P,C)$ to $P$.
See~\cite[\S4]{helm_katz12:monodrom_filtration} for details.  Note that although
Helm--Katz work over a discretely valued field, their construction only depends
on the (geometric) special fiber of $\sX$, so one may as well pass to the
completion of the algebraic closure first.  The complex
$\HK(X,\sY_\Sigma)$ inherits an integral $\Gamma$-affine
structure~\cite[\S2]{gubler_rabinoff_werner16:skeleton_tropical} from
$\Trop(X)$.

Let $(\sX,H)$ be the strictly semistable pair associated to the normal crossings
pair $(X,\sY_\Sigma)$ in Proposition~\ref{prop:schon.ss.pair}.  It follows
from~\cite[3.2]{gubler_rabinoff_werner16:skeleton_tropical} that $(\sX,H)$ is a
strictly semistable pair in the sense
of~\cite[Definition~3.1]{gubler_rabinoff_werner16:skeleton_tropical}.  Such a
pair admits a canonical skeleton $S(\sX,H)\subset X^\an$ as constructed
in~\cite[\S4]{gubler_rabinoff_werner16:skeleton_tropical}.
By~\cite[Proposition~4.10]{gubler_rabinoff_werner16:skeleton_tropical} there is
a bijective, order-reversing correspondence between strata $S$ in $\sX_s$ and
polyhedra $\Delta_S$ in the skeleton $S(\sX,H)$.  These strata are precisely
the connected components of the intersections with torus orbits $\sX\cap O(P)$
for $P\in\Sigma_1$: indeed, the stratum closures of $\sX_s$ are the
intersections with stratum closures in $\sY_\Sigma$ because they are smooth of
the correct dimension by Theorem~\ref{thm:trop.fan.facts}(3) and
Proposition~\ref{prop:schon.equiv}(3).  Hence the polyhedra $\Delta_S$
correspond to cells in $\HK(X,\sY_\Sigma)$.  This correspondence respects the
facet relation, so that $\HK(X,\sY_\Sigma)$ and $S(\sX,H)$ are identified as
piecewise linear sets.  It is not obvious that this identification respects the
respective integral $\Q$-affine structures (see Lemma~\ref{lem:skel.to.trop}
below).

\begin{thm}\label{thm:skeletons.complex}
  Let $X\subset T$ be a sch\"on subvariety defined over $\bar K_0$, and let
  $(X,\sY_\Sigma)$ be a normal crossings pair.  Let $(\sX,H)$ be the associated
  strictly semistable pair (Proposition~\ref{prop:schon.ss.pair}), with skeleton
  $S(\sX,H)$ as above.  Then
  \begin{enumerate}
  \item $S(\sX,H) = \STrop(X)$ as subsets of $X^\an$, and
  \item there is a canonical isomorphism $S(\sX,H)\isom\HK(X,\sY_\Sigma)$ of
    abstract integral $\Q$-affine piecewise linear sets,
    making the following triangle commute:
    \[\xymatrix @=.15in{
      {S(\sX,H)} \ar[rr]^\sim \ar[dr]_{\trop} & & {\HK(X,\sY_\Sigma)} \ar[dl] \\
      & {\Trop(X)} &
    }\]
  \end{enumerate}
\end{thm}

One consequence of Theorem~\ref{thm:skeletons.complex} is that for sch\"on $X$,
one can construct the (Berkovich) skeleton $S(\sX,H)$ using only tropical data,
namely, $\Trop(X)$ and the initial degenerations of $X$.  This is a kind of
``faithful tropicalization'' result.  As a consequence, any invariant of $X$
that can be recovered from $S(\sX,H)$ can be calculated tropically.  For example,
the skeleton $S(\sX,H)$ is a strong deformation retraction of $X^\an$
by~\cite[\S4.9]{gubler_rabinoff_werner16:skeleton_tropical}, so we have the
following Corollary.

\begin{cor}\label{cor:complex.htpy.type}
  With the notation in Theorem~\ref{thm:skeletons.complex}, there is a canonical
  homotopy equivalence $X^\an\to\HK(X,\sY_\Sigma)$.
\end{cor}

\begin{rem} \label{helm-katz:main result} Suppose now that $X\subset T$ is a
  sch\"on subvariety and that $K_0$ is a \emph{local} field (i.e.,\ $\td K_0$ is
  finite).  Assume that $X$ arises as the base change of a subvariety
  $X_{\bar K_0}$ of $\Spec(\bar K_0[M])$.
The main theorem of the paper of
Helm--Katz~\cite[Theorem~6.1]{helm_katz12:monodrom_filtration}, which relates
the weight-zero \'etale cohomology of $X_{\bar K_0}$ with the singular cohomology of the
parameterizing complex of $X$, is now seen in the light of Theorem~\ref{thm:skeletons.complex} to be a consequence%
\footnote{At least over local fields; Helm--Katz work over slightly more general
  $K_0$.} 
of a general result of
Berkovich~\cite{berkovic00:tate_conjecture}, relating the weight-zero \'etale
cohomology of $X_{\bar K_0}$ with the singular cohomology of $X^\an$.  This answers a
question posed in the introduction of~\cite{helm_katz12:monodrom_filtration}.
\end{rem}

Before beginning the proof of Theorem~\ref{thm:skeletons.complex}, we need the
following Lemma, which roughly says that $S(\sX,H)\to\Trop(X)$ is an
``unramified covering'' of integral $\Q$-affine piecewise linear sets.

\begin{lem}\label{lem:skel.to.trop}
  With the notation in Theorem~\ref{thm:skeletons.complex}, let $P\in\Sigma_1$,
  let $S$ be a connected component of $\sX\cap O(P)$, and let
  $\Delta_S\subset S(\sX,H)$ be the corresponding cell.  Then $\trop$ maps
  $\Delta_S$ bijectively onto $P$, and this map is unimodular in the sense
  of~\cite[\S2]{gubler_rabinoff_werner16:skeleton_tropical}.
\end{lem}

\begin{proof}
  First suppose that $X=T$ and $\sX = \sY_\Sigma$.  In this case, it is
  straightforward to check that the skeleton $S(\sX,H)$ is canonically
  identified with $\Sigma_1$ (considered as a polyhedral subdivision of
  $N_\R = \Trop(T)$), and that the retraction to the skeleton
  $\tau\colon X^\an\to S(\sX,H)$ is identified with $\trop\colon T^\an\to N_\R$.
  In particular,
  by~\cite[Proposition~4.10]{gubler_rabinoff_werner16:skeleton_tropical}
  or~\cite[Proposition~8.8]{gubler13:guide_tropical},
  all points of $\trop\inv(\relint P)$ reduce to the stratum $S$ under the
  reduction map $\red\colon X^\an\to\sX_s$, and all points of $X^\an$
  reducing to $S$ are contained in $\trop\inv(\relint P)$.

  No longer assuming $\sX = \sY_\Sigma$, but restricting the previous sentence
  to $\sX\subset\sY_\Sigma$, we see that a point $x\in X^\an$ reduces to
  $\sX\cap O(P)$ if and only if $\trop(x)\in\relint P$.
  Applying~\cite[Proposition~4.10]{gubler_rabinoff_werner16:skeleton_tropical}
  to $(\sX,H)$, this shows that $\trop\inv(\relint P)\cap S(\sX,H)$ is the
  disjoint union of the interiors $\relint\Delta_{S'}$ for $S'$ a component of
  $\sX\cap O(P)$.  In particular, $\trop$ maps $\relint\Delta_S$ (but not the
  boundary of $\Delta_S$) into $\relint P$.  Taking closures, we see that
  $\trop$ maps $\Delta_S$ into $P$.  The map $\trop\colon\Delta_S\to P$ is
  integral $\Q$-affine
  by~\cite[Proposition~8.2]{gubler_rabinoff_werner16:skeleton_tropical}.  This
  is enough to ensure $\trop$ is injective on $\Delta_S$.  The dimensions of $S$
  and $P$ are complementary by Theorem~\ref{thm:trop.fan.facts}(3)
  and~\eqref{eq:complementary.dims}, so $\dim(\Delta_S) = \dim(P)$.  Hence
  $\trop(\relint\Delta_S)$ is open in $\relint P$, and it is also closed since
  the boundary of $\Delta_S$ does not map into $\relint P$.  It follows that
  $\trop$ maps $\relint\Delta_S$ (resp.\ $\Delta_S$) bijectively onto
  $\relint P$ (resp.\ $P$).

  Now we treat unimodularity.  Since $\Trop(X)$ is pure dimensional, the same is
  true of $S(\sX,H)$.  Hence it suffices
  to consider $P$ of maximal dimension.  Let $S_1,\ldots,S_r$ be the connected
  components of $\sX\cap O(P)$.  Choose $\omega\in(\relint P)\cap N_\Gamma$.  By
  the skeletal Sturmfels--Tevelev multiplicity
  formula~\cite[Theorem~8.4]{gubler_rabinoff_werner16:skeleton_tropical}, we
  have
  \begin{equation}\label{eq:st.formula}
    m_{\Trop}(\omega) = \sum_{i=1}^r [N_P:N_{\Delta_{S_i}}] \geq r,
  \end{equation}
  where the terms
  $[N_P:N_{\Delta_{S_i}}]$ are the lattice indices introduced
  in~\cite[\S2]{gubler_rabinoff_werner16:skeleton_tropical}.
  By~\eqref{eq:orbit.inn}, we have
  $\inn_\omega(X)\cong\sT_P\times\Djunion_{i=1}^r S_i$, so it is clear that
  $m_{\Trop}(\omega) = r$.  It follows that $[N_P:N_{\Delta_{S_i}}]=1$ for all
  $i $, which is to say that $\trop$ is unimodular on $\Delta_{S_i}$.
\end{proof}

\begin{proof}[Proof of Theorem~\ref{thm:skeletons.complex}]
  We have already argued that the cells and facet relations of $S(\sX,H)$ are
  in natural bijection with those of $\HK(X,\sY_\Sigma)$.  The integral
  $\Q$-affine structures of the cells of $\HK(X,\sY_\Sigma)$ are inherited
  from those of $\Sigma_1$, and Lemma~\ref{lem:skel.to.trop} proves that the
  same is true for the cells of $S(\sX,H)$.  Moreover, the cell of $S(\sX,H)$
  corresponding to a cell $(P,C)$ of $\HK(X,\sY_\Sigma)$ maps to $P$ under
  $\trop$ by Lemma~\ref{lem:skel.to.trop}, so we have proved~(2).

  To prove~(1), we will adapt the argument
  of~\cite[Proposition~10.8]{gubler_rabinoff_werner16:skeleton_tropical}.  Let
  $K'$ be an algebraically closed, complete valued field extension of $K$ whose
  value group is all of $\R$, let $X'\coloneq X \otimes_K K'$ and let
  $\pi\colon X'^\an\to X^\an$ be the structural morphism.  By
  Lemma~\ref{lem:Strop.extend.scalars}, we have $\pi(\STrop(X')) = \STrop(X)$.
  Let $(\sX',H')$ be the strictly semistable pair over $K'^\circ$ in the sense
  of~\cite[Definition~3.1]{gubler_rabinoff_werner16:skeleton_tropical} given by
  base change $(\sX,H)$ to $K'^\circ$.  It is easy to see from the
  construction~\cite[\S4]{gubler_rabinoff_werner16:skeleton_tropical} that $\pi$
  induces an isomorphism $S(\sX',H') \cong S(\sX,H)$ identifies faces and
  strata.  Therefore to prove (1), it is enough to show that
  $\STrop(X')=S(\sX',H')$.

  Let $\omega\in\Trop(X)$.  By hypothesis
  $\inn_\omega(X')=\inn_\omega(X)\tensor_{\td K}\td K'$ is smooth, so its
  irreducible components are connected components.  As $\inn_\omega(X')$ is the
  special fiber of the tropical formal model $\fX_\omega'$, it follows 
  {from~\cite[Proposition~3.18]{baker_payne_rabinoff16:tropical_curves}}
  that the canonical model of
  $X_\omega'$ coincides with $\fX_\omega'$, so the canonical reduction
  $\td X_\omega'$ is equal to $\inn_\omega(X')$.  Let $C\subset\inn_\omega(X')$ be
  a connected component.  Then $X_C' \coloneq \red\inv(C)$ is a connected
  component of $X_\omega'$ by anti-continuity of $\red$, and it has a unique
  Shilov boundary point $\xi_C$ by Proposition~\ref{prop:can.red.surj}.  We
  claim that $X_C'\cap S(\sX',H') = \{\xi_C\}$; this suffices to prove~(1), as
  clearly $X_C'\cap\STrop(X') = \{\xi_C\}$, and $X'^\an$ is covered
  (set-theoretically) by the affinoid domains $X_C'$.

  We introduce a partial ordering $\leq$ on $X'^\an$ by declaring that
  $x\leq y$ if $|f(x)|\leq|f(y)|$ for all $f\in K'[M]$,
  where $M$ is the character lattice of $T'=T \otimes_K K'$.  The tropicalization
  $\trop\colon X'^\an\to N_\R$ factors through the retraction to the skeleton
  $\tau\colon (X'\setminus H')^\an\to S(\sX',H')$
  by~\cite[Proposition~8.2]{gubler_rabinoff_werner16:skeleton_tropical}, so
  $\tau(\xi_C)\in X_\omega'$.
  By~\cite[Theorem~5.2(ii)]{berkovic99:locally_contractible_I} (as applied to a
  suitable affinoid neighborhood of $\xi_C$; see the proof
  of~\cite[Proposition~10.8]{gubler_rabinoff_werner16:skeleton_tropical}), we
  have $\xi_C\leq\tau(\xi_C)$.  As $\xi_C$ is by definition maximal with respect
  to $\leq$, this implies $\xi_C = \tau(\xi_C)$, so $\xi_C\in S(\sX',H')$.

  Let $P\in\Sigma_1$ be the polyhedron containing $\omega$ in its relative
  interior, and let $C_1,\ldots,C_r$ be the connected components of
  $\inn_\omega(X')$.  We have shown the inclusion
  $\{\xi_{C_1},\ldots,\xi_{C_r}\}\subset S(\sX',H')\cap\trop\inv(\omega)$.
  By~\eqref{eq:orbit.inn}, $C_1,\ldots,C_r$ correspond to the open strata
  $S_1,\ldots,S_r$ of $\sX'$ lying on $\sX'\cap O(P)$, and by
  Lemma~\ref{lem:skel.to.trop}, each cell $\Delta_{S_i}\subset S(\sX',H')$ maps
  bijectively onto $P$ under $\trop$, with no other such cells mapping into
  $\relint(P)$.  It follows that $S(\sX',H')\cap\trop\inv(\omega)$ contains
  exactly $r$ points, so
  $\{\xi_{C_1},\ldots,\xi_{C_r}\}= S(\sX',H')\cap\trop\inv(\omega)$.  This
  completes the proof.
\end{proof}

\begin{rem}
  With the notation in Theorem~\ref{thm:skeletons.complex}, suppose in addition
  that all initial degenerations $\inn_\omega(X)$ are irreducible, or
  equivalently, that $m_{\Trop}(\omega)=1$ for all $\omega\in\Trop(X)$.  Then
  Theorem~\ref{thm:skeletons.complex} and its proof imply that
  $\trop\colon \STrop(X) = S(\sX,H)\to\Trop(X)$ is an isomorphism of integral
  $\Gamma$-affine piecewise linear sets, and that the canonical section
  $\Trop(X)\to\STrop(X)$ of~\S\ref{sec:section.trop} is the inverse isomorphism.
  Compare~\cite[Proposition~10.8]{gubler_rabinoff_werner16:skeleton_tropical}.
\end{rem}

\vspace{.2in}

\appendix

\section{Summary of notations}
~

\subsection{Analytic spaces and formal schemes}
~

{\small
\begin{longtable}[H]{rll}
  \hspace{1em}
  $X_{K'}$ & The extension of scalars of an object $X/K$ to $K'/K$.
  & \ref{notn:ext.scalar}, p.\pageref{notn:ext.scalar} \\
  $K$ & A non-Archimedean field.
  & \ref{notn:nonarch.field}, p.\pageref{notn:nonarch.field} \\
  $\val$ & $\colon K\to\R\cup\{\infty\}$, a complete
  non-Archimedean valuation.
  & \ref{notn:valuation}, p.\pageref{notn:valuation} \\
  $|\scdot|$ & $= \exp(-\val(\scdot))$, an associated absolute value.
  & \ref{notn:abs.val}, p.\pageref{notn:abs.val} \\
  $K^\circ$ & The valuation ring in $K$.
  & \ref{notn:ints}, p.\pageref{notn:ints}\\
  $K^\ccirc$ & The maximal ideal in $K^\circ$.
  & \ref{notn:ideal}, p.\pageref{notn:ideal} \\
  $\td K$ & $= K^\circ/K^\ccirc$, the residue field of $K$.
  & \ref{notn:res.field}, p.\pageref{notn:res.field} \\
  $\Gamma$ & $=\Gamma_K = \val(K^\times)\subset\R$, the value group of $K$. 
  & \ref{notn:value.gp}, p.\pageref{notn:value.gp} \\
  $\sqrt\Gamma$ & The saturation of $\Gamma$ in $\R$. 
  & \ref{notn:saturation}, p.\pageref{notn:saturation} \\
  $K\angles{r\inv x}$ & The generalized Tate algebra; also for $n$ variables. 
  & \ref{notn:gen.tate.alg}, p.\pageref{notn:gen.tate.alg} \\
  $\sH(x)$ & The completed residue field at a point $x$ of a $K$-analytic
  space. 
  & \ref{notn:completed.res.field}, p.\pageref{notn:completed.res.field} \\
  $\sM(\cA)$ & The Berkovich spectrum of a $K$-affinoid algebra $\cA$. 
  & \ref{notn:spectrum}, p.\pageref{notn:spectrum} \\
  $X^\an$ & The analytification of a locally finite-type $K$-scheme $X$. 
  & \ref{notn:analytification}, p.\pageref{notn:analytification} \\
  $\fX_s$ & $=\fX\tensor_{K^\circ}\td K$, the special fiber of a formal
  $K^\circ$-scheme $\fX$. 
  & \ref{notn:special.fiber}, p.\pageref{notn:special.fiber} \\
  $\fX_\eta$ & The analytic generic fiber of an admissible formal
  $K^\circ$-scheme $\fX$. 
  & \ref{notn:generic.fiber}, p.\pageref{notn:generic.fiber} \\
  $\cA^\circ$ & Ring of power-bounded elements in a {strictly} $K$-affinoid
  algebra $\cA$. 
  & \ref{notn:power.bounded}, p.\pageref{notn:power.bounded} \\
  $\cA^\ccirc$ & The ideal of topologically nilpotent elements in
  $\cA^\circ$. 
  & \ref{notn:top.nilpotent}, p.\pageref{notn:top.nilpotent} \\
  $\td\cA$ & $=\cA^\circ/\cA^\ccirc$, a $\td K$-algebra of finite type. 
  & \ref{notn:reduction.alg}, p.\pageref{notn:reduction.alg} \\
  $|\scdot|_{\sup}$ & The supremum (or spectral) semi-norm on a
  $K$-affinoid algebra $\cA$. 
  & \ref{notn:sup.norm}, p.\pageref{notn:sup.norm} \\
  $\fX^\can$ & $=\Spf(\cA^\circ)$, the canonical model of $X = \sM(\cA)$. 
  & \ref{notn:canonical.model}, p.\pageref{notn:canonical.model} \\
  $\td X$ & $= \Spec(\td\cA)$, the canonical reduction of $X = \sM(\cA)$. 
  & \ref{notn:canonical.reduction}, p.\pageref{notn:canonical.reduction} \\
  $B(X)$ & The Shilov boundary of a $K$-affinoid space $X$. 
  & \ref{notn:shilov.boundary}, p.\pageref{notn:shilov.boundary} \\
  $\red$ & $\colon X\to\td X$, the reduction map to the canonical reduction. 
  & \ref{eq:red.can}, p.\pageref{eq:red.can} \\
  $\red$ & $\colon\fX_\eta\to\fX_s$, reduction map of an admissible formal
  $K^\circ$-scheme $\fX$.
  & \ref{eq:red.model}, p.\pageref{eq:red.model} 
\end{longtable}}
~

\subsection{Toric varieties and tropicalizations}
~

{\small
\begin{longtable}[H]{rll}
  \hspace{1em}
  $M$ & $\cong\Z^n$, a finitely generated free abelian group. 
  & \ref{notn:char.group}, p.\pageref{notn:char.group} \\
  $N$ & $= \Hom(M,\Z)$, its dual. 
  & \ref{notn:char.group}, p.\pageref{notn:char.group} \\
  $M_G$ & $= M\tensor_\Z G\subset M_\R$ for $G\subset\R$ an additive subgroup. 
  & \ref{notn:bigger.char.group}, p.\pageref{notn:bigger.char.group} \\
  $N_G$ & Likewise. 
  & \ref{notn:bigger.char.group}, p.\pageref{notn:bigger.char.group} \\
  $\angles{\scdot,\scdot}$ & $\colon M_G\times N_G\to\R$, the evaluation
  pairing. 
  & \ref{notn:eval.pairing}, p.\pageref{notn:eval.pairing} \\
  $R[S]$ & The monoid ring over a ring $R$ of a monoid $S$. 
  & \ref{notn:monoid.ring}, p.\pageref{notn:monoid.ring} \\
  $\chi^u$ & $\in R[S]$, the character corresponding to $u\in S$. 
  & \ref{notn:character}, p.\pageref{notn:character} \\
  $T$ & $= \Spec(K[M])\cong\bG_{m,K}^n$, a split $K$-torus. 
  & \ref{notn:the.torus}, p.\pageref{notn:the.torus} \\
  $S_\sigma$ & $= \sigma^\vee\cap M$, the monoid associated to a cone
  $\sigma\subset N_\R$. 
  & \ref{notn:affine.toric}, p.\pageref{notn:affine.toric} \\
  $Y_\sigma$ & $=\Spec(K[S_\sigma])$, the affine toric variety from a rational
  cone $\sigma\subset N_\R$. 
  & \ref{notn:affine.toric}, p.\pageref{notn:affine.toric} \\
  $Y_\Delta$ & The toric variety associated to a rational pointed fan $\Delta$
  in $N_\R$. 
  & \ref{notn:general.toric}, p.\pageref{notn:general.toric} \\
  $M_G(\sigma)$ & $= (\sigma^\perp\cap M)\tensor_\Z G$ for $\sigma\subset N_\R$
  a rational cone and $G\subset\R$. 
  & \ref{notn:MGsigma}, p.\pageref{notn:MGsigma} \\
  $N_G(\sigma)$ & $= (N/\angles{\sigma}\cap N)\tensor_\Z G$. 
  & \ref{notn:MGsigma}, p.\pageref{notn:MGsigma} \\
  $\pi_\sigma$ & $\colon N_\R\to N_\R(\sigma)$, the projection. 
  & \ref{notn:pisigma}, p.\pageref{notn:pisigma} \\
  $O(\sigma)$ & $= \Spec(K[M(\sigma)])$, the torus orbit in $Y_\Delta$ coming
  from $\sigma\in\Delta$. 
  & \ref{notn:torus.orbit}, p.\pageref{notn:torus.orbit} \\
  $\bar\R$ & $= \R\cup\{\infty\}$, an additive monoid. 
  & \ref{notn:barR}, p.\pageref{notn:barR} \\
  $\bar N{}_G^\sigma$ & $= \Djunion_{\tau\prec\sigma} N_G(\tau)$. 
  & \ref{notn:barNRs}, p.\pageref{notn:barNRs} \\
  $\bar N{}_G^\Delta$ & $= \Djunion_{\sigma\in\Delta} N_G(\sigma)$, the $G$-points in a
  partial compactification of $N_\R$. 
  & \ref{notn:barNRD}, p.\pageref{notn:barNRD} \\
  $\trop$ & $\colon Y_\Delta^\an\to\bar N{}_\R^\Delta$ or $Y_\sigma^\an\to\bar
  N{}_\R^\sigma$, the tropicalization map. 
  & \ref{notn:trop.map}, p.\pageref{notn:trop.map} \\
  $s$ & $\colon\bar N{}_\R^\Delta\to Y_\Delta^\an$ or $\bar N{}_\R^\sigma\to Y_\sigma^\an$, the section of tropicalization. 
  & \ref{notn:trop.section}, p.\pageref{notn:trop.section} \\
  $|\scdot|_\omega$ & $= s(\omega)$ for $\omega\in\bar N{}_\R^\sigma$. 
  & \ref{notn:norm.omega}, p.\pageref{notn:norm.omega} \\
  $S(Y_\sigma^\an)$ & $= s(\bar N{}_\R^\sigma)$, the skeleton of
  $Y_\sigma^\an$. 
  & \ref{notn:skeleton.toric}, p.\pageref{notn:skeleton.toric} \\
  $S(Y_\Delta^\an)$ & $= s(\bar N{}_\R^\Delta)$, the skeleton of
  $Y_\Delta^\an$. 
  & \ref{notn:skeleton.toric2}, p.\pageref{notn:skeleton.toric2} \\
  $\Trop(X)$ & $\subset \bar N{}_\R^\Delta$ (resp.\ $\bar N{}_\R^\sigma$), the tropicalization of $X\subset Y_\Delta$ (resp.\ $Y_\sigma$). 
  & \ref{notn:TropX}, p.\pageref{notn:TropX} \\
  $U_\omega$ & $= \trop\inv(\omega)\subset Y_\Delta^\an$ for
  $\omega\in\bar N{}_\R^\Delta$. 
  & \ref{notn:Uomega}, p.\pageref{notn:Uomega} \\
  $\fU_\omega$ & $=\Spf(K\angles{U_\omega}^\circ)$, the canonical model of
  $U_\omega$ for $\omega\in\bar N{}_\Gamma^\Delta$. 
  & \ref{notn:fUomega}, p.\pageref{notn:fUomega} \\
  $X_\omega$ & $= U_\omega\cap X^\an$ for $X\subset Y_\Delta$ a closed
  subscheme. 
  & \ref{notn:Xomega}, p.\pageref{notn:Xomega} \\
  $\fX_\omega$ & The tropical formal model of $X_\omega$, a closed formal
  subscheme of $\fU_\omega$. 
  & \ref{notn:fXomega}, p.\pageref{notn:fXomega} \\
  $\inn_\omega(X)$ & $= (\fX_\omega)_s$, the initial degeneration of $X$ at
  $\omega$. 
  & \ref{notn:initial.degen}, p.\pageref{notn:initial.degen} \\
  $m_Z$ & The multiplicity of an irreducible component $Z$ of
  $\inn_\omega(X)$. 
  & \ref{notn:comp.mult}, p.\pageref{notn:comp.mult} \\
  $m_{\Trop}(\omega)$ & The tropical multiplicity of $X$ at $\omega$. 
  & \ref{notn:mTrop}, p.\pageref{notn:mTrop} \\
  $\LC_\omega(\Pi)$ & The local cone at $\omega$ of a polyhedral complex $\Pi$
  in $N_\R$. 
  & \ref{notn:local.cone}, p.\pageref{notn:local.cone} \\
  $\rho(P)$ & The recession cone of a polyhedron $P$. 
  & \eqref{notn:recession.cone}, p.\pageref{notn:recession.cone} \\
  $s_X$ & $\colon\multone\to\STrop(X)$, the section of tropicalization.
  & \ref{def:section}, p.\pageref{def:section}
\end{longtable}
}

\newpage
\bibliographystyle{amsalpha}
\bibliography{papers}

\vspace{.2in}

\end{document}